\documentclass[11pt]{article}
\usepackage[margin=2.5cm]{geometry}
\usepackage{pdflscape}
\usepackage{amsfonts}
\usepackage{amsmath}
\usepackage{amsthm}
\usepackage{amssymb}
\usepackage{amsmath,amscd}
\usepackage{graphicx}
\usepackage[rightcaption]{sidecap}
\usepackage{enumerate}
\usepackage{subcaption}
\graphicspath{ {./Figure/} }
\usepackage[pdffitwindow=false,
colorlinks=true,linkcolor=black,urlcolor=black,citecolor=black]{hyperref}
\usepackage{authblk}

\newcommand*{\email}[1]{%
    \normalsize\href{mailto:#1}{#1}\par
    }

\newtheorem{theorem}{Theorem}[section]
\newtheorem{lemma}[theorem]{Lemma}
\theoremstyle{definition}
\newtheorem{definition}[theorem]{Definition}

\newtheorem{proposition}[theorem]{Proposition}
\newtheorem{corollary}[theorem]{Corollary}
\newtheorem{remark}[theorem]{Remark}
\numberwithin{equation}{section}

\newcommand{\spin}{\mathrm{Spin}}

\newcommand{\trace}{\mathrm{tr}}

\newcommand{\ricci}{\mathrm{Ric}}

\begin{document}
\title {Positive Einstein metrics with $\mathbb{S}^{4m+3}$ as principal orbit}
\author[]{Hanci Chi}
\affil[]{Department of Foundational Mathematics, Xi’an Jiaotong-Liverpool University\\
\email{hanci.chi@xjtlu.edu.cn}}

\date{\today}

\maketitle
\abstract
We prove that there exists at least one positive Einstein metric on $\mathbb{HP}^{m+1}\sharp \overline{\mathbb{HP}}^{m+1}$ for $m\geq 2$. Based on the existence of the first Einstein metric, we give a criterion to check the existence of a second Einstein metric on $\mathbb{HP}^{m+1}\sharp \overline{\mathbb{HP}}^{m+1}$. We also investigate the existence of cohomogeneity one positive Einstein metrics on $\mathbb{S}^{4m+4}$ and prove the existence of a non-standard Einstein metric on $\mathbb{S}^8$.

\section{Introduction}
A Riemannian manifold $(M,g)$ is \emph{Einstein} if its Ricci curvature is a constant multiple of $g$:
$$
\ricci(g)=\Lambda g.
$$
The metric $g$ is then called an \emph{Einstein metric} and $\Lambda$ is the \emph{Einstein constant}. Depending on the sign of $\Lambda$, we call $g$ a positive Einstein ($\Lambda>0$) metric, a negative Einstein ($\Lambda<0$) metric or a Ricci-flat ($\Lambda=0$) metric. A positive Einstein manifold is compact by Myers' Theorem \cite{myers_riemannian_1941}.

In this article, we investigate the existence of positive Einstein metrics of \emph{cohomogeneity one}. A Riemannian manifold $(M, g)$ is of cohomogeneity one if a Lie Group $\mathsf{G}$ acts isometrically on $M$ such that the principal orbit $\mathsf{G}/\mathsf{K}$ is of codimension one. The first example of an inhomogeneous positive Einstein metric was constructed in \cite{page_compact_1978}. The metric is defined on $\mathbb{CP}^{2}\sharp \overline{\mathbb{CP}}^{2}$ and is of cohomogeneity one. The result was later generalized in \cite{berard-bergery_sur_1982}, \cite{koiso_non-homogeneous_1986}, \cite{sakane_examples_1986}, \cite{Page_1987}, \cite{koiso_non-homogeneous_1988} and \cite{wang_einstein_1998}. A common feature shared by positive Einstein metrics constructed in this series of works is the principal orbits being principal $U(1)$-bundles over either a Fano manifold or a product of Fano manifolds. From this perspective, one can view the Einstein metric on $\mathbb{HP}^{2}\sharp \overline{\mathbb{HP}}^{2}$ in \cite{bohm_inhomogeneous_1998} as another type of generalization to the Page's metric, whose principal orbit is a principal $Sp(1)$-bundle over $\mathbb{HP}^1$. 

A natural question arises whether there exists a positive Einstein metric of cohomogeneity one on $\mathbb{HP}^{m+1}\sharp \overline{\mathbb{HP}}^{m+1}$ with $m\geq 2$, where the principal orbit is the total space of the quaternionic Hopf fibration formed by the following group triple:
\begin{equation}
\label{eqn:_group triple}
(\mathsf{K},\mathsf{H},\mathsf{G})=(Sp(m)\Delta Sp(1), Sp(m)Sp(1)Sp(1), Sp(m+1)Sp(1)).
\end{equation}
The condition of being $\mathsf{G}$-invariant reduces the Einstein equations to an ODE system defined on the 1-dimensional orbit space. The solution takes the form of 
$g=dt^2+g(t),
$
where $g(t)$ is a $\mathsf{G}$-invariant metric on $\mathbb{S}^{4m+3}$ for each $t$. One looks for a $g(t)$ that is defined on a closed interval $[0, T]$ with an initial condition and a terminal condition. If $g(t)$ collapses to the quaternionic-K\"ahler metric on the singular orbit $\mathbb{HP}^m$ at $t=0$ and $t=T$, then $g$ defines a positive Einstein metric on the connected sum $\mathbb{HP}^{m+1}\sharp \overline{\mathbb{HP}}^{m+1}$, or equivalently, an $\mathbb{S}^4$-bundle over $\mathbb{HP}^m$. It has been conjectured that such an Einstein metric exists on $\mathbb{HP}^{m+1}\sharp \overline{\mathbb{HP}}^{m+1}$ for all $m\geq 2$, as indicated by numerical evidence provided in \cite{page_einstein_1986}, \cite{bohm_inhomogeneous_1998} and \cite{dancer_cohomogeneity_2013}.

Some well-known Einstein metrics are realized as integral curves to the cohomogeneity one Einstein equation. For example, the standard sphere metric, the sine cone over Jensen's sphere, and the quaternionic-K\"ahler metric on $\mathbb{HP}^{m+1}$ are represented by integral curves to the cohomogeneity one system. Furthermore, the cone solution is an attractor to the system. It was realized in \cite{bohm_inhomogeneous_1998} that the winding of integral curves around the cone solution plays an important role in the existence problem described above. To investigate the winding, one studies a quantity (denoted as $\sharp C_w(\bar{h})$ in \cite{bohm_inhomogeneous_1998}) that is assigned to each local solution that does not globally define a complete Einstein metric on $\mathbb{HP}^{m+1}\sharp \overline{\mathbb{HP}}^{m+1}$. From the point of view of geometry, the quantity records the number of times that the principal orbit becomes isoparametric while its mean curvature remains positive. In general, an estimate for $\sharp C_w(\bar{h})$ can be obtained from the linearization along the cone solution. For $m=1$, the estimate is good enough to prove the global existence. This is not the case, however, if $m\geq 2$. For higher dimensional cases, it is from the global analysis of the system that we obtain a further estimate for $\sharp C_w(\bar{h})$ and we prove the following existence theorem.
\begin{theorem}
\label{thm: first Einstein metric}
On each $\mathbb{HP}^{m+1}\sharp \overline{\mathbb{HP}}^{m+1}$ with $m\geq 2$, there exists at least one positive Einstein metric with $\mathsf{G}/\mathsf{K}=\mathbb{S}^{4m+3}$ as its principal orbit.
\end{theorem}

Numerical studies in \cite{bohm_inhomogeneous_1998} and \cite{dancer_cohomogeneity_2013} indicate that there exists another Einstein metric on $\mathbb{HP}^{m+1}\sharp \overline{\mathbb{HP}}^{m+1}$ with $m\geq 2$. Based on Theorem \ref{thm: first Einstein metric}, an estimate for $\sharp C_w(\bar{h})$ in a limiting subsystem (essentially obtained from the linearization along the cone solution) helps us propose a criterion to check the existence of the second Einstein metric. Let $n$ be the dimension of $\mathsf{G}/\mathsf{K}$. Such a criterion only depends on $n$ (or $m$). 
\begin{theorem}
\label{thm: second Einstein metric}
Let $\theta_\Psi$ be the solution to the following initial value problem:
\begin{equation}
\label{eqn:_determine existence}
\frac{d\theta}{d\eta}=\frac{n-1}{2n}\tanh\left(\frac{\eta}{n}\right)\sin(2\theta)+\frac{2}{n}\sqrt{\frac{(2m+1)(2m+2)(2m+3)}{(2m+3)^2+2m}},\quad \theta(0)=0.
\end{equation}
Let $\Omega=\lim\limits_{\eta\to \infty}\theta_\Psi$. For $m\geq 2$, there exist at least two positive Einstein metrics on $\mathbb{HP}^{m+1}\sharp \overline{\mathbb{HP}}^{m+1}$ if $\Omega < \frac{3\pi}{4}$.
\end{theorem}
The upper bound for $\Omega$ in Theorem \ref{thm: second Einstein metric} is not sharp. Although it is difficult to solve the initial value problem \eqref{eqn:_determine existence} explicitly, one can use the Runge--Kutta 4-th order algorithm to approximate $\Omega$. Since the RHS of \eqref{eqn:_determine existence} does not vanish at $\eta=0$, the initial Runge--Kutta step is well-defined. Our numerical study shows that $\Omega< \frac{3\pi}{4}$ for integers $m\in[2,100]$.

We also look into the case where $\mathsf{G}/\mathsf{K}$ completely collapses at two ends of a compact manifold. In that case, the cohomogeneity one space is $\mathbb{S}^{4m+4}$. No new Einstein metric is found on $\mathbb{S}^{4m+4}$ for $m\geq 2$. For $m=1$, however, we obtained a non-standard positive Einstein metric on $\mathbb{S}^8$. Such a metric is inhomogeneous by the classification in \cite{ziller_homogeneous_1982}.
\begin{theorem}
\label{thm: 8-sphere}
There exists a non-standard $Sp(2)Sp(1)$-invariant positive Einstein metric $\hat{g}_{\mathbb{S}^8}$ on $\mathbb{S}^{8}$.
\end{theorem}

It is worth mentioning that all new solutions found are symmetric. Metrics that are represented by these solutions all have a totally geodesic principal orbit.

This article is structured as follows. In Section \ref{sec: Cohomogeneity one system}, we present the dynamical system for positive Einstein metrics of cohomogeneity one with $\mathsf{G}/\mathsf{K}$ as the principal orbit. Then we apply a coordinate change that makes the cohomogeneity one Ricci-flat system serve as a limiting subsystem. Initial conditions and terminal conditions are transformed into critical points of the new system. The new system admits $\mathbb{Z}_2$-symmetry. By a sign change, one can transform initial conditions into terminal conditions. Hence the problem of finding globally defined positive Einstein metrics boils down to finding heteroclines that join two different critical points.

In Section \ref{sec: Linearization at critical points}, we compute linearizations of the critical points mentioned above and obtain two 1-parameter families of locally defined positive Einstein metrics. One family is defined on a tubular neighborhood around $\mathbb{HP}^m$, represented by a 1-parameter family of integral curves $\gamma_{s_1}$. The other family is defined on a neighborhood of a point in $\mathbb{S}^{4m+4}$, represented by another 1-parameter family of integral curves $\zeta_{s_2}$.

In Section \ref{sec: Existence of the first Einstein metric}, we make a little modification on the quantity $\sharp C_w(\bar{h})$ in \cite{bohm_inhomogeneous_1998} and it is assigned to both $\gamma_{s_1}$ and $\zeta_{s_2}$ (hence denoted as $\sharp C(\gamma_{s_1})$ and $\sharp C({\zeta_{s_2}})$). We construct a compact set to obtain an estimate for $\sharp C(\gamma_{s_1})$ of some local solutions. Then we apply Lemma 4.4 in \cite{bohm_inhomogeneous_1998} and prove Theorem \ref{thm: first Einstein metric}.  

In Section \ref{sec: Limiting Winding Angle}, we apply another coordinate change that allows us to obtain more information on $\sharp C(\gamma_{s_1})$ and $\sharp C(\zeta_{s_2})$, which is encoded in the initial value problem \eqref{eqn:_determine existence} in Theorem \ref{thm: second Einstein metric}. We also prove Theorem \ref{thm: 8-sphere}. 

Visual summaries of Theorem \ref{thm: first Einstein metric}-\ref{thm: 8-sphere} are presented at the end of this article.

\textbf{Acknowledgments} The research is funded by NSFC (No. 12071489), the Foundation for Young Scholars of Jiangsu Province, China (BK-20220282), and XJTLU Research Development Funding (RDF-21-02-083). The author is grateful to McKenzie Wang for his constant support and encouragement. The author would like to thank Christoph B\"ohm for his helpful suggestions and remarks on this project. The author also thanks Cheng Yang and Wei Yuan for many inspiring discussions.

\section{Cohomogeneity one system}
\label{sec: Cohomogeneity one system}
Consider the group triple $(\mathsf{K},\mathsf{H},\mathsf{G})$ in \eqref{eqn:_group triple}. The isotropy representation $\mathfrak{g}/\mathfrak{k}$ consists of two inequivalent irreducible summands $\mathfrak{p}_1=\mathfrak{h}/\mathfrak{k}$ and $\mathfrak{p}_2=\mathfrak{g}/\mathfrak{h}$. Let the standard sphere metric $g_{\mathbb{S}^{4m+3}}$ on $\mathsf{G}/\mathsf{K}=\mathbb{S}^{4m+3}$ be the background metric. As any $\mathsf{G}$-invariant metric on $\mathsf{G}/\mathsf{K}$ is determined by its restriction to one tangent space $\mathfrak{g}/\mathfrak{k}$, the metric has the form of
$$
f_1^2 \left.g_{\mathbb{S}^{4m+3}}\right|_{\mathfrak{p}_1}+f_2^2 \left.g_{\mathbb{S}^{4m+3}}\right|_{\mathfrak{p}_2}.
$$
Let $f_1$ and $f_2$ be functions that are defined on the 1-dimensional orbit space. We consider Einstein equations for the cohomogeneity one metric
$$
g:=dt^2+f_1^2 \left.g_{\mathbb{S}^{4m+3}}\right|_{\mathfrak{p}_1}+f_2^2 \left.g_{\mathbb{S}^{4m+3}}\right|_{\mathfrak{p}_2}.
$$
By \cite{eschenburg_initial_2000}, the metric $g$ is an Einstein metric on $(t_*-\epsilon,t_*+\epsilon)\times \mathsf{G}/\mathsf{K}$ if $(f_1,f_2)$ is a solution to
\begin{equation}
\label{eqn:_Original Einstein}
\begin{split}
&\frac{\ddot{f_1}}{f_1}-\left(\frac{\dot{f_1}}{f_1}\right)^2=-\left(3 \frac{\dot{f_1}}{f_1}+4m\frac{\dot{f_2}}{f_2}\right)\frac{\dot{f_1}}{f_1}+2\frac{1}{f_1^2}+4m\frac{f_1^2}{f_2^4}-\Lambda,\\
&\frac{\ddot{f_2}}{f_2}-\left(\frac{\dot{f_2}}{f_2}\right)^2=-\left(3 \frac{\dot{f_1}}{f_1}+4m\frac{\dot{f_2}}{f_2}\right)\frac{\dot{f_2}}{f_2}+(4m+8)\frac{1}{f_2^2}-6\frac{f_1^2}{f_2^4}-\Lambda,
\end{split}
\end{equation}
with a conservation law
\begin{equation}
\label{eqn:_original conservation}
3\left(\frac{\dot{f_1}}{f_1}\right)^2+4m\left(\frac{\dot{f_2}}{f_2}\right)^2-\left(d_1 \frac{\dot{f_1}}{f_1}+d_2\frac{\dot{f_2}}{f_2}\right)^2+6\frac{1}{f_1^2}+4m(4m+8)\frac{1}{f_2^2}-12m\frac{f_1^2}{f_2^4}-(n-1)\Lambda=0.
\end{equation}
To fix homothety, we set $\Lambda=n$ in this article. We leave $\Lambda$ in the equations for readers to trace the Einstein constant.

\begin{remark}
If we replace the principal orbit $\mathsf{G}/\mathsf{K}$ by $\mathbb{S}^{4m+3}=[Sp(m+1)U(1)]/[Sp(m)\Delta U(1)]$, then the isotropy representation $\mathfrak{g}/\mathfrak{k}$ consists of three inequivalent irreducible summands. The principal orbit can collapse either as $\mathbb{HP}^m$ or $\mathbb{CP}^{2m+1}$, depending on the choice of intermediate group. For such a principal orbit, the dynamical system of cohomogeneity one Einstein metrics involves three functions and has \eqref{eqn:_Original Einstein} as its subsystem. A numerical solution in \cite{Hiragane_2003} indicates the existence of a positive Einstein metric where $\mathsf{G}/\mathsf{K}$ collapse to $\mathbb{HP}^m$ on one end and $\mathbb{CP}^{2m+1}$ on the other end.
\end{remark}

We consider \eqref{eqn:_Original Einstein} and \eqref{eqn:_original conservation} with the following two initial conditions. By \cite{eschenburg_initial_2000}, for the metric $g$ to extend smoothly to the singular orbit $\mathbb{HP}^{m}$, we have 
\begin{equation}
\label{eqn:_initial condition 1}
\lim_{t\to 0}\left(f_1,f_2,\dot{f}_1,\dot{f}_2\right)=\left(0,f,1,0\right)
\end{equation}
for some $f>0$.
On the other hand, for $g$ to extend smoothly to a point where $\mathsf{G}/\mathsf{K}$ fully collapses, one considers
\begin{equation}
\label{eqn:_initial condition 2}
\lim_{t\to 0}\left(f_1,f_2,\dot{f}_1,\dot{f}_2\right)=\left(0,0,1,1\right).
\end{equation}
By Myers' theorem, any solution obtained from \eqref{eqn:_Original Einstein} that represents an Einstein metric on $\mathbb{HP}^{m+1}\sharp\overline{\mathbb{HP}}^{m+1}$ must be defined on $[0,T]$ for some finite $T>0$. Specifically, one looks for solutions with the initial condition \eqref{eqn:_initial condition 1} and the terminal condition 
\begin{equation}
\label{eqn:_terminal condition 1}
\lim_{t\to T}\left(f_1,f_2,\dot{f}_1,\dot{f}_2\right)=\left(0,\bar{f},-1,0\right)
\end{equation}
for some $\bar{f}>0$.
Similarly, to construct an Einstein metric on $\mathbb{S}^{4m+4}$, one looks for solutions with the initial condition \eqref{eqn:_initial condition 2} and the terminal condition
\begin{equation}
\label{eqn:_terminal condition 2}
\lim_{t\to T}\left(f_1,f_2,\dot{f}_1,\dot{f}_2\right)=\left(0,0,-1,-1\right).
\end{equation}

\begin{remark}
In \cite{koiso_hypersurfaces_1981}, one takes a non-collapsed principal orbit $\mathsf{G}/\mathsf{K}$ as the initial data. Specifically, consider 
$$\left(f_1,f_2,\dot{f}_1,\dot{f}_2\right)=\left(\bar{f}_1,\bar{f}_2,\bar{h}_1,\bar{h}_2\right)$$ 
for some positive $\bar{f}_i$'s. To construct a positive Einstein metric, one looks for a solution that extends backward and forward smoothly to either $\mathbb{HP}^m$ or a point on $\mathbb{S}^{4m+4}$ in finite time.
\end{remark}

Inspired by a personal communication with Wei Yuan, we introduce a coordinate change that transforms \eqref{eqn:_Original Einstein} to a polynomial ODE system. Let $L$ be the shape operator of principal orbit. Define 
$$
X_1:=\frac{\frac{\dot{f_1}}{f_1}}{\sqrt{(\trace{L})^2+n\Lambda}},\quad X_2:=\frac{\frac{\dot{f_2}}{f_2}}{\sqrt{(\trace{L})^2+n\Lambda}},
\quad Y:=\frac{\frac{1}{f_1}}{\sqrt{(\trace{L})^2+n\Lambda}}, \quad Z:=\frac{\frac{f_1}{f_2^2}}{\sqrt{(\trace{L})^2+n\Lambda}}.
$$
Also, define 
\begin{equation*}
\begin{split}
&H:=3X_1+4mX_2, \quad G:=3X_1^2+4mX_2^2,\\
&R_1:=2Y^2+4mZ^2,\quad R_2:=(4m+8)YZ-6Z^2.
\end{split}
\end{equation*}
Consider $d\eta=\sqrt{\trace(L)^2+n\Lambda}dt$. Let $'$ denote taking the derivative with respect to $\eta$. Then \eqref{eqn:_Original Einstein} becomes
\begin{equation}
\label{eqn:_new Positive Einstein system}
\begin{bmatrix}
X_1\\
X_2\\
Y\\
Z
\end{bmatrix}'= V(X_1,X_2,Y,Z)=
\begin{bmatrix}
X_1H\left(G+\frac{1}{n} (1-H^2)-1\right)+R_1-\frac{1}{n}(1-H^2)\\
X_2H\left(G+\frac{1}{n} (1-H^2)-1\right)+R_2-\frac{1}{n} (1-H^2)\\
Y\left(H\left(G+\frac{1}{n} (1-H^2)\right)-X_1\right)\\
Z\left(H\left(G+\frac{1}{n} (1-H^2)\right)+X_1-2X_2\right)
\end{bmatrix}.
\end{equation}
The conservation law \eqref{eqn:_original conservation} becomes
\begin{equation}
\label{eqn:_new positive conservation law 1}
\mathcal{C}_{\Lambda \geq 0}: G+\frac{1}{n}(1-H^2)+6Y^2+4m(4m+8)YZ-12mZ^2=1.
\end{equation}
Or equivalently,
\begin{equation}
\label{eqn:_new positive conservation law2}
\mathcal{C}_{\Lambda \geq 0}: \frac{12m}{n}(X_1-X_2)^2+6Y^2+4m(4m+8)YZ-12mZ^2=1-\frac{1}{n}.
\end{equation}
We can retrieve the original system by
\begin{equation}
\label{eqn:_how to get back to the original system 2}
t= \int_{\eta_*}^\eta \sqrt{\frac{1-H^2}{n\Lambda}} d\tilde{\eta}, \quad f_1=\frac{1}{Y}\sqrt{\frac{1-H^2}{n\Lambda}},\quad  f_2=\frac{1}{\sqrt{YZ}}\sqrt{\frac{1-H^2}{n\Lambda}}.
\end{equation}

It is clear that $H^2\leq  1$ by the definition of $H$ and $X_i$'s. However, such a piece of information can be obtained from the new system alone without \eqref{eqn:_Original Einstein} and \eqref{eqn:_original conservation}.
Note that
\begin{equation}
\label{eqn:_derivative of H}
\begin{split}
H'&=\langle\nabla H, V\rangle\\
&=H^2\left(G+\frac{1}{n}(1-H^2)-1\right)+6Y^2+4m(4m+8)YZ-12mZ^2-(1-H^2)\\
&=H^2\left(G+\frac{1}{n}(1-H^2)-1\right)+1-G-\frac{1}{n}(1-H^2)-(1-H^2)\quad \text{by \eqref{eqn:_new positive conservation law 1}}\\
&=(H^2-1)\left(G+\frac{1}{n}(1-H^2)\right)=(H^2-1)\left(\frac{1}{n}+\frac{12m}{n}(X_1-X_2)^2\right).
\end{split}
\end{equation}
Therefore, the following algebraic surface in $\mathbb{R}^4$ with boundary 
$$
\mathcal{E}:=\mathcal{C}_{\Lambda \geq 0}\cap \{Y,Z\geq 0\} \cap \{H^2\leq 1\}
$$
is invariant. Moreover,
$\mathcal{E}\cap \{H=\pm 1\}$ are two invariant sets of lower dimension. The $\mathbb{Z}_2$-symmetry on the sign of $(X_1,X_2)$ gives a one to one correspondence between integral curves on $\mathcal{E}\cap \{H= 1\}$ and those on $\mathcal{E}\cap \{H= -1\}$.

\begin{remark}
\label{rem: H^2=1 invariant}
The restricted system of \eqref{eqn:_new Positive Einstein system} on $\mathcal{E}\cap \{H= 1\}$ is in fact \eqref{eqn:_Original Einstein} with $\Lambda=0$ under the coordinate change $d\eta=(\trace{L}) dt$. The dynamical system is essentially the same as the one that appears in \cite{wink_cohomogeneity_2017}. An integral curve on the subsystem is known for representing a complete Ricci-flat metric defined on the non-compact manifold $\mathbb{HP}^{m+1}\backslash\{*\}$ \cite{bohm_inhomogeneous_1998}. The Ricci-flat metric on $\mathbb{HP}^{m+1}\backslash\{*\}$ is the limit cone for locally defined positive Einstein metrics on the tubular neighborhood around $\mathbb{HP}^m$.
\end{remark} 

\begin{remark}
\label{rem: fix the guage}
If an integral curve to \eqref{eqn:_new Positive Einstein system} enters $\mathcal{E}\cap \{H<1\}$ and is defined on $\mathbb{R}$, then from \eqref{eqn:_derivative of H} it must cross $\mathcal{E}\cap\{H=0\}$ transversally. The crossing point corresponds to the \emph{turning point} in \cite{bohm_inhomogeneous_1998}. For any integral curve to \eqref{eqn:_new Positive Einstein system} that has a turning point, we choose the $\eta_*$ in \eqref{eqn:_how to get back to the original system 2} so that $t_*:= \int_{\eta_*}^0 \sqrt{\frac{1-H^2}{n\Lambda}} d\tilde{\eta}$ is the value at which $\trace{L}$ vanishes. By our choice of $\eta_*$, the integral curve crosses $\mathcal{E}\cap\{H=0\}$ at $\eta=0$. There are cohomogeneity one Einstein systems with additional geometric structure, e.g. the one considered in \cite{foscolo_new_2017}, where every trajectory has a turning point. 
\end{remark}

\begin{remark}
\label{rem: w-intersect}
From \eqref{eqn:_new positive conservation law2}, the inequality $6Y^2+4m(4m+8)YZ-12mZ^2\leq 1-\frac{1}{n}$ is always valid. Therefore, the set
$\mathcal{E}\cap \{Z-\rho Y\leq 0\}$ is compact for any fixed $\rho\in \left[0,\frac{m(4m+8)+\sqrt{m^2(4m+8)^2+18m}}{6m}\right)$. If the maximal interval of existence of an integral curve to \eqref{eqn:_new Positive Einstein system} is $(-\infty,\bar{\eta})$ for some $\bar{\eta}\in \mathbb{R}$, it must escape $\mathcal{E}\cap \{Z-\rho Y\leq 0\}$. The crossing point corresponds to the \emph{W-intersection point} in \cite{bohm_inhomogeneous_1998}. In Proposition \ref{prop: noncompact invariant set} and Definition \ref{def: W-intersection}, we introduce an invariant set $\mathcal{W}$ and a modified definition for the $W$-intersection point, which fixes $\rho=1$ in the original definition in \cite{bohm_inhomogeneous_1998}.
\end{remark}

\section{Linearization at critical points}
\label{sec: Linearization at critical points}
The local existence of positive Einstein metrics around the singular orbit $\mathbb{HP}^{m}$ is well-established in \cite{bohm_inhomogeneous_1998}. We interpret the result using the new coordinate. For $m\geq 2$, the vector field $V$ has in total $10$ critical points ($12$ critical points for $m=1$) on $\mathcal{E}$. As indicated by their superscripts, these critical points lie on either $\mathcal{E}\cap \{H=1\}$ or $\mathcal{E}\cap \{H=-1\}$. 
\begin{itemize}
\item
$p_0^\pm=\left(\pm\frac{1}{3},0,\frac{1}{3},0\right)$

These points represent the initial condition \eqref{eqn:_initial condition 1} and the terminal condition \eqref{eqn:_terminal condition 1}. Integral curves that emanate from $p_0^+$ and enter $\mathcal{E}\cap\{H<1\}$ represent positive Einstein metrics defined on a tubular neighborhood around $\mathbb{HP}^m$. A complete Einstein metric on $\mathbb{HP}^{m+1}\sharp\overline{\mathbb{HP}}^{m+1}$ is represented by a heterocline that joins $p_0^\pm$.
\item
$p_1^\pm=\left(\pm\frac{1}{n},\pm\frac{1}{n},\frac{1}{n},\frac{1}{n}\right)$

These points represent the initial condition \eqref{eqn:_initial condition 2} and the terminal condition \eqref{eqn:_terminal condition 2}. Integral curves that emanate from $p_0^+$ and enter $\mathcal{E}\cap\{H<1\}$ represent positive Einstein metrics defined on a tubular neighborhood around a point on $\mathbb{S}^{4m+4}$. The standard sphere metric is represented by a straight line that joins $p_1^\pm$. It is also worth mentioning that the quaternionic-K\"ahler metric on $\mathbb{HP}^{m+1}$ (resp. $\overline{\mathbb{HP}}^{m+1}$) is represented by an integral curve that joins $p_0^+$ and $p_1^-$ (resp. $p_0^-$ and $p_1^+$). 
\item
$p_2^\pm=\left(\pm\frac{1}{n},\pm \frac{1}{n},(2m+3)z_0,z_0\right),\quad z_0=\frac{1}{n}\sqrt{\frac{2m+1}{2m+(2m+3)^2}}$

These points represent the initial condition and the terminal condition where the principal orbit collapses as Jensen's sphere\cite{jensen_einstein_1973}. There is only one integral curve that emanates from $p_2^+$ and it represents the singular sine metric cone with its base as Jensen's sphere \cite{jensen_einstein_1973}. It is also worth mentioning that $p_2^+$ is a sink for the Ricci-flat subsystem of \eqref{eqn:_new Positive Einstein system} restricted on $\mathcal{E}\cap \{H=1\}$, representing the asymptotically conical limit.
\item
$q_1^\pm=\left(\pm\frac{3+ 2\sqrt{12m^2+6m}}{3n},\pm\frac{4m-2\sqrt{12m^2+6m}}{4mn},0,0\right),$\\$q_2^\pm=\left(\pm\frac{3-2 \sqrt{12m^2+6m}}{3n},\pm\frac{4m+2\sqrt{12m^2+6m}}{4mn},0,0\right)$

These critical points are in general ``bad'' points for our study. Integral curves that converge to $q_1^-$ or $q_2^-$ represent metrics with blown up $\dot{f}_1$ and $\dot{f}_2$. Straightforward computations show that for $m\geq 2$, critical points $q_1^+$ and $q_2^+$ are sources for \eqref{eqn:_new Positive Einstein system} on \eqref{eqn:_new positive conservation law 1}. By the $\mathbb{Z}_2$ symmetry, critical points $q_1^-$ and $q_2^-$ are sinks for $m\geq 2$.
\item
$q_3^\pm=\left(\mp\frac{1}{3},\pm\frac{2}{4m},0,\frac{\sqrt{3-2m}}{6m}\right)$

These critical points are also ``bad'' points as $q_1^\pm$ and $q_2^\pm$. They only exist for $m=1$. Integral curves that converge to $q_3^-$ represent metrics with blown up $\dot{f_1}$. For $m=1$, critical points $q_1^+$ and $q_3^+$ are sources and $q_1^-$ and $q_3^-$ are sinks; $q_2^+$ and $q_2^-$ are saddles.
\end{itemize}

\begin{proposition}
\label{prop: all the points}
The list above exhausts all critical points on $\mathcal{E}$.
\end{proposition}
\begin{proof}
By \eqref{eqn:_derivative of H}, it is clear that critical points on $\mathcal{E}$ must lie on $\{H^2=1\}$. The list is complete by considering the vanishing of the $Y$-entry and $Z$-entry.
\end{proof}

For any $m$, linearizations at $q_i^\pm$'s show that the phase space $\mathcal{E}$ is ``filled'' with integral curves that emanate from $q_i^+$ or those that converge to $q_i^-$. Hence most integral curves that emanate from $p_0^+$ or $p_1^+$ are anticipated to converge to one of these $q_i^-$'s. In the following, we give a detailed analysis of the linearizations at $p_0^+$ and $p_1^+$ and integral curves that emanate from these critical points. 

The linearization at $p_0^+$ is 
$$
\begin{bmatrix}
-\frac{8m-6}{3n}&-\frac{32m^2-24m}{9n}&\frac{4}{3}&0\\
\frac{6}{n}&\frac{16m-6}{3n}&0&\frac{4m+8}{3}\\
\frac{8m}{3n}&\frac{16m^2-12m}{9n}&0&0\\
0&0&0&\frac{2}{3}
\end{bmatrix}.
$$
Eigenvalues and eigenvectors are 
$$\lambda_1=\lambda_2=\frac{2}{3},\quad \lambda_3=-\frac{2}{3},\quad \lambda_4=\frac{8m}{3n};$$
$$
v_1=\begin{bmatrix}
-(8m^2+18m+18)\\
-9\\
-(8m^2+18m)\\
9
\end{bmatrix},
\quad
v_2=\begin{bmatrix}
-4m(m+2)\\
3(m+2)\\
-2m(m+2)\\
3
\end{bmatrix},
\quad
v_3=\begin{bmatrix}
-4m\\
3\\
2m\\
0
\end{bmatrix},
\quad
v_4=\begin{bmatrix}
-2(4m-3)\\
18\\
4m-3\\
0
\end{bmatrix}.
$$
The first three eigenvectors are tangent to $\mathcal{E}$. Consider linearized solutions in the form of 
\begin{equation}
\label{eqn:_linearized solution near p_0}
p_0^++ e^{\frac{2}{3}\eta}v_1+s_1e^{\frac{2}{3}\eta}v_2
\end{equation}
for some $s_1\in \mathbb{R}$.
By Hartman–Grobman theorem, there is a 1 to 1 correspondence between each choice of $s_1\in \mathbb{R}$ and an actual solution curve that emanates from $p_0^+$ and leaves $\mathcal{E}\cap \{H=1\}$ initially. Hence we use $\gamma_{s_1}$ to denote an actual solution that approaches the linearized solution \eqref{eqn:_linearized solution near p_0} near $p_0^+$. Moreover, by the unstable version of Theorem 4.5 in Chapter 13 of \cite{coddington_theory_1955}, there is some $\epsilon>0$ that 
$$\gamma_{s_1}= p_0^+ + e^{\frac{2}{3}\eta}v_1+s_1e^{\frac{2}{3}\eta}v_2+O\left(e^{\left(\frac{2}{3}+\epsilon\right)\eta}\right).
$$
From the linearization at $p_0^+$ and \eqref{eqn:_how to get back to the original system 2}, the parameter $s_1$ is related to the initial condition $f$ in \eqref{eqn:_initial condition 1} as follows.
\begin{equation}
\label{eqn:_initial condition related to s_1}
f=\lim\limits_{\eta\to -\infty} \left(\frac{1}{\sqrt{YZ}}\frac{1-H^2}{n\Lambda}\right)(\gamma_{s_1})=\sqrt{\frac{6m+18}{n}\frac{1}{3+s_1}}.
\end{equation}

We set $s_1>-3$ so that $f$ is positive. From another perspective, in order to have $\gamma_{s_1}$ be in $\mathcal{E}$, we only consider $\gamma_{s_1}$ with $s_1> -3$ so that $Z$ is positive initially along the integral curve. Note that $v_2$ is tangent to $\mathcal{E}\cap \{H=1\}$. Therefore, it makes sense to let $\gamma_\infty$ denote the integral curve that lies in $\mathcal{E}\cap \{H=1\}$ such that
\begin{equation}
\label{eqn:_linearized RF near p_0}
\gamma_\infty\sim p_0^+ +0\cdot e^{\frac{2}{3}\eta}v_2+ 1\cdot e^{\frac{2}{3}\eta}v_2
\end{equation}
near $p_0^+$. The integral curve $\gamma_\infty$ represents the Ricci-flat metric on $\mathbb{HP}^{m+1}\backslash\{*\}$ constructed in \cite{bohm_inhomogeneous_1998}. For $m=1$, the metric is the $\spin(7)$ metric in \cite{bryant_construction_1989} and \cite{gibbons_einstein_1990}. Furthermore, as shown in Proposition 6.3 in \cite{chi_einstein_2020}, the integral curve $\gamma_\infty$ lies on the following $1$-dimensional invariant set
\begin{equation}
\label{eqn:_characterization of gamma_infty}
\mathcal{B}_{\spin(7)}:=\mathcal{E}\cap \{Y-2Z-X_1=0\}\cap \{3Z-X_2=0\}
\end{equation}
and it joins $p_0^+$ and $p_2^+$.
\begin{remark}
\label{rem: gamma_infty for m=1}
The defining equations in \eqref{eqn:_characterization of gamma_infty} are equivalent to the cohomogeneity one $\spin(7)$ condition on $\mathbb{HP}^2\backslash\{*\}$. Specifically, we have the following dynamical system.
\begin{equation}
\begin{split}
&\frac{\dot{f_1}}{f_1}=\frac{1}{f_1}-2\frac{f_1}{f_2^2},\\
&\frac{\dot{f_2}}{f_2}=3\frac{f_1}{f_2^2}.
\end{split}
\end{equation}
Similar to the initial value problem in Remark \ref{rem: standard QK}, the initial condition can be obtained from the coordinate of $p_0^+$ and the limit $\lim\limits_{\eta\to-\infty}\left(\frac{X_2}{\sqrt{YZ}}\right)(\gamma_\infty(\eta))$. Since a Ricci-flat metric is homothety invariant, the extra freedom allows us to set $f_2(0)$ as any positive number. Solving the initial value problem with
$$(f_1(0), \dot{f_1}(0),\dot{f_2}(0),\dot{f_2}(0)))=(0,1,\beta,0),\quad \beta>0$$
yields the homothetic family of $\spin(7)$ metrics in \cite{bryant_construction_1989} and \cite{gibbons_einstein_1990}.
\end{remark}

The linearization at $p_1^+$ is 
$$
\begin{bmatrix}
-\frac{16m^2+8m-6}{n^2}&\frac{16m(m+1)}{n^2}&\frac{4}{n}&\frac{8m}{n}\\
\frac{12m+12}{n^2}&-\frac{4m+6}{n^2}&\frac{4m+8}{n}&\frac{4m-4}{n}\\
-\frac{4m}{n^2}&\frac{4m}{n^2}&0&0\\
\frac{4m+6}{n^2}&-\frac{4m+6}{n^2}&0&0
\end{bmatrix}.
$$
Eigenvalues and eigenvectors are 
$$\mu_1=\mu_2=\frac{2}{n},\quad \mu_3=0,\quad \mu_4=-\frac{4(m+1)}{n};$$
$$
w_1=\begin{bmatrix}
-1\\
-1\\
0\\
0
\end{bmatrix},\quad 
w_2=\begin{bmatrix}
-4m\\
3\\
2m\\
-(2m+3)
\end{bmatrix},\quad
w_3=\begin{bmatrix}
-(n-1)\\
-(n-1)\\
1\\
1
\end{bmatrix},\quad
w_4=\begin{bmatrix}
-8m(m+1)\\
6(m+1)\\
-2m\\
2m+3
\end{bmatrix}.
$$
The first three eigenvectors are tangent to $\mathcal{E}$. Hence there exists a 1-parameter family of integral curves $\zeta_{s_2}$ that emanate from $p_1^+$ and 
\begin{equation}
\label{eqn:_linearized positive einstein near P1+}
\zeta_{s_2}= p_1^+ + e^{\frac{2}{n}\eta}w_1+s_2e^{\frac{2}{n}\eta}w_2+O\left(e^{\left(\frac{2}{3}+\epsilon\right)\eta}\right).
\end{equation}
The initial condition \eqref{eqn:_initial condition 2} has a degree of freedom in the second-order derivative. Specifically, the parameter $s_2$ is related to the limit $\lim\limits_{t\to 0}\frac{\ddot{f_2}}{f_2}\frac{f_1}{\ddot{f_1}}$. From \eqref{eqn:_Original Einstein}, \eqref{eqn:_how to get back to the original system 2} and the linearization at $p_1^+$, we have 
\begin{equation}
\label{eqn:_initial condition related to s_2}
\lim\limits_{t\to 0}\frac{\ddot{f_2}}{f_2}\frac{f_1}{\ddot{f_1}}=\lim\limits_{\eta\to -\infty}\frac{X_2^2-HX_2+R_2-\frac{1-H^2}{n}}{X_1^2-HX_1+R_1-\frac{1-H^2}{n}}=\frac{n}{4m(1+6ms_2)}-\frac{9}{12m}.
\end{equation}
Although it is clear that $\zeta_{s_2}$ is in $\mathcal{E}$ for any $s_2\in \mathbb{R}$, we mainly consider $s_2\geq 0$ in this article. We have the following proposition for $\zeta_{s_2}$ with $s_2<0$.

\begin{proposition}
\label{prop: noncompact invariant set}
Each $\zeta_{s_2}$ with $s_2<0$ either does not converge to any critical point in $\mathcal{E}$, or it converges to $q_2^-$ ($q_2^-$ or $q_3^-$ if $m=1$).
\end{proposition}
\begin{proof}
From the linearized solution, it is clear that each $\zeta_{s_2}$ with $s_2<0$ is initially in
\begin{equation}
\label{eqn:_the set W}
\mathcal{W}:=\mathcal{E}\cap \{Z-Y\geq 0\}\cap \{X_1-X_2\geq 0\}.
\end{equation}
As $\mathcal{W}$ includes all points $\left(\sqrt{\frac{n-1}{12m}}\cosh(\lambda),0,0,\sqrt{\frac{n-1}{12nm}}\sinh(\lambda)\right)$ with $\lambda\geq 0$, the set is non-compact. Furthermore, since 
\begin{equation}
\label{eqn:_derivative of Z-Y}
\begin{split}
\left\langle\nabla \left(\frac{Y}{Z}\right),V\right\rangle&=-2\frac{Y}{Z}(X_1-X_2)\leq 0
\end{split}
\end{equation}
and
\begin{equation}
\label{eqn:_derivative of X1-X2}
\begin{split}
\langle\nabla(X_1-X_2),V\rangle&=(X_1-X_2)H\left(G+\frac{1}{n}(1-H^2)-1\right)+2(Z-Y)((2m+3)Z-Y)
\end{split},
\end{equation}
the set $\mathcal{W}$ is invariant. Note that the second term in \eqref{eqn:_derivative of X1-X2} is non-negative in $\mathcal{W}\cap \{X_1-X_2=0\}$.

By \eqref{eqn:_derivative of Z-Y}, it is clear that the function $\frac{Y}{Z}$ monotonically decreases from $1$ along each $\zeta_{s_2}$ with $s_2<0$. If the function $\frac{Y}{Z}$ converges to some positive number and $X_1-X_2$ would converge to zero. From \eqref{eqn:_new positive conservation law2}, we know that both $Y$ and $Z$ converge to some positive numbers with $\frac{Y}{Z}<1$. Hence the RHS of \eqref{eqn:_derivative of X1-X2} is eventually positive as $X_1-X_2$ converges to zero, a contradiction. Hence the function $\frac{Y}{Z}$ converges to zero. Therefore, if a $\zeta_{s_2}$ with $s_2<0$ converges to a critical point in $\mathcal{E}$, it must be $q_2^-$ ($q_2^-$ or $q_3^-$ if $m=1$).
\end{proof}

\begin{remark}
For $s_2=0$, it is clear that $\zeta_0$ lies on the 1-dimensional invariant set 
\begin{equation}
\label{eqn:_standard sine cone}
\mathcal{E}\cap \{X_1=X_2\}\cap \left\{Y=Z=\frac{1}{n}\right\}
\end{equation}
and joins $p_1^\pm$. The integral curve represents the standard sphere metric $g_{\mathbb{S}^{4m+4}}$ on $\mathbb{S}^{4m+4}$. Specifically, defining equations in \eqref{eqn:_standard sine cone} give the initial value problem
\begin{equation}
\begin{split}
\frac{\dot{f_1}}{f_1}&=\frac{\dot{f_2}}{f_2},\quad f_1^2+(\dot{f_1})^2=1,\\
f_1(0)&=f_2(0)=0,\quad \dot{f_1}(0)=\dot{f_2}(0)=1,
\end{split}
\end{equation}
in the original coordinates. The solution is exactly the standard sphere metric
$$
g_{\mathbb{S}^{4m+4}}=dt^2+\sin^2(t)g_{\mathbb{S}^{4m+3}}.
$$
\end{remark}
We define $\zeta_\infty$ to be the integral curve that emanates from $p_1^+$ and lies in $\mathcal{C}_{\Lambda\geq 0}\cap \{H=1\}$. We have
\begin{equation}
\label{eqn:_linearized RF near p_1}
\zeta_\infty\sim p_1^+ +e^{\frac{2}{n}\eta}w_2.
\end{equation}
As studied in \cite{chi_einstein_2020}, the integral curve $\zeta_\infty$ is known to be defined on $\mathbb{R}$ and it joins $p_1^+$ and $p_2^+$ and it represents a complete non-trivial Ricci-flat metric defined on $\mathbb{R}^{4m+4}$.

As shown in the following proposition, there exists an integral curve that joins $p_0^+$ and $p_1^-$, and it represents the standard quaternionic-K\"ahler metric on $\mathbb{HP}^{m+1}$. By the $\mathbb{Z}_2$-symmetry of \eqref{eqn:_new Positive Einstein system} on the sign of $(X_1,X_2)$. We know that there also exists an integral curve that emanates from $p_1^+$ and tends to $p_0^-$, and it represents the standard quaternionic-K\"ahler metric on $\overline{\mathbb{HP}}^{m+1}$.
\begin{proposition}
\label{prop: qk metric as integral curve}
The integral curve $\gamma_{0}$ lies on the 1-dimensional invariant set 
$$
\mathcal{B}_{QK}:=\mathcal{E}\cap \left\{X_1-X_2+Z-Y=0\right\} \cap \{X_2+Z=0\}.
$$
The integral curve $\zeta_{\frac{1}{2m+6}}$ lies on the 1-dimensional invariant set
$$
\bar{\mathcal{B}}_{QK}:=\mathcal{E}\cap \{X_2-X_1+Z-Y=0\}\cap\{X_2-Z=0\}.
$$

\end{proposition}
\begin{proof}
As $X_1-X_2=Y-Z$ and $X_2=-Z$ on $\mathcal{B}_{QK}$, we can eliminate $X_1$ and $X_2$ in \eqref{eqn:_new positive conservation law2}. Hence
\begin{equation}
\label{eqn:_qk conservation}
\frac{12m}{n}(Y-Z)^2+6Y^2+4m(4m+8)YZ-12mZ^2=1-\frac{1}{n}
\end{equation}
holds on $\mathcal{B}_{QK}$. Therefore,
\begin{equation}
\begin{split}
&\left.\langle\nabla(X_2+Z),V\rangle\right|_{\mathcal{B}_{QK}}\\
&=(X_2+Z)H\left(G+\frac{1}{n}(1-H^2)\right)-\frac{1}{n}(1-H^2)-X_2H+(4m+8)YZ-6Z^2+Z(X_1-2X_2)\\
&=-\frac{1}{n}(1-H^2)-X_2H+(4m+8)YZ-6Z^2+Z(X_1-2X_2)\\
&=\frac{1}{n-1}\left(\frac{12m}{n}(Z-Y)^2+6Y^2+4m(4m+8)YZ-12mZ^2-\frac{n-1}{n}\right)\\
&\quad \text{Eliminate $X_1$ and $X_2$ by the definition of $\mathcal{B}_{QK}$}\\
&=0 \quad \text{by \eqref{eqn:_qk conservation}}.
\end{split}
\end{equation}
On the other hand, we have
\begin{equation}
\begin{split}
&\left.\langle\nabla(X_1-X_2+Z-Y),V\rangle\right|_{\mathcal{B}_{QK}}\\
&=(X_1-X_2+Z-Y)\left(H\left(G+\frac{1}{n}(1-H^2)-1\right)+4Z-2Y\right)+(n-1)(X_2+Z)(Z-Y)\\
&=0.
\end{split}
\end{equation}
Therefore, the set $\mathcal{B}_{QK}$ is indeed invariant. 

Since $X_1=Y-2Z$ and $X_2=-Z$ on $\mathcal{B}_{QK}$, one can realize $\mathcal{B}_{QK}$ as a hyperbola \eqref{eqn:_qk conservation}. Note that $p_0^+$ and $p_1^-$ are the only critical points in $\mathcal{B}_{QK}$ and they are in the same connected component in \eqref{eqn:_qk conservation}. Therefore, there is an integral curve that joins $p_0^+$ and $p_1^-$ and lies on $\mathcal{B}_{QK}$. Hence the integral curve must be some $\gamma_{s_1}$. Let $v(\eta)$ be the normalized velocity of the linearized solution that uniquely corresponds to $\gamma_0$. It is clear that $\lim\limits_{\eta\to-\infty}v(\eta)=\frac{v_1}{\|v_1\|}$ is tangent to $\mathcal{B}_{QK}$ at $p_0^+$. Hence we know that $\gamma_0$ lies on $\mathcal{B}_{QK}$. By the $\mathbb{Z}_2$-symmetry on the sign of $(X_1,X_2)$, we know that $\zeta_{\frac{1}{2m+6}}$ lies on the invariant set
$$
\bar{\mathcal{B}}_{QK}:=\mathcal{C}_{\Lambda\geq 0}\cap \{X_2-X_1+Z-Y=0\}\cap\{X_2-Z=0\}
$$
and joins $p_1^+$ and $p_0^-$.
\end{proof}

\begin{remark}
\label{rem: standard QK}
The defining equations for $\mathcal{B}_{QK}$ are equivalent to the following dynamical system
\begin{equation}
\begin{split}
&\frac{\dot{f_1}}{f_1}=-2\frac{f_1}{f_2^2}+\frac{1}{f_1},\\
&\frac{\dot{f_2}}{f_2}=-\frac{f_1}{f_2^2}.\\
\end{split}
\end{equation}
While $f_1(0)=0$ and $\dot{f_1}(0)=1$ can be obtained from the coordinate of $p_0^+$, the initial conditions $f_2(0)$ and $\dot{f_2}(0)$ are obtained from $v_1$. Specifically, from \eqref{eqn:_linearized solution near p_0} we have 
$$
f_2(0)=\lim\limits_{\eta\to-\infty}\left(\frac{\sqrt{1-H^2}}{n\sqrt{YZ}}\right)(\gamma_0(\eta))=\sqrt{\frac{4(m+3)}{n}},\quad \dot{f_2}(0)=\lim\limits_{\eta\to-\infty}\left(\frac{X_2}{\sqrt{YZ}}\right)(\gamma_0(\eta))=0.
$$
Solving the initial value problem, the standard quaternionic-K\"ahler metric on $\mathbb{HP}^{m+1}$ is 
$$
g=dt^2+\frac{m+3}{n}\sin^2\left(2\sqrt{\frac{n}{4(m+3)}}t\right) \left.g_{\mathbb{S}^{4m+3}}\right|_{\mathfrak{p}_1}+\frac{4(m+3)}{n}\cos^2\left(\sqrt{\frac{n}{4(m+3)}}t\right) \left.g_{\mathbb{S}^{4m+3}}\right|_{\mathfrak{p}_2}.
$$
\end{remark}


Lastly, we consider the linearization at $p_2^+$. We have 
$$
\begin{bmatrix}
-\frac{16m^2+8m-6}{n^2}&\frac{16m(m+1)}{n^2}&(8m+12)z_0&8mz_0\\
\frac{12m+12}{n^2}&-\frac{4m+6}{n^2}&4(m+2)z_0&4(2m+1)(m+3)z_0\\
-\frac{4m(2m+3)z_0}{n}&\frac{4m(2m+3)z_0}{n}&0&0\\
\frac{(4m+6)z_0}{n}&-\frac{(4m+6)z_0}{n}&0&0
\end{bmatrix}.
$$
Eigenvalues and eigenvectors are 
$$\delta_1=-\frac{2m+1-\sqrt{(2m+1)^2-8(2m+3)(m+1)n^2z_0^2}}{n},\quad \delta_2=-\frac{2m+1+\sqrt{(2m+1)^2-8(2m+3)(m+1)n^2z_0^2}}{n},
$$
$$\delta_3=\frac{2}{n},\quad \delta_4=0;$$
$$
u_1=\begin{bmatrix}
\frac{2mn}{2m+3}\delta_1\\
-\frac{3n}{2(2m+3)}\delta_1\\
-2mnz_0\\
nz_0
\end{bmatrix},\quad 
u_2=\begin{bmatrix}
\frac{2mn}{2m+3}\delta_2\\
-\frac{3n}{2(2m+3)}\delta_2\\
-2mnz_0\\
nz_0
\end{bmatrix},\quad
u_3=\begin{bmatrix}
-1\\
-1\\
0\\
0
\end{bmatrix},\quad
u_4=\begin{bmatrix}
-2n((2m+3)^2+2m)z_0\\
-2n((2m+3)^2+2m)z_0\\
2m+3\\
1
\end{bmatrix}.
$$
The first three eigenvectors are tangent to $\mathcal{E}$. Furthermore, the first two eigenvectors are tangent to $\mathcal{E}\cap \{H=1\}$ and $\delta_2<\delta_1<0$. For $m\geq 1$, the critical point $p_2^+$ is a stable node for the restricted system on $\mathcal{E}\cap \{H=1\}$. Let $\Phi$ be the only integral curve that emanates from $p_2^+$. It converges to $p_2^-$ and lies on the 1-dimensional invariant set 
\begin{equation}
\label{eqn:_jensen based cone}
\mathcal{E}\cap \{X_1=X_2\}\cap \left\{Y=(2m+3)Z=(2m+3)z_0\right\}.
\end{equation}
\begin{remark}
\label{rem: jensen cone}
The initial value problem from the defining equations \eqref{eqn:_jensen based cone} is similar to the one from \eqref{eqn:_standard sine cone}. Specifically, we have
\begin{equation}
\begin{split}
f_1^2&=\frac{1}{2m+3}f_2^2,\quad \frac{(2m+1)(2m+3)^2}{2m+(2m+3)^2}\left(f_1^2+(\dot{f_1})^2\right)=1,\\
f_1(0)&=f_2(0)=0,\quad \dot{f_1}(0)=\frac{1}{\sqrt{2m+3}}\dot{f_2}(0)=\frac{1}{2m+3}\sqrt{\frac{2m+(2m+3)^2}{2m+1}}.
\end{split}
\end{equation}
Hence $\Phi$ represents the sine cone over Jensen's sphere
\begin{equation*}
\begin{split}
g&=dt^2+\frac{2m+(2m+3)^2}{(2m+1)(2m+3)^2}\sin^2(t) \left.g_{\mathbb{S}^{4m+3}}\right|_{\mathfrak{p}_1}+\frac{2m+(2m+3)^2}{(2m+1)(2m+3)}\sin^2(t) \left.g_{\mathbb{S}^{4m+3}}\right|_{\mathfrak{p}_2}\\
&=dt^2+\sin^2(t)g_{Jensen}.
\end{split}
\end{equation*}
\end{remark}

\section{Existence of the first Einstein metric}
\label{sec: Existence of the first Einstein metric}
We prove the existence of a heterocline that joins $p_0^\pm$ in this section. The technique is to construct a compact set $\mathcal{S}$ such that a $\gamma_{s_1}$ that enters the set can only escape through points in $\mathcal{E} \cap\{H\geq 0\}\cap \{X_1-X_2=0\}$. Then we apply Lemma 4.4 in \cite{bohm_inhomogeneous_1998} to complete the proof.

We define the compact set $\mathcal{S}$ as follows.
Define polynomials
\begin{equation}
\begin{split}
A&:=YX_2-\frac{3}{m} Z\left(X_1+\frac{2m}{3} X_2\right),\\
B&:=\frac{1-H^2}{n}-\frac{2n^2(2m+3)(m-1)}{m(2m+1)(8m+3)} YZ,\\
P&:=X_1\left(R_2-\frac{1}{n}(1-H^2)\right)-X_2\left(R_1-\frac{1}{n}(1-H^2)\right)-2X_2\left(X_1+\frac{2m}{3}X_2\right)(X_1-X_2),\\
Q&:=-4X_2Y^2-(4m+8)(4mX_2+2X_1)YZ+(2X_1+(4m+2)X_2)\frac{1-H^2}{n}\\
&\quad +4X_2\left(X_1+\frac{2m}{3}X_2\right)\left(H+\frac{2m}{3}X_2\right).
\end{split}
\end{equation}
Define
$$
\mathcal{S}:=\mathcal{E}\cap \{X_1-X_2\geq 0\}\cap \{X_2\geq 0\}\cap \{A\geq 0\}\cap \{B\geq 0\}\cap \{P\geq 0\}.
$$
For illustration, we present $\mathcal{S}$ for $m=5$ in Figure \ref{fig: S}. The following proposition lists some basic properties of $\mathcal{S}$.
\begin{proposition}
\label{prop: basic prop of S1}
The set $\mathcal{S}$ has the following property:
\begin{enumerate}[(a)]
\item
For $m\geq 1$, the set $\mathcal{S}\cap\{X_2=0\}$ is a union of $\{p_0^+\}$ and a $1$-dimensional curve  $\Gamma:=\mathcal{S}\cap \{X_1=X_2=0\}$. For $m\geq 2$, the set $\Gamma$ is bounded;
\item
The variable $Y$ is positive in $\mathcal{S}$ for $m\geq 1$;
\item
For $m\geq 1$, the set $\mathcal{S}\cap\{Z=0\}$ is $\left\{p_0^+, \left(0,0,\sqrt{\frac{n-1}{6n}},0\right)\right\}$;
\item
The set $\mathcal{S}$ is compact for $m\geq 2$.
\end{enumerate}
\end{proposition}
\begin{proof}
Since $A\geq 0$ in $\mathcal{S}$, a point in $\mathcal{S}$ with vanishing $X_2$-coordinate must have $ZX_1=0$. If we further assume $Z=0$ at that point, then from $P\geq 0$ we have $X_1(1-H^2)=X_1(1-9X_1^2)\leq 0$. Hence we obtain $X_1=\frac{1}{3}$ and that point must be $p_0^+$. We obtain $\Gamma$ if $X_1=0$ is further assumed.

It is obvious that $\Gamma=\mathcal{E}\cap \{X_1=X_2=0\}$ for $m=1$ and the set is non-compact. We prove that $\Gamma$ is bounded for $m\geq 2$. From $B\geq 0$ and \eqref{eqn:_new positive conservation law2}, we have 
\begin{equation}
\begin{split}
&6Y^2+4m(4m+8)YZ-12mZ^2=\frac{n-1}{n}\geq (n-1)\frac{2n^2(2m+3)(m-1)}{m(2m+1)(8m+3)}YZ\\
&\Leftrightarrow 6Y^2+\left(4m(4m+8)-(n-1)\frac{2n^2(2m+3)(m-1)}{m(2m+1)(8m+3)}\right)YZ-12mZ^2\geq 0.
\end{split}
\end{equation}
For $m\geq 2$, the inequality above implies $mY-Z\geq 0$. Then from \eqref{eqn:_new positive conservation law2} we have 
$
\frac{n-1}{n}\geq 6Y^2+(4m+32)Z^2. 
$ The first claim is clear.

Suppose there is a point in $\mathcal{S}$ with vanishing $Y$-coordinate. From $A\geq 0$ we know that 
$
-\frac{3}{m}Z\left(X_1+\frac{2m}{3}X_2\right)=0
$
at that point. If $Z\neq 0$, then $X_1=X_2=Y=0$ at that point, which is impossible from \eqref{eqn:_new positive conservation law2}. If $X_1+\frac{2m}{3}X_2\neq 0$, then $Y=Z=0$ at that point. Then from \eqref{eqn:_new positive conservation law2}, we have $X_1-X_2\neq 0$. From $P\geq 0$ we know that 
$
-\frac{1}{n}(1-H^2)-2X_2\left(X_1+\frac{2m}{3}X_2\right)\geq 0
$
at that point. The point has to be $\left(\frac{1}{3},0,0,0\right)$, which does not lie on $\mathcal{E}$. The above discussion proves the second claim.

Since $P\geq 0$ on a point with vanishing $Z$-coordinate in $\mathcal{S}$, we have 
$$
(X_2-X_1)\frac{1}{n}(1-H^2)-2X_2Y^2-2X_2\left(X_1+\frac{2m}{3}X_2\right)(X_1-X_2)\geq 0.
$$
By the definition of $\mathcal{S}$, each term in the above inequality is non-positive. Since $Y>0$ from the second claim, the variable $X_2$ must vanish and the third claim is clear.

Finally, from $A\geq 0$ and $X_1-X_2\geq 0$ in $\mathcal{S}$, we know that $X_2(Y-\frac{2m+3}{m}Z)\geq 0$ in $\mathcal{S}$. If $X_2\neq 0$, then $mY\geq (2m+3)Z$ and the boundedness of all variables is obtained from \eqref{eqn:_new positive conservation law2}. If $X_2=0$, then the boundedness comes from the first claim. Hence $\mathcal{S}$ is a compact set.
\end{proof}
\begin{figure}[h!]
\begin{subfigure}{.33\textwidth}
  \centering
  \includegraphics[clip,width=1\linewidth]{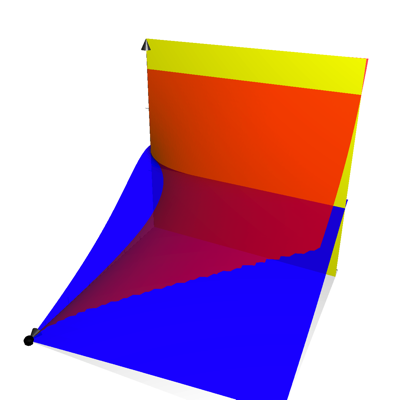}\caption{$m=5$}
\end{subfigure}
\begin{subfigure}{.33\textwidth}
  \centering
  \includegraphics[clip,width=1\linewidth]{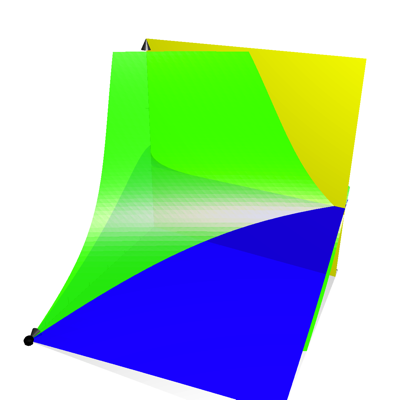}
  \caption*{}
\end{subfigure}
\begin{subfigure}{.33\textwidth}
  \centering
  \includegraphics[clip,width=1\linewidth]{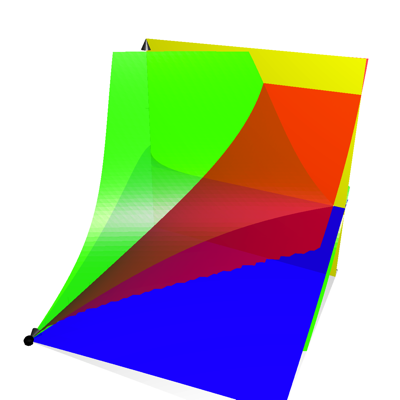}
\caption*{}
\end{subfigure}\\
\caption{Figures above illustrate the set $\mathcal{S}$ for $m=5$. The variable $Y$ is replaced by a function in $(X_1, X_2, Z)$ obtained from \eqref{eqn:_new positive conservation law2}. From readers' point of view, points in $\mathcal{S}$ are at the front of $X_1-X_2=0$ (yellow), below $A=0$ (red), behind $B=0$ (green) and above $P=0$ (blue).}
\label{fig: S}
\end{figure}

The case $m=1$ is very special. In the following proposition, we show that for $m=1$, the defining inequalities $A\geq 0$ and $P\geq 0$ must be equalities. The set $\mathcal{S}$ is closely related to the integral curves $\gamma_\infty$ in Remark \ref{rem: gamma_infty for m=1} and $\Phi$ in Remark \ref{rem: jensen cone}.

\begin{proposition}
\label{prop: S1 for m=1}
For $m=1$, the set $\mathcal{S}$ is the union 
$$ \{p_0^+,p_2^+\}\cup \gamma_{\infty} \cup\Gamma\cup \left(\Phi\cap\{X_1,X_2> 0\}\right).$$
\end{proposition}
\begin{proof}
It is apparent from Figure \ref{fig: 3dm=1} that the proposition holds. Consider $\mathcal{S}$ with $m=1$. If $X_2=0$ is imposed, we obtain either the point $p_0^+$ or $\Gamma$ from Proposition \ref{prop: basic prop of S1} (a). Hence we assume $X_2>0$ in the following. From $A\geq 0$ it is clear that $X_2(Y-5Z)\geq YX_2-3Z\left(X_1+\frac{2}{3}X_2\right)\geq 0$ and hence $Y-5Z\geq 0$.

One can easily verify that
\begin{equation}
\label{eqn:_AP for m=1}
\begin{split}
P&=(X_1-X_2)\left(6YZ-2X_2\left(X_1+\frac{2}{3}X_2\right)\right)-\frac{1}{7}(X_1-X_2)(1-H^2)-2(Y-Z)A
\end{split}.
\end{equation}
If $X_1=X_2=X>0$, the first two terms in \eqref{eqn:_AP for m=1} vanish while the last term is non-positive. Then we must have $A=0$, from which we deduce $Y-5Z= 0$. By \eqref{eqn:_jensen based cone} we obtain the line segment $\Phi\cap \{X_1,X_2> 0\}$. 

For $X_1\neq X_2$, we rewrite \eqref{eqn:_AP for m=1} as follows:
\begin{equation}
\label{eqn:_AP for m=1b}
\begin{split}
P=(X_1-X_2)\left(6YZ-2X_2\left(X_1+\frac{2}{3}X_2\right)-\frac{10}{147}(1-H^2)\right)-\frac{11}{147}(X_1-X_2)(1-H^2)-2(Y-Z)A
\end{split}.
\end{equation}
The last two terms in \eqref{eqn:_AP for m=1b} are non-positive. Hence from $P\geq 0$ we have
\begin{equation}
\label{eqn:_m=1 inequality}
\begin{split}
0&\leq 6YZ-2X_2\left(X_1+\frac{2}{3}X_2\right)-\frac{10}{147}(1-H^2)\\
&= 6YZ-2X_2\left(X_1+\frac{2}{3}X_2\right)-\frac{10}{147}\left(\frac{7}{6}\left(6Y^2+48YZ-12Z^2+\frac{12}{7}(X_1-X_2)^2\right)-H^2\right)\\
&\quad \text{by \eqref{eqn:_new positive conservation law2}}\\
&= \frac{1}{21}(5X_1+4X_2+5Y+2Z)(X_1-X_2-Y+5Z)\\
&\quad +\frac{1}{21}\left(5X_1+4X_2-5Y-2Z\right)(X_1-X_2+Y-5Z).
\end{split}
\end{equation}
Suppose $X_1-X_2-Y+5Z\leq 0$, then the first term in the last line of \eqref{eqn:_m=1 inequality} is non-positive. As the summation above is non-negative, we know that the second term in the last line of \eqref{eqn:_m=1 inequality} must be non-negative. As $X_1-X_2+Y-5Z\geq 0$, we must have $\frac{1}{5}\left(5X_1+4X_2-2Z\right)\geq Y$. From $A\geq 0$ we have 
$$
\frac{X_2}{5}\left(5X_1+4X_2-2Z\right)\geq 3Z\left(X_1+\frac{2}{3}X_2\right) \Leftrightarrow (5X_1+4X_2)(X_2-3Z)\geq 0 \Leftrightarrow X_2-3Z\geq 0.
$$
Then we claim that the first term in \eqref{eqn:_AP for m=1} is non-positive since
\begin{equation}
\begin{split}
6YZ-2X_2\left(X_1+\frac{2}{3}X_2\right)&\leq \frac{6}{5}\left(5X_1+4X_2-2Z\right)Z-2X_2\left(X_1+\frac{2}{3}X_2\right)\\
&\leq \frac{2}{5}\left(5X_1+4X_2-\frac{2}{3}X_2\right)X_2-2X_2\left(X_1+\frac{2}{3}X_2\right)\\
&=0.
\end{split}
\end{equation}
But $P\geq 0$. Hence assumptions $X_1\neq X_2$ and $X_1-X_2-Y+5Z\leq 0$ lead to the vanishing of $A$ and $P$.

Suppose $X_1-X_2-Y+5Z\geq 0$. Then $A\geq 0$ implies
$$
(X_1-X_2+5Z)X_2\geq 3Z\left(X_1+\frac{2}{3}X_2\right) \Leftrightarrow X_2\geq 3Z.
$$
Then we claim that the first term in \eqref{eqn:_AP for m=1} is also non-positive since
\begin{equation}
\begin{split}
6YZ-2X_2\left(X_1+\frac{2}{3}X_2\right)&\leq 6(X_1-X_2+5Z)Z-2X_2\left(X_1+\frac{2}{3}X_2\right)\\
&\leq 2\left(X_1-X_2+\frac{5}{3}X_2\right)X_2-2X_2\left(X_1+\frac{2}{3}X_2\right)\\
&=0.
\end{split}
\end{equation}
Hence the assumption $X_1\neq X_2$ and $X_1-X_2-Y+5Z\geq 0$ also leads to the vanishing of $A$ and $P$. 

Therefore, points in $\mathcal{S}$ with $X_1\neq X_2$ must have vanished $A$ and $P$, which leads to the following equalities.
$$
H=1,\quad 3Z=X_2,\quad 3Y=3X_1+2X_2.
$$
Note that the last two equations above are equivalent to the defining equations in \eqref{eqn:_characterization of gamma_infty} for $m=1$. We obtain $\gamma_\infty$ and critical points $p_0^+$ and $p_2^+$.
\end{proof}
\begin{remark}
\label{rem: integral curve chi}
For $m=1$, there is another heterocline $\chi$ that also lies on the invariant set $\mathcal{B}_{\spin(7)}$. The integral curve joins $q_3^+$ and $p_2^+$. Note that from the proof to Proposition \ref{prop: S1 for m=1}, it is clear that $X_1-X_2$ and $Y-5Z$ are positive along $\gamma_\infty$. These two polynomials are negative along $\chi$ and hence the integral curve is not in $\mathcal{S}$.
\end{remark}
\begin{figure}[h!] 
\begin{subfigure}{.24\textwidth}
  \centering 
  \includegraphics[clip,width=1\linewidth]{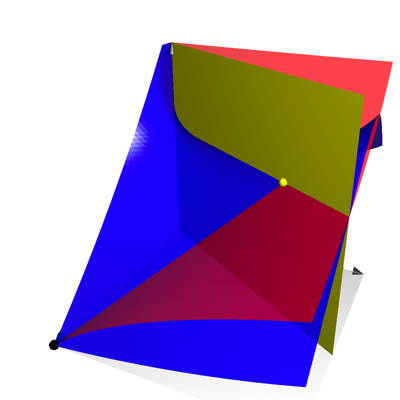}
\end{subfigure}
\begin{subfigure}{.24\textwidth}
  \centering
  \includegraphics[clip,width=1\linewidth]{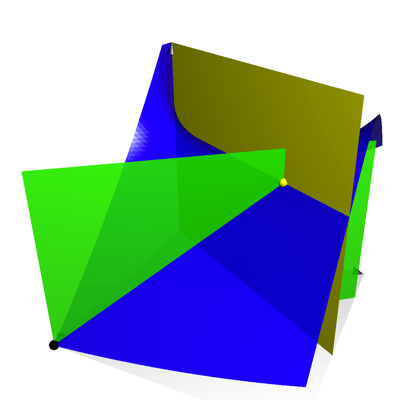}
\end{subfigure}
\begin{subfigure}{.24\textwidth}
  \centering
  \includegraphics[clip,width=1\linewidth]{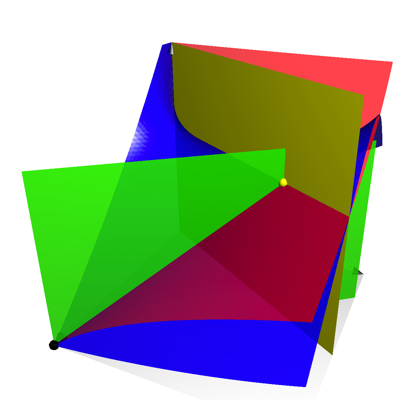}
\end{subfigure}
\begin{subfigure}{.24\textwidth}
  \centering
  \includegraphics[clip,width=1\linewidth]{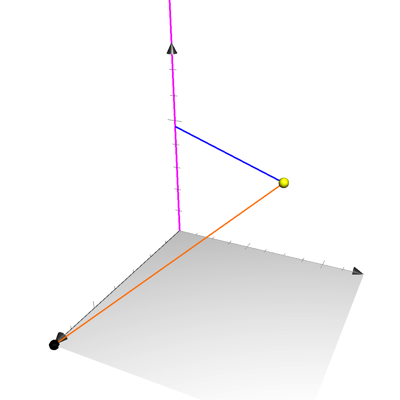}
\end{subfigure}
\caption{Figures above illustrate the set $\mathcal{S}$ for $m=1$ in the $X_1X_2Z$-space. The set $\mathcal{S}$ is a 1-dimensional set as shown in the last figure. The orange line segment is $\gamma_\infty$ that joins $p_0^+$ (black) and $p_2^+$ (yellow), the blue line is part of $\Phi$ and the magenta line is the non-compact part of $\mathcal{S}$ with vanished $X_1$ and $X_2$.}
\label{fig: 3dm=1}
\end{figure}

With Proposition \ref{prop: S1 for m=1} established, we can take $m=1$ as our ``initial case'' for further analysis of cases with $m>1$. In particular, we prove the following technical proposition.
\begin{proposition}
\label{prop: Q is nonnegative}
For $m\geq 1$, the inequality $Q\geq 0$ holds on the set $\mathcal{S}\cap \{P=0\}$. For $m=1$, the inequality reaches equality only at $\{p_0^+, p_2^+\} \cup\gamma_\infty\cup\Gamma$. For $m\geq 2$, the inequality reaches equality only at $p_0^+$, a point on $\Phi$, and points in $\Gamma$. 
\end{proposition}

\begin{proof}
We consider $\mathcal{S}\cap\{P=0\}$ as a union of slices
$$
\mathcal{S}\cap\{P= 0\}=\bigcup_{\kappa\in[0,1]}\mathcal{L}_\kappa,\quad \mathcal{L}_\kappa:=\mathcal{S}\cap\{P= 0\}\cap\{X_2-\kappa X_1=0\}.
$$
Note that each $\mathcal{L}_\kappa$ contains $\Gamma$. As $X_1\geq X_2\geq 0$ in $\mathcal{S}$, the function $Q$ vanishes once $X_1$ does. We assume $X_1>0$ in the following discussion.

For $\mathcal{L}_0$, we have
$$\left.Q\right|_{\mathcal{L}_0}=-2(4m+8)X_1YZ+2X_1\frac{1-H^2}{n}.$$ 
From $A\geq 0$ we have $Z=0$. Then $\left.Q\right|_{\mathcal{L}_0}=2X_1\frac{1-H^2}{n}$, and $P=0$ becomes $-X_1\frac{1-H^2}{n}=0$. Hence $\mathcal{L}_0=\Gamma\cup \{p_0^+\}$ and $\left.Q\right|_{\mathcal{L}_0}= 0$.

If $\kappa=1$, let $X:=X_1=X_2>0$. Then $A\geq 0$ and $P= 0$ respectively become
$$
X\left(Y-\frac{2m+3}{m}Z\right)\geq 0,\quad X(Y-(2m+3)Z)(Y-Z)= 0.
$$
For $X>0$ we must have $Y= (2m+3)Z$. Then from \eqref{eqn:_new positive conservation law2} we find that $Y=(2m+3)Z=(2m+3)z_0$. Hence $\mathcal{L}_1$ is a union of $\Gamma$ and a part of the invariant set \eqref{eqn:_jensen based cone}. More specifically, from $B\geq 0$ we have
\begin{equation}
\label{eqn:_cancel 1 for k=1}
\begin{split}
\frac{1}{n}-nX^2-\frac{2n^2(2m+3)(m-1)}{m(2m+1)(8m+3)}YZ=\frac{1}{n}-nX^2-\frac{2n^2(2m+3)^2(m-1)}{m(2m+1)(8m+3)}z_0^2\geq 0
\end{split}.
\end{equation}
and it follows that
\begin{equation}
\mathcal{L}_1=\Gamma \cup \left(\Phi \cap \left\{0<X\leq \frac{1}{n}\sqrt{\frac{36m^3+90m^2+117m+54}{m(8m+3)(4m^2+14m+9)}}\leq \frac{1}{n}\right\}\right)
\end{equation}
for $m\geq 2$. For $m=1$, we have 
\begin{equation}
\mathcal{L}_1=\Gamma \cup \left(\Phi \cap \left\{0<X\leq \frac{1}{n}\right\}\right)\cup \{p_2^+\}.
\end{equation}
And $\left.Q\right|_{\mathcal{L}_1}$ becomes
\begin{equation}
\begin{split}
&\left.Q\right|_{\mathcal{L}_1}\\
&=X\left(-4Y^2-(4m+8)(4m+2)YZ+(n+1)\left(\frac{1}{n}-nX^2\right)+4\left(1+\frac{2m}{3}\right)\left(n+\frac{2m}{3}\right)X^2\right)\\
&= X\left(-4(2m+3)^2z_0^2-(4m+8)(4m+2)(2m+3)z_0^2+\frac{n+1}{n}-\frac{4}{9}m(8m+3)X^2\right)\\
&= \frac{4}{9}m(8m+3)X\left(\frac{1}{n^2}\frac{36m^3+90m^2+117m+54}{m(8m+3)(4m^2+14m+9)}-X^2\right)\\
&\geq 0.
\end{split}
\end{equation}
Hence for $m\geq 1$, the function $\left.Q\right|_{\mathcal{L}_1}\geq 0$ and it only vanishes on $\Gamma$ and boundary points of $\mathcal{L}_1\cap \Phi$. Note that for $m=1$, one of the boundary points of $\mathcal{L}_1\cap \Phi$ is $p_2^+$.

For each $\mathcal{L}_\kappa$ with $\kappa\in(0,1)$, replace $X_2$ with $X_1\kappa$. Then \eqref{eqn:_new positive conservation law2} and $P= 0$ respectively become
\begin{equation}
\label{eqn:_cancel 1}
\begin{split}
&\frac{n-1}{n}-\frac{12m}{n}(1-\kappa)^2X_1^2=6Y^2+4m(4m+8)YZ-12mZ^2,\\
&\frac{\kappa-1}{n}-\frac{\kappa-1}{3n}(32\kappa^2m^2-12\kappa^2m+48m\kappa-18\kappa+27)X_1^2= 2\kappa Y^2-(4m+8)YZ+(4m\kappa+6)Z^2.
\end{split}
\end{equation}

With the equations above, we can write the constant $1$ as a homogeneous polynomial $One(Y,Z)$ of degree $2$. Multiplying the constant term in $B$ by $One(Y, Z)$, the function $B$ restricted to each $\mathcal{L}_\kappa\cap \{P=0\}$ then becomes a homogeneous polynomial in $Y$ and $Z$. Factor out $X_1$ in $Q$ and multiply the constant term in $\frac{Q}{X_1}$ by $One(Y,Z)$. We see that $Q$ restricted on each $\mathcal{L}_\kappa\cap \{P=0\}$ is a homogeneous polynomial in $X_1$, $Y$ and $Z$. In summary, we have
\begin{equation}
\label{eqn:_homogenize QAB}
\begin{split}
&\left.A\right|_{\mathcal{L}_\kappa\cap \{P=0\}}=X_1\left(Y\kappa-\frac{3}{m}Z\left(1+\frac{2m}{3}\kappa\right)\right),\\ 
&\left.B\right|_{\mathcal{L}_\kappa\cap \{P=0\}}=b_2Z^2+b_1YZ+b_0Y^2,\quad \left.Q\right|_{\mathcal{L}_\kappa\cap \{P=0\}}=X_1(q_2 Z^2+q_1YZ+q_0 Y^2).
\end{split}
\end{equation}
The coefficients $b_i$ and $q_i$ in \eqref{eqn:_homogenize QAB} are rational functions in $m$ and $\kappa$. Explicit formulas for each coefficient are presented in the Appendix for the sake of simplicity. 
As $Y$ is positive in $\mathcal{S}$ from Proposition \ref{prop: basic prop of S1}, we factor out $Y$ and consider $\left.A\right|_{\mathcal{L}_\kappa\cap \{P=0\}}=X_1Y\tilde{A}_{m,\kappa}\left(\frac{Z}{Y}\right)$, $\left.B\right|_{\mathcal{L}_\kappa\cap \{P=0\}}=Y^2\tilde{B}_{m,\kappa}\left(\frac{Z}{Y}\right)$ and $\left.Q\right|_{\mathcal{L}_\kappa\cap \{P=0\}}=X_1Y^2\tilde{Q}_{m,\kappa}\left(\frac{Z}{Y}\right)$, where
\begin{equation}
\label{eqn:_tilde QAB}
\tilde{A}_{m,\kappa}(x)=\kappa-\left(\frac{3}{m}+2\kappa\right)x,\quad 
\tilde{B}_{m,\kappa}(x)=b_2 x^2+b_1 x+b_0,\quad \tilde{Q}_{m,\kappa}(x)=q_2 x^2+q_1 x+q_0.
\end{equation}

We have
\begin{equation}
\label{eqn:_coefficient q2}
q_2=-\frac{4(2\kappa m+3)(32\kappa^3m^3+\kappa^2m^2(96-68\kappa)+\kappa m(90-84\kappa)+27(1-\kappa))}{3(1-\kappa)(2\kappa^2m(8m-5)+6\kappa(4m-1)+9)}\leq 0
\end{equation}
for any $(m,\kappa)\in [1,\infty)\times (0,1)$. Therefore, the restricted function $\left.Q\right|_{\mathcal{L}_\kappa\cap \{P=0\}}$ is non-negative if it is so on the boundary of $\mathcal{L}_\kappa\cap \{P=0\}$. 
The upper bound and the lower bound of $\frac{Z}{Y}$ on each slice are respectively provided by $\left.A\right|_{\mathcal{L}_\kappa\cap \{P=0\}}\geq 0$ and $\left.B\right|_{\mathcal{L}_\kappa\cap \{P=0\}}\geq 0$. Specifically, from $\left.A\right|_{\mathcal{L}_\kappa\cap \{P=0\}}\geq 0$ we have $\frac{Z}{Y}\leq \frac{m\kappa}{3+2m\kappa}$. As shown in \eqref{eqn:_coefficients for tildeB}, we have $b_2<0$. By Proposition \ref{prop: some property of B} in the Appendix, the smaller real root $\sigma(m,\kappa)$ of $\tilde{B}_{m,\kappa}$ is in the interval $(0,\frac{m\kappa}{3+2m\kappa})$. Hence from $\left.B\right|_{\mathcal{L}_\kappa\cap \{P=0\}}\geq 0$ we have $\frac{Z}{Y}\geq \sigma(m,\kappa)$. By the arbitrariness of $\kappa$, it is clear that the minimizing point of $Q$ on ${\mathcal{S}\cap\{P=0\}}$ lies on $\mathcal{S}\cap\{P=0\}\cap\{A=0\}$ or $\mathcal{S}\cap\{P=0\}\cap\{B=0\}$.

We have 
\begin{equation}
\label{eqn:_Q restricted at A=0}
\begin{split}
\left.Q\right|_{\mathcal{L}_\kappa\cap \{A=0\}}&=X_1Y^2\tilde{Q}_{m,\kappa}\left(\frac{m\kappa}{3+2m\kappa}\right)\\
&=X_1Y^2\frac{4(m+3)(m-1)(4\kappa^3m^2(8m-1)+4\kappa^2 m(8m-3)+18\kappa m+9(1-\kappa))\kappa^2}{3(2\kappa m+3)(1-\kappa)(2\kappa^2m(8m-5)+6\kappa(4m-1)+9)}\\
&\geq 0.
\end{split}
\end{equation}
Therefore, for $m\geq 2$, the function $Q$ is positive on $\mathcal{L}_\kappa\cap \{A=0\}$ for any $\kappa\in(0,1)$. For $m=1$, the function $Q$ vanishes at $\mathcal{S}\cap \{P=0\}\cap \{A=0\}$.

Proving the non-negativity of $\left.Q\right|_{\mathcal{L}_\kappa\cap \{B=0\}\cap \{P=0\}}$ is a bit more computationally involved. From Proposition \ref{prop: S1 for m=1}, we know that for $m=1$, the polynomial $A$, $B$ and $P$ identically vanish at $\gamma_\infty$. Therefore, an explicit formula for the root $\sigma(1,\kappa)$ can be obtained from $A=0$. We have
$$
\sigma(1,\kappa)=\frac{Z}{Y}=\frac{X_2}{3\left(X_1+\frac{2}{3}X_2\right)}=\frac{\kappa}{2\kappa+3}.
$$
Define the function $F(m,\kappa):=\tilde{Q}_{m,\kappa}(\sigma(m,\kappa))$. To show that  $\left.Q\right|_{\mathcal{L}_\kappa\cap \{P=0\}\cap \{B=0\}}\geq 0$, it suffices to show that $F(m,\kappa)\geq 0$ for any $(m,\kappa)\in [1,\infty)\times (0,1)$. Note that the vanishing of $F(m,\kappa)$ means $\sigma(m,\kappa)$ being also a root of $\tilde{Q}_{m,\kappa}$. From the computation \eqref{eqn:_Q restricted at A=0} we have 
$$F(1,\kappa):=\tilde{Q}_{1,\kappa}(\sigma(1,\kappa))=\tilde{Q}_{1,\kappa}\left(\frac{\kappa}{3+2\kappa}\right)=0$$
for any $\kappa\in (0,1)$. Furthermore, by implicit derivative, we have 
$$
\left(\frac{\partial \sigma}{\partial m}\right)(1,\kappa)=\left(-\frac{\frac{\partial b_2}{\partial m}\sigma^2+\frac{\partial b_1}{\partial m}\sigma+\frac{\partial b_0}{\partial m}}{2b_2\sigma+b_1}\right)(1,\kappa)=-\frac{(92\kappa^4+1202\kappa^3+1458\kappa^2-367\kappa-537)\kappa}{66(3+4\kappa)(\kappa+1)(3+2\kappa)^2},$$
and it follows that
$$
\left(\frac{\partial F}{\partial m}\right)(1,\kappa)=\left(\frac{\partial q_2}{\partial m}\sigma^2+\frac{\partial q_1}{\partial m}\sigma+\frac{\partial q_0}{\partial m}+2q_2\sigma\frac{\partial \sigma}{\partial m}+q_1\frac{\partial \sigma}{\partial m}\right)(1,\kappa)=\frac{4(\kappa-1)(184\kappa^2-244\kappa-339)\kappa^2}{99(3+4\kappa)(\kappa+1)(2\kappa+3)}>0.
$$
Hence $F\geq 0$ on a neighborhood around $\{(1,\kappa)\mid \kappa\in(0,1)\}\subset [1,\infty)\times (0,1)$. In other words, for an $m$ that is slightly larger than $1$, the root $\sigma(m,\kappa)$ of $\tilde{B}_{m,\kappa}$ is strictly between the two real roots of $\tilde{Q}_{m,\kappa}$. Hence proving $F\geq 0$ on $[1,\infty)\times (0,1)$ is equivalent to showing that $\sigma(m,\kappa)$ \emph{stays} between the two real roots of $\tilde{Q}_{m,\kappa}$ for varying $(m,\kappa)$. This idea leads us to consider the resultant $r(\tilde{Q}_{m,\kappa},\tilde{B}_{m,\kappa})$ for the two polynomials. We have
\begin{equation}
\label{eqn:_resultant QB}
\begin{split}
r(\tilde{Q}_{m,\kappa},\tilde{B}_{m,\kappa})&=-\frac{64\kappa^2(m-1)(2\kappa m+3)}{(8m+3)^2(2m+1)^2m^2(1-\kappa)(2\kappa^2m(8m-5)+6\kappa(4m-1)+9)^2}\tilde{r},\\
\tilde{r}&=262144 \kappa ^4 m^{10}+(516096 \kappa ^4+679936 \kappa ^3) m^9+(373760 \kappa ^4+1233920 \kappa ^3+675840 \kappa ^2) m^8\\
&\quad +(-275904 \kappa ^4+1151040 \kappa ^3+1143744 \kappa ^2+308160 \kappa ) m^7\\
&\quad +(-926496 \kappa ^4+248832 \kappa ^3+1432512 \kappa ^2+507456 \kappa +54432) m^6\\
&\quad +(-800496 \kappa ^4-1256256 \kappa ^3+1472688 \kappa ^2+857520 \kappa +92016) m^5\\
&\quad +(-281880 \kappa ^4-1525392 \kappa ^3+266328 \kappa ^2+1353024 \kappa +199584) m^4\\
&\quad +(-33048 \kappa ^4-644436 \kappa ^3-672624 \kappa ^2+957420 \kappa +375192) m^3\\
&\quad +(-90396 \kappa ^3-433026 \kappa ^2 +186624 \kappa +333882) m^2\\
&\quad +(-74358 \kappa ^2-59049 \kappa +133407) m+19683(1-\kappa).
\end{split}
\end{equation}
As verified in Proposition \ref{prop: resultant is non-positive} in the Appendix, the inequality $r(\tilde{Q}_{m,\kappa},\tilde{B}_{m,\kappa})
\leq 0$ is valid for any $(m,\kappa)\in [1,\infty)\times (0,1)$. Furthermore, the function $r(\tilde{Q}_{m,\kappa},\tilde{B}_{m,\kappa})$ vanishes if and only if $m=1$. In particular, both $\tilde{Q}_{1,\kappa}$ and $\tilde{B}_{1,\kappa}$ have $\sigma(1,\kappa)=\frac{\kappa}{2\kappa+3}$ as their roots. For $m>1$, the polynomials $\tilde{Q}_{m,\kappa}$ and $\tilde{B}_{m,\kappa}$ do not share any common root. Hence $F\geq 0$ on $[1,\infty)\times (0,1)$ and the function $F$ vanishes if and only if $m=1$. Therefore, for $m\geq 2$, the inequality $\left.Q\right|_{\mathcal{L}_\kappa\cap \{P=0\}\cap \{B=0\}}\geq 0$ is valid and the equality is reached only at $p_0^+$ and a point $\Phi\cap \{B=0\}$. The proof is complete.
\end{proof}
\begin{figure}[h!]
\begin{subfigure}{.32\textwidth}
  \centering
  \includegraphics[clip,width=1\linewidth]{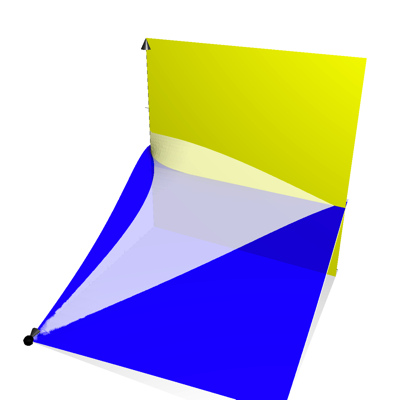}
\caption{$\mathcal{S}\cap\{P=0\}\cap\{Q\geq 0\}$}
\end{subfigure}
\begin{subfigure}{.32\textwidth}
  \centering
  \includegraphics[clip,width=1\linewidth]{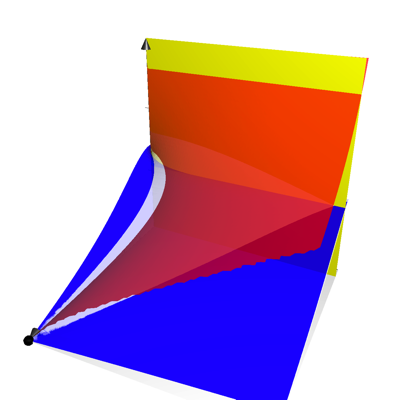}
 \caption{$\mathcal{S}\cap\{P=0\}\cap\{A\geq 0\}\cap\{Q\geq 0\}$}
\end{subfigure}
\begin{subfigure}{.32\textwidth}
  \centering
  \includegraphics[clip,width=1\linewidth]{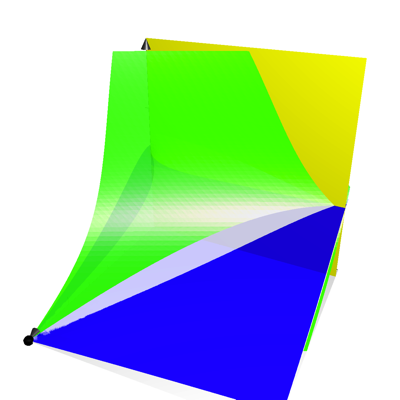}
 \caption{$\mathcal{S}\cap\{P=0\}\cap\{B\geq 0\}\cap\{Q\geq 0\}$}
\end{subfigure}
\caption{Figures above illustrate Proposition \ref{prop: Q is nonnegative} for $m=5$. Points in $\mathcal{S}\cap\{P=0\}$ (blue) are at the front ot $\{X_1-X_2=0\}$ (yellow), below $\{A=0\}$ (red) and behind $\{B=0\}$ (green). Proposition \ref{prop: Q is nonnegative} shows that $\mathcal{S}\cap\{P=0\}$ is below $\mathcal{S}\cap\{Q=0\}$ (white).}
\label{fig: Q}
\end{figure}

With the help of the proceeding proposition, we are ready to prove the following lemma.
\begin{lemma}
\label{lem: leave S only through X1=X2}
For $m\geq 2$, integral curves $\gamma_{s_1}$ that are in the interior of $\mathcal{S}$ can only escape through some point in $\mathcal{E}\cap\{H\geq 0\}\cap \{X_1-X_2=0\}$.
\end{lemma}
\begin{proof}
The boundary $\partial\mathcal{S}$ is a union of the following five sets:
$$
\mathcal{S}\cap\{X_1-X_2=0\},\quad \mathcal{S}\cap\{X_2=0\},\quad \mathcal{S}\cap \{A=0\},\quad \mathcal{S}\cap \{B=0\},\quad \mathcal{S}\cap \{P=0\}.
$$
By Proposition \ref{prop: basic prop of S1}, if a $\gamma_{s_1}$ escapes $\mathcal{S}$ through some point with vanished $X_2$-coordinate, the point must lie in $\mathcal{S}\cap\{A=0\}\cap \{P=0\}$. Hence we aim to show that $V$ points inward when restricted to the last three parts of $\partial\mathcal{S}$. 

A straightforward computation shows that 
\begin{equation}
\label{eqn:_deri fof A}
\begin{split}
&\left.\left\langle\nabla A,V\right\rangle\right|_{A=0}\\
&=A\left(2H\left(G+\frac{1}{n}(1-H^2)\right)-2X_1-(4m+2)X_2\right) +\frac{A}{X_1+\frac{2m}{3}X_2}\left(R_1+\frac{2m}{3}R_2-\left(1+\frac{2m}{3}\right)\frac{1}{n}(1-H^2)\right)\\
&\quad +\frac{Y}{X_1+\frac{2m}{3}X_2}P\\
&=\frac{Y}{X_1+\frac{2m}{3}X_2}P\quad \text{since $A=0$}\\
&\geq 0.
\end{split}
\end{equation}
It is confirmed that $\left.V\right|_{A=0}$ points inward $\mathcal{S}$.

Since
\begin{equation}
\label{eqn:_deri fof B}
\begin{split}
&\left.\left\langle\nabla B,V\right\rangle\right|_{B=0}\\
&=-2\frac{H}{n}(H^2-1)\left(G+\frac{1}{n}(1-H^2)\right)-\frac{2n^2(2m+3)(m-1)}{m(16m^2+14m+3)}YZ \left(2H\left(G+\frac{1}{n}(1-H^2)\right)-2X_2\right)\\
&\quad \text{by \eqref{eqn:_derivative of H}}\\
&=2BH\left(G+\frac{1}{n}(1-H^2)\right)+\frac{4n^2(2m+3)(m-1)}{m(16m^2+14m+3)}YZ X_2\\
&=\frac{4n^2(2m+3)(m-1)}{m(16m^2+14m+3)}YZ X_2\quad \text{since $B=0$}\\
&\geq 0,
\end{split}
\end{equation}
it is clear that $\left.V\right|_{B=0}$ points inward $\mathcal{S}$.

Since
\begin{equation}
\label{eqn:_deri fof P}
\begin{split}
&\left.\langle\nabla P,V\rangle\right|_{P=0}\\
&=P\left(H\left(3G+\frac{3}{n} (1-H^2)-1\right)+\frac{4m}{3}X_2\right)+(X_1-X_2)Q\\
&=(X_1-X_2)Q \quad \text{since $P=0$}\\
&\geq 0 \quad \text{by Proposition \ref{prop: Q is nonnegative}},
\end{split}
\end{equation}
we learn that $\left.V\right|_{P=0}$ also points inward $\mathcal{S}$.

Finally, we need to exclude the possibility of non-transversal passing of a $\gamma_{s_1}$ through some point with $X_1\neq X_2$. By \eqref{eqn:_deri fof B} and Proposition \ref{prop: basic prop of S1}, such a point does not exist on $\mathcal{S}\cap \{B=0\}$. Suppose there were such a point on $\mathcal{S}\cap \{P=0\}$, then $P=Q=0$ at that point by \eqref{eqn:_deri fof P}. Then by Proposition \ref{prop: Q is nonnegative}, we know that such a point is either $p_0^+$ or a point on $\Phi$, which is impossible. Suppose the non-transversal passing point exists on $\mathcal{S}\cap \{A=0\}$, then by \eqref{eqn:_deri fof A} and Proposition \ref{prop: basic prop of S1}, we must have $A=P=0$ at that point, which is also impossible.
\end{proof}

To show that some $\gamma_{s_1}$ is initially in the set $\mathcal{S}$ we need the following technical proposition.
\begin{proposition}
\label{prop: Q is positive initially}
Define $\check{A}:=YX_2-\frac{m+2}{m}Z\left(X_1+\frac{2m}{3}X_2\right)$.
For $m\geq 2$, the function $Q$ is positive on the set
$$
\check{\mathcal{S}}=\mathcal{E}\cap\{X_1-X_2> 0\}\cap \{X_2> 0\}\cap\{A> 0\}\cap \{\check{A}< 0\}\cap \{P< 0\}.
$$
\end{proposition}

\begin{proof}
We consider $\check{\mathcal{S}}$ as a union of slices
$$
\check{\mathcal{S}}=\bigcup_{\kappa\in (0,1)}\check{\mathcal{L}}_\kappa, \quad \check{\mathcal{L}}_\kappa:=\check{\mathcal{S}}\cap \{X_2-\kappa X_1=0\}.
$$
For each $\check{\mathcal{L}}_\kappa$ with $\kappa\in(0,1)$, replace $X_2$ with $\kappa X_1$. Then from \eqref{eqn:_new positive conservation law2} and $P< 0$ we have
\begin{equation}
\label{eqn:_cancel X1 tilde S}
\begin{split}
0&> \left(\frac{1-\kappa}{3n}(32\kappa^2m^2-12\kappa^2m+48m\kappa-18\kappa+27)-\frac{12m}{n(n-1)}(1-\kappa)^3\right)X_1^2\\
&\quad -\left(2\kappa+\frac{6(1-\kappa)}{n-1}\right)Y^2+(4m+8)\left(1-4m\frac{1-\kappa}{n-1}\right)YZ+\left(\frac{1-\kappa}{n-1}12m-(4m\kappa+6)\right)Z^2\\
&= \frac{(1-\kappa)(16\kappa^2m^2-10\kappa^2 m+24\kappa m-6\kappa+9)}{6m+3}X_1^2\\
&\quad -\left(2\kappa+\frac{6(1-\kappa)}{n-1}\right)Y^2+(4m+8)\left(1-4m\frac{1-\kappa}{n-1}\right)YZ+\left(\frac{1-\kappa}{n-1}12m-(4m\kappa+6)\right)Z^2.
\end{split}
\end{equation}
The coefficient for $X_1^2$ in \eqref{eqn:_cancel X1 tilde S} is obviously positive for any $(m,\kappa)\in [2,\infty)\times (0,1)$.
On the other hand, use \eqref{eqn:_new positive conservation law2} to replace the constant term in $Q$ with homogeneous polynomials in $X_1$, $Y$, and $Z$. The polynomial $Q$ restricted on $\check{\mathcal{L}}_\kappa$ becomes 
\begin{equation}
\label{eqn:_homgenized Q tilde S}
\begin{split}
\left.Q\right|_{\check{\mathcal{L}}_\kappa}& =\frac{-64\kappa^3 m^3+(40\kappa-96)\kappa^2 m^2+(36\kappa^2+60\kappa-108)\kappa m+54(\kappa-1)}{18m+9}X_1^3\\
&\quad +X_1\left(\left(2\kappa+\frac{6}{2m+1}\right)Y^2-\frac{8m+16}{2m+1}YZ-\left(12m\kappa + \frac{12 m}{2m+1}\right)Z^2\right).
\end{split}
\end{equation}
It is obvious that the coefficient for $X_1^3$ in \eqref{eqn:_homgenized Q tilde S} is negative for any $(m,\kappa)\in [2,\infty)\times (0,1)$. Note that if \eqref{eqn:_cancel X1 tilde S} reaches equality, then one can write $X_1^2$ as a homogeneous polynomial in $Y$ and $Z$. Moreover, substituting the $X_1^2$ by the homogeneous polynomial in \eqref{eqn:_homgenized Q tilde S} gives the formula for $X_1Y^2\tilde{Q}\left(\frac{Z}{Y}\right)$ as in \eqref{eqn:_tilde QAB}. Therefore, from \eqref{eqn:_cancel X1 tilde S} we obtain
$$
\left.Q\right|_{\check{\mathcal{L}}_\kappa} > X_1Y^2\tilde{Q}\left(\frac{Z}{Y}\right).
$$
Note that $\tilde{Q}\left(\frac{Z}{Y}\right)$ above is defined on $\check{\mathcal{S}}$ instead of $\mathcal{S}\cap \{P=0\}$. On the other hand, the inequalities $A> 0$ and $\check{A}< 0$ become
$$
\frac{3m\kappa}{3(m+2)+2m(m+2)\kappa}<\frac{Z}{Y}<\frac{m\kappa}{3+2m\kappa}.
$$
As shown in \eqref{eqn:_coefficient q2}, it is clear that the coefficient $q_2$ is negative. It suffices to show that $\tilde{Q}$ is positive at $\frac{m\kappa}{3+2m\kappa}$ and $\frac{3m\kappa}{3(m+2)+2m(m+2)\kappa}$ to prove that the polynomial is positive on the open interval in between.
From \eqref{eqn:_Q restricted at A=0} it is clear that $\tilde{Q}\left(\frac{m\kappa}{3+2m\kappa}\right)>0$. For any $(m,\kappa)\in [2,\infty)\times (0,1)$, a straightforward computation shows that
\begin{equation}
\begin{split}
&\tilde{Q}\left(\frac{3m\kappa}{3(m+2)+2m(m+2)\kappa}\right)\\
&=\frac{512 (m-1)\kappa^2}{3(1 - \kappa) (2 \kappa m + 3)(2\kappa^2 m(8m-5)+6\kappa(4m-1)+9)(m+2)^2}\left( \left(m^3 + \frac{11}{8} m^2 - \frac{19}{4} m - \frac{3}{8}\right) m^2\kappa^3\right.\\
&\quad  \left. + \frac{3}{2}\left(m^3 - \frac{5}{8}m^2 +3 m -\frac{3}{4} \right)m \kappa^2 + \left(\frac{9}{8}m^3 - \frac{99}{64} m^2 + \frac{27}{16}m - \frac{27}{32}\right) \kappa + \left(\frac{27}{64} m^2 + \frac{27}{32}\right)\right)\\
&>0.
\end{split}
\end{equation}
 The proof is complete.
\end{proof}

We are ready to prove the following lemma.
\begin{lemma}
\label{lem: initially in S}
For $m\geq 2$, the integral curve $\gamma_{s_1}$ is initially inside the interior of $\mathcal{S}$ if $s_1\in \left(\frac{3}{m-1},\frac{9(5m+3)(4m^2+4m+3)}{n^2(2m+3)(m-1)}\right)$.
\end{lemma}
\begin{proof}
With the linearized solution \eqref{eqn:_linearized solution near p_0}, we have
\begin{equation}
\label{eqn:_all the linearized evaluation}
\begin{split}
X_1(\gamma_{s_1})&=\frac{1}{3}-(8m^2+18m+18+(4m+8)ms_1)e^{\frac{2}{3}\eta}+O\left(e^{\left(\frac{2}{3}+\epsilon\right)\eta}\right)\\
X_2(\gamma_{s_1})&=(-9+3s_1(m+2))e^{\frac{2}{3}\eta}+O\left(e^{\left(\frac{2}{3}+\epsilon\right)\eta}\right)\\
A(\gamma_{s_1})&= \frac{m+3}{m}((m-1)s_1-3)e^{\frac{2}{3}\eta}+O\left(e^{\left(\frac{2}{3}+\epsilon\right)\eta}\right)\\
B(\gamma_{s_1})&= \frac{2n^2(2m+3)(m-1)}{m(2m+1)(8m+3)}\left(\frac{9(5m+3)(4m^2+4m+3)}{n^2(2m+3)(m-1)}-s_1\right)e^{\frac{2}{3}\eta}+O\left(e^{\left(\frac{2}{3}+\epsilon\right)\eta}\right)\\
P(\gamma_{s_1})&= O\left(e^{\left(\frac{2}{3}+\epsilon\right)\eta}\right)\\
Q(\gamma_{s_1})&= O\left(e^{\left(\frac{2}{3}+\epsilon\right)\eta}\right)
\end{split}
\end{equation}
near $p_0^+$. Hence functions $X_2$, $A$ and $B$ are positive along $\gamma_{s_1}$ near $p_0^+$ if $s_1\in \left(\frac{3}{m-1},\frac{9(5m+3)(4m^2+4m+3)}{n^2(2m+3)(m-1)}\right)$. Furthermore, the last two equalities in \eqref{eqn:_all the linearized evaluation} show that $\nabla P$ and $\nabla Q$ are perpendicular to the linearized solution \eqref{eqn:_linearized solution near p_0} at $p_0$. Hence the integral curve $\gamma_{s_1}$ is tangent to $\{P=0\}$ and $\{Q=0\}$ for any $s_1$. It takes a little bit more work to show that the function $P$ is initially positive along $\gamma_{s_1}$. Recall in \eqref{eqn:_deri fof P}, we have 
\begin{equation}
\label{eqn:_P'}
P'=P\left(H\left(3G+\frac{3}{n}(1-H^2)-1\right)+\frac{4m}{3}X_2\right)+(X_1-X_2)Q.
\end{equation}
Note that $$\left(H\left(3G+\frac{3}{n}(1-H^2)-1\right)+\frac{4m}{3}X_2\right)(\gamma_{s_1})= -(48m^2+84m+20m(m+2)s_1)e^{\frac{2}{3}\eta}+O\left(e^{\left(\frac{2}{3}+\epsilon\right)\eta}\right)$$
near $p_0^+$.
If $P$ were negative initially along $\gamma_{s_1}$ near $p_0^+$, the first term in \eqref{eqn:_P'} is positive. Moreover, from the linearized solution \eqref{eqn:_linearized solution near p_0}, we have
$$
\check{A}(\gamma_{s_1})=-\frac{1}{m}\left(6m+6-(m+2)(m-1)s_1\right)e^{\frac{2}{3}\eta}+O\left(e^{\left(\frac{2}{3}+\epsilon\right)\eta}\right).
$$
Since $\frac{9(5m+3)(4m^2+4m+3)}{n^2(2m+3)(m-1)}< \frac{6(m+1)}{(m+2)(m-1)}$ for $m\geq 2$, we know that $\check{A}(\gamma_{s_1})$ is initially negative along $\gamma_{s_1}$. Based on the assumption that $P$ is initially negative along $\gamma_{s_1}$, we know that $\gamma_{s_1}$ is initially in $\check{\mathcal{S}}$. By Proposition \ref{prop: Q is positive initially}, we know that the second term in \eqref{eqn:_P'} is also positive and so is $P'$, which is a contradiction. Therefore, the function $P$ must be positive initially along $\gamma_{s_1}$ near $p_0^+$.
\end{proof}

According to Section 4 in \cite{bohm_inhomogeneous_1998}, the existence of the heterocline that joins $p_0^\pm$ relies on the number of critical points of $\sqrt{\frac{Z}{Y}}$ that appear before the turning point. The number is originally denoted by $\sharp C_w(\bar{h})$ in \cite{bohm_inhomogeneous_1998}, where $w$ is the ratio $\frac{f_1}{f_2}=\sqrt{\frac{Z}{Y}}$ and $\bar{h}$ corresponds to the initial data $f$ in \eqref{eqn:_initial condition 1}. We introduce the following modified definitions of $\sharp C_w(\bar{h})$ and $W$-intersection points.
\begin{definition}
\label{def: counting}
For a $\gamma_{s_1}$ that is not a heterocline, i.e., a $\gamma_{s_1}$ that is not defined on $\mathbb{R}$ or $\lim\limits_{\eta\to \infty}\gamma_{s_1}\neq p_0^-$, let $\sharp C(\gamma_{s_1})$ be the number of critical points of the function $\sqrt{\frac{Z}{Y}}$ along $\gamma_{s_1}$ that appear in $\mathcal{E}\cap\{H>0\}\cap\{Y-Z> 0\}$.
\end{definition}

\begin{definition}
\label{def: W-intersection}
The point where $\gamma_{s_1}$ or $\zeta_{s_2}$ intersects $\mathcal{E}\cap\{H > 0\}\cap\{Y-Z=0\}$ is called a $W$-intersection point.
\end{definition}

We have the following proposition.
\begin{proposition}
\label{prop:turning or W-intersection}
Any $\gamma_{s_1}$ with $s_1>-3$ (or $\zeta_{s_2}$ with $s_2>0$) has a turning point at $\mathcal{E}\cap\{H=0\}$ or a $W$-intersection point at $\mathcal{E}\cap \{Y-Z=0\}$.
\end{proposition}
\begin{proof}
An integral curve $\gamma_{s_1}$ with $s_1>-3$ (or $\zeta_{s_2}$ with $s_2>0$) is initially in the interior of the compact set $\mathcal{E}\cap\{H\geq 0\}\cap \{Y-Z\geq 0\}$. Suppose the integral curve does not have any turning point or any $W$-intersection point. Then it must be defined on $\mathbb{R}$. Furthermore, such an integral curve is in $\mathcal{E}\cap \{H<1\}$ initially. From \eqref{eqn:_derivative of H}, along the integral curve we eventually have $H^2<1-\epsilon$ for some $\epsilon>0$. Hence 
$$
H'=(H^2-1)\left(\frac{1}{n}+\frac{12m}{n}(X_1-X_2)^2\right)\leq \frac{H^2-1}{n}<-\frac{\epsilon}{n}
$$
eventually, meaning the function $H$ must at vanish some point. We reach a contradiction. Each of the integral curve without a turning point must have a $W$-intersection point.
\end{proof}
We rephrase Lemma 4.4 in \cite{bohm_inhomogeneous_1998} with these new definitions.
\begin{theorem}
\label{thm: Bohm local constant}
If $\gamma_{s_1}$ is not a heterocline for any $s_1\in [a,b]$, then $\sharp C(\gamma_{s_1})$ is a constant for all $s_1\in [a,b]$.
\end{theorem}

Immediately, we have the following proposition 
\begin{proposition}
\label{prop: positive X1-X2}
The quantities $\sharp C(\gamma_{0})$ and $\sharp C(\zeta_{\frac{1}{2m+6}})$ are both zero.
\end{proposition}
\begin{proof}
From defining equations of $\mathcal{B}_{QK}$, we have $X_1=Y-2Z$ and $X_2=-Z$ along $\gamma_{0}$. Then $H^2< 1$ becomes $(3Y-(n+3)Z)^2< 1$. From \eqref{eqn:_qk conservation} we have 
\begin{equation}
\begin{split}
&\frac{12m}{n}(Y-Z)^2+6Y^2+4m(4m+8)YZ-12mZ^2> \frac{n-1}{n}(3Y-(4m+6)Z)^2\\
\Leftrightarrow& n(n+9)(n-1)(Y-Z)Z> 0.
\end{split}
\end{equation}
Hence $X_1-X_2=Y-Z> 0$ along $\gamma_0$. Similarly, we have $X_2-X_1=Y-Z> 0$ along $\zeta_{\frac{1}{2m+6}}$. Note that $X_1-X_2$ vanishes at the critical point $p_1^-$ and it stays positive along $\gamma_0$.
\end{proof}

We are ready to prove Theorem \ref{thm: first Einstein metric}.
\begin{proof}[Proof of Theorem \ref{thm: first Einstein metric}]
Consider a $\gamma_{s_1}$ with $s_1\in \left(\frac{3}{m-1},\frac{9(5m+3)(4m^2+4m+3)}{n^2(2m+3)(m-1)}\right)$. From Lemma \ref{lem: initially in S} we know that such a $\gamma_{s_1}$ is initially in $\mathcal{S}$. From \eqref{eqn:_derivative of H} we know that such a $\gamma_{s_1}$ must exit $\mathcal{S}$. From Lemma \ref{lem: leave S only through X1=X2} we know that such a $\gamma_{s_1}$ exits $\mathcal{S}$ through the face $\mathcal{E} \cap\{H> 0\}\cap \{Y-Z\geq 0\}\cap \{X_1-X_2=0\}$. It is clear that $\left(\sqrt{\frac{Z}{Y}}\right)'=\sqrt{\frac{Z}{Y}}(X_1-X_2)$. Since $Z$ is initially positive along $\gamma_{s_1}$ with $s_1> \frac{3}{m-1}$, we know that $\sharp C(\gamma_{s_1})$ is exactly the number of times that $\gamma_{s_1}$ intersects $\mathcal{E} \cap\{H> 0\}\cap \{Y-Z\geq 0\}\cap \{X_1-X_2=0\}$. Hence $\sharp C(\gamma_{s_1})\geq 1$ for $s_1\in \left(\frac{3}{m-1},\frac{9(5m+3)(4m^2+4m+3)}{n^2(2m+3)(m-1)}\right)$.

On the other hand, the integral curve $\gamma_0$ joins $p_0^+$ and $p_1^-$. From Proposition \ref{prop: positive X1-X2}, the function $X_1-X_2$ stays positive along $\gamma_0$. Hence we have $\sharp C(\gamma_0)=0$. Therefore, by Theorem \ref{thm: Bohm local constant} there exists some $s_\star\in \left(0,\frac{9(5m+3)(4m^2+4m+3)}{n^2(2m+3)(m-1)}\right)$ such that $\gamma_{s_\star}$ is a heterocline that joins $p_0^\pm$. Theorem \ref{thm: first Einstein metric} is proved.
\end{proof}

\begin{remark}
With some small modifications, the polynomial $\mathcal{S}$ can be applied to prove the existence of positive Einstein metrics on  $F^{m+1}$, a cohomogeneity one space formed by the group triple $(Sp(m)U(1),Sp(m)Sp(1),Sp(m+1))$. Furthermore, some non-existence results can also be obtained from the defining polynomial $P$. For some cohomogeneity one spaces, the function $P$ is negative along all $\gamma_{s_1}$ in $\mathcal{E}\cap \{X_1-X_2\geq 0\}$, essentially forcing $X_1-X_2$ to be positive along these integral curves. For example, there is no $\spin(9)$-invariant cohomogeneity one positive Einstein metric on $\mathbb{OP}^2\sharp \overline{\mathbb{OP}}^2$. A systematic study of the existence problem on all cohomogeneity one spaces with two isotropy summands will be soon presented in later work.
\end{remark}

\begin{remark}
One can recover B\"ohm's metric on $\mathbb{HP}^2\sharp \overline{\mathbb{HP}}^2$ by enlarging the set $\mathcal{S}$ for $m=1$. Specifically, we can increase the coefficient for $Y$ in the polynomial $A$ properly so that: 1. Integral curves $\gamma_{s_1}$ with large enough $s_1$ are in the enlarged $\mathcal{S}$ initially. 2. Integral curves that are in the enlarged $\mathcal{S}$ must exit through the face $\mathcal{S}\cap\{X_1-X_2=0\}$. Note that the derivative of the new polynomial $A$ still depends on the non-negativity of the same polynomial $P$. Hence $\sharp C(\gamma_{s_1})\geq 1$ with large enough $s_1$ while  $\sharp C(\gamma_{0})=0$, and Theorem \ref{thm: Bohm local constant} can be applied to prove the existence.
\end{remark}

We end this section with the following remark to discuss the motivation in defining $\mathcal{S}$.
\begin{remark}
\label{rem: motivation for S}
Inspired by Corollary 5.8 in \cite{bohm_inhomogeneous_1998} and by the fact that $p_2^+$ is a stable node, we realized that taking the limit $s_1\to \infty$ for $\gamma_{s_1}$ may not provide enough information to prove the existence theorem for all $m$. From \eqref{eqn:_initial condition related to s_1} we know that $s_1$ is related to the initial condition $f$. Numerical data in Table 2 in \cite{bohm_inhomogeneous_1998} indicate that the winding angle of $\gamma_{s_1}$ around $\Phi$ may not be monotonic as $s_1$ increases. Hence it is reasonable to find a bounded interval of $s_1$ for which the winding angle of $\gamma_{s_1}$ is large enough.

From \eqref{eqn:_deri fof B} and \eqref{eqn:_all the linearized evaluation}, it appears that the upper bound for $s_1$ can be easily controlled by changing the coefficient for $YZ$ in $B$. Looking for an appropriate lower bound for $s_1$, on the other hand, is relatively more difficult.  Originally we choose to obtain the lower bound from $f_2\dot{f_2}\geq f_1\dot{f_1}$. The inequality is equivalent to $YX_2\geq ZX_1$. Although the inequality is geometrically motivated, showing it can be maintained before the winding angle gets large enough seems to be too difficult. We eventually define the polynomial $A$, whose first and second derivatives are relatively easier to be controlled by polynomials $P$ and $Q$.
\end{remark}

\section{Limiting Winding Angle}
\label{sec: Limiting Winding Angle}
From \cite{wink_cohomogeneity_2017} we know that $\gamma_\infty$ joins $p_0^+$ and $p_2^+$. By the symmetry of $\eqref{eqn:_new Positive Einstein system}$, we know that there exists $\bar{\gamma}_\infty$ that joins $p_2^-$ and $p_0^-$. As $\Phi$ joins $p_2^\pm$, we have this set of heteroclines $\{ \gamma_\infty, \Phi, \bar{\gamma}_\infty \}$ that joins $p_0^\pm$. From this perspective, the critical point $p_2^+$ is anticipated to play an important role in the qualitative analysis. 
Intuitively speaking, if $p_2^+$ were a stable focus in the Ricci-flat subsystem, the integral curve $\gamma_{s_1}$ would wind around $\Phi$ more frequently as $s_1$ increases. 
From Section \ref{sec: Linearization at critical points}, however, we learn that $p_2^+$ is a stable node in the Ricci-flat subsystem. Hence the winding behavior of $\gamma_{s_1}$ around $\Phi$ is less obvious. The new coordinate change helps us estimate the limiting winding angle of $\gamma_{s_1}$ as $s_1\to \infty$ and establish Theorem \ref{thm: second Einstein metric}. On the other hand, another set of heteroclines $\{\zeta_\infty, \Phi, \bar{\zeta}_\infty\}$ joins $p_1^\pm$. It is natural to ask if some heterocline other than $\zeta_0$ joins $p_1^\pm$. The new coordinate change also helps us to answer this question and prove Theorem \ref{thm: 8-sphere}.

We introduce some known estimates in the Ricci-flat system in the following. The Ricci-flat subsystem on $\mathcal{E}\cap \{H=1\}$ is simply a subsystem of (2.16) in \cite{chi_einstein_2020}, with $Y_1 \to \sqrt{2}$, $Y_2\to \sqrt{2}Y$ and $Y_3\to 2\sqrt{2}Z$. From Lemma 4.4 in \cite{chi_einstein_2020}, we learn that the compact set 
$$
\hat{\mathcal{B}}_{RF}:=\mathcal{E}\cap \{H=1\}\cap\{Y-Z\geq 0\}\cap \{X_2-X_1+2Y-2Z\geq 0\}\cap \left\{X_1\leq \frac{1}{2}\right\}\cap \{X_2\geq 0\}
$$
is invariant. Critical points $p_0^+$ and $p_1^+$ are on the boundary of $\hat{\mathcal{B}}_{RF}$ while $p_2^+$ is in the interior. Straightforward computations show that $\gamma_\infty$ and $\zeta_\infty$ are initially in $\hat{\mathcal{B}}_{RF}$. From Lemma 5.7 in \cite{chi_einstein_2020}, it is known that these two integral curves converge to $p_2^+$. We construct the following invariant set introduced in \cite{chi_cohomogeneity_2019}, which gives us more information on $\gamma_\infty$ near $p_2^+$. 
\begin{proposition}
\label{prop: Ricci flat invariant set}
The set 
$$
\mathcal{B}_{RF}:=\mathcal{E}\cap \{H=1\}\cap\{Y-(2m+3)Z\geq 0\}\cap \{X_2-X_1+Y-(2m+3)Z\geq 0\}
$$
is compact and invariant.
\end{proposition}
\begin{proof}
The compactness is derived from $Y-(2m+3)Z\geq 0$ and \eqref{eqn:_new positive conservation law2}.

Since
\begin{equation}
\label{eqn:_derivative of (2m+3)Z-Y}
\begin{split}
&\langle\nabla(Y-(2m+3)Z),V\rangle_{\mathcal{B}_{RF}\cap \{Y-(2m+3)Z=0\}}\\
&=(Y-(2m+3)Z)\left(H\left(G+\frac{1}{n}(1-H^2)\right)-X_1-(4m+6)Z\right)\\
&\quad +(4m+6)Z(X_2-X_1+Y-(2m+3)Z)\\
&=(4m+6)Z(X_2-X_1+Y-(2m+3)Z)\quad \text{since $Y-(2m+3)Z=0$}\\
&\geq 0,\\
\end{split}
\end{equation}
the vector field $V$ points inward on the boundary $\mathcal{B}_{RF}\cap\{Y-(2m+3)Z=0\}$. As for $\mathcal{B}_{RF}\cap\{X_2-X_1+Y-(2m+3)Z=0\}$, from \eqref{eqn:_new positive conservation law2} we have 
\begin{equation}
\label{eqn:_RF restriction}
\begin{split}
&\frac{12m}{n}(Y-(2m+3)Z)^2+6Y^2+4m(4m+8)YZ-12mZ^2=\frac{n-1}{n}\\
&\Leftrightarrow
18(2m+1)Y^2+8m(8m^2+16m+3)YZ+24m(2m^2+4m+3)Z^2=4m+2.\\
\end{split}
\end{equation}
Then we have 
\begin{equation}
\label{eqn:_derivative of X_2-X_1+Y-(2m+3)Z}
\begin{split}
&\langle\nabla(X_2-X_1+Y-(2m+3)Z),V\rangle_{\mathcal{B}_{RF}\cap \{X_2-X_1+Y-(2m+3)Z)=0\}}\\
&=(X_2-X_1+Y-(2m+3)Z))\left(H\left(G+\frac{1}{n}(1-H^2)-1\right)+\frac{4m}{n}Y+\frac{8m^2+24m+18}{n}Z\right)\\
&\quad +\frac{1}{n}(Y-(2m+3)Z)((4m+2)H-(12m+6)Y-(8m^2+16m+12)Z)\\
&=\frac{1}{n}(Y-(2m+3)Z)(4m+2-(12m+6)Y-(8m^2+16m+12)Z)\\
&\quad \text{since $H=1$ and $X_2-X_1+Y-(2m+3)Z=0$}.
\end{split}
\end{equation}
With $Y,Z\geq 0$, showing $4m+2-(12m+6)Y-(8m^2+16m+12)Z\geq 0$ is equivalent to showing $(4m+2)^2\geq ((12m+6)Y+(8m^2+16m+12)Z)^2$. Note that $(4m+2)^2$ is simply the LHS of \eqref{eqn:_RF restriction} multiplied by $4m+2$. Hence one can obtain the non-negativity by verifying
\begin{equation}
\begin{split}
&(4m+2)(18(2m+1)Y^2+8m(8m^2+16m+3)YZ+24m(2m^2+4m+3)Z^2)\\
&\geq ((12m+6)Y+(8m^2+16m+12)Z)^2\\
&\Leftrightarrow\\
&16 n(m-1)((4m^2+8m+3) YZ + (2m^2+ 4m+3)Z^2)\geq 0.
\end{split}
\end{equation}
Note that the equality is reached by $m=1$.
Hence \eqref{eqn:_derivative of X_2-X_1+Y-(2m+3)Z} is non-negative and it identically vanishes if $m=1$. Hence $\mathcal{B}_{RF}$ is invariant.
\end{proof}

\begin{remark}
\label{rem: more on spin(7)}
With \eqref{eqn:_new positive conservation law2} one can easily show that 
\begin{equation}
\begin{split}
\mathcal{B}_{\spin(7)}&=\mathcal{E}\cap \{Y-2Z-X_1=0\}\cap \{3Z-X_2=0\}\\
&=\mathcal{E}\cap\{H=1\}\cap\{X_2-X_1+Y-5Z=0\}.
\end{split}
\end{equation}
Therefore, the fact that \eqref{eqn:_derivative of X_2-X_1+Y-(2m+3)Z} identically vanishes for $m=1$ recovers the invariant set $\mathcal{B}_{\spin(7)}$ as in \eqref{eqn:_characterization of gamma_infty}.
\end{remark}

\begin{remark}
\label{rem: another Ricci flat invariant set}
By \eqref{eqn:_derivative of (2m+3)Z-Y} and \eqref{eqn:_derivative of X_2-X_1+Y-(2m+3)Z}, one can also show that the set 
$$
\tilde{\mathcal{B}}_{RF}:=\mathcal{E}\cap \{H=1\}\cap \{Y-(2m+3)Z\leq 0\}\cap \{X_2-X_1+Y-(2m+3)Z\leq 0\}.
$$
is also invariant. Furthermore,
\end{remark}

\begin{proposition}
\label{prop: no rotation}
For $m\geq 2$, the integral curve $\gamma_\infty$ is in $\mathcal{B}_{RF}$ initially. For $m=1$, the integral curve $\gamma_\infty$ stays on the boundary of $\mathcal{B}_{RF}$. For $m\geq 1$, the integral curve $\zeta_\infty$ is in $\tilde{\mathcal{B}}_{RF}$ initially.
\end{proposition}
\begin{proof}
The statement is clear by \eqref{eqn:_linearized RF near p_0} and \eqref{eqn:_linearized RF near p_1}. Note that for $m=1$, the function $X_2-X_1+Y-5Z$ is identically zero on $\gamma_\infty$.
\end{proof}
To obtain more information on how $\gamma_{s_1}$ and $\zeta_{s_2}$ wind around $\Phi$ as $s_1,s_2\to \infty$, we consider the following ``cylindrical'' coordinate.
$$
r\sin(\theta)=X_1-X_2, \quad r\cos(\theta)=\sqrt{\frac{2m+2}{2m+3}}(Y-(2m+3)Z).
$$
The system \eqref{eqn:_new Positive Einstein system} is transformed into
\begin{equation}
\label{eqn:_rotational system}
\begin{bmatrix}
H\\
r\\
\theta\\
Y
\end{bmatrix}'=
\begin{bmatrix}
(H^2-1)\left(\frac{1}{n}+\frac{12m}{n} r^2\sin^2(\theta)\right)\\
rH\left(\frac{1}{n}+\frac{12m}{n}r^2\sin^2(\theta)\right)-Hr\sin^2(\theta)-\frac{H}{n}r\cos^2(\theta)+\left(\frac{m+2}{m+1}+\frac{3}{n}\right)r^2\sin(\theta)\cos^2(\theta)\\
2\sqrt{\frac{2m+2}{2m+3}}Y+\frac{1}{m+1}r\cos(\theta)-\frac{n-1}{n}H\sin(\theta)\cos(\theta)-\left(\frac{m+2}{m+1}+\frac{3}{n}\right)r\cos(\theta)\sin^2(\theta)\\
Y\left(\frac{12m}{n} H r^2\sin^2(\theta)-\frac{4m}{n}r\sin(\theta)\right)
\end{bmatrix}
\end{equation}
with conservation law \eqref{eqn:_new positive conservation law2} rewritten as 
$$
C_{\Lambda\geq 0}\colon \frac{12m}{n}r^2\sin^2(\theta)+6Y^2+4m(4m+8)Y\tilde{Z}-12m\tilde{Z}^2
=1-\frac{1}{n},\quad \tilde{Z}=\left(\frac{Y}{2m+3}-\frac{r\cos(\theta)}{\sqrt{(2m+2)(2m+3)}}\right).$$
The set $\mathcal{E}$ is then defined in the $(H,r,\theta,Y)$-coordinate accordingly. The variable $r$ tells the distance from a point to $\Phi$ and $\theta$ records the winding angle around $\Phi$. 

With the new conservation law, setting $r=0$ implies $Y=(2m+3)z_0$. Restricting \eqref{eqn:_rotational system} to the invariant set $\mathcal{E}\cap\{r=0\}$ gives the following subsystem.
\begin{equation}
\label{eqn:_rotational subsystem}
\begin{bmatrix}
H\\
\theta
\end{bmatrix}'=
\begin{bmatrix}
(H^2-1)\frac{1}{n}\\
2\sqrt{(2m+2)(2m+3)}z_0-\frac{n-1}{n}H\sin(\theta)\cos(\theta)
\end{bmatrix}, \quad r=0,\quad Y=(2m+3)z_0
\end{equation}
The subsystem above is essentially the integral curve $\Phi$. Straightforward computations show that \eqref{eqn:_rotational system} has the following four sequences of critical points in $\mathcal{E}\cap\{r=0\}$:
\begin{equation}
\label{eqn:_critical point in rotational}
\begin{split}
\{A_i^\pm:=(\pm 1,0,a_i^\pm,(2m+3)z_0)\}_{i\in\mathbb{Z}},\quad a_i^\pm=\pm\arctan\left(-\frac{\delta_1}{2\sqrt{(2m+3)(2m+2)}z_0}\right)+i\pi,\\
\{B_i^\pm:=(\pm 1,0,b_i^\pm,(2m+3)z_0)\}_{i\in\mathbb{Z}},\quad b_i^\pm=\pm\arctan\left(-\frac{\delta_2}{2\sqrt{(2m+3)(2m+2)}z_0}\right)+i\pi.
\end{split}
\end{equation}
Furthermore, for each $i\in\mathbb{Z}$ we have 
\begin{equation}
\label{eqn:_basic range}
\begin{split}
&a_i^+\in \left(i\pi,\frac{\pi}{4}+i\pi\right),\quad b_i^+\in\left(\frac{\pi}{4}+i\pi,\frac{\pi}{2}+i\pi\right),\\
&b_i^-\in \left(-\frac{\pi}{2}+i\pi,-\frac{\pi}{4}+i\pi\right),\quad a_i^-\in \left(-\frac{\pi}{4}+i\pi,i\pi\right).
\end{split}
\end{equation}

\begin{remark}
\label{rem: an old point into two new points}
Computations show that $A_i^+$'s and $B_i^+$'s are transformed respectively from the two stable eigenvectors $u_1$ and $u_2$ of the linearization at $p_2^+$. Recall from Section \ref{sec: Linearization at critical points} that both $u_1$ and $u_2$ are tangent to $\mathcal{E}\cap\{H=1\}$ and the corresponding eigenvalues $\delta_1$ and $\delta_2$ are real numbers and we have $\delta_2<\delta_1<0$. Each linearized solution to the Ricci-flat subsystem around $p_2^+$ must have $e^{\delta_2\eta}\ll e^{\delta_1\eta}$ as $\eta\to \infty$. Hence integral curves $\gamma_\infty$ and $\zeta_\infty$ converge to $p_2^+$ along $u_1$. Hence it is not surprising that $A_i^+$'s are sinks and $B_i^+$'s are saddles in the subsystem of \eqref{eqn:_rotational system} restricted to $\mathcal{E}\cap\{H=1\}$. Furthermore, for the subsystem \eqref{eqn:_rotational subsystem}, critical points $A_i^+$'s are saddles, and $B_i^+$'s are sources.
\end{remark}

Thanks to the invariant set $\mathcal{B}_{RF}$ in Proposition \ref{prop: Ricci flat invariant set}, whose boundary contains $p_2^+$. We are now ready to show that both $\gamma_\infty$ and $\zeta_\infty$ do not wind fully around $p_2^+$. 
\begin{proposition}
Consider the $(H,r,\theta,Y)$-coordinate. For $m\geq 2$,
the integral curve $\gamma_\infty$ joins the critical points $p_0^+$ and $A_0^+$. For $m=1$, the variable $\theta$ remains a constant along $\gamma_\infty$ and the integral curve joins $p_0^+$ and $B_0^+$. For $m\geq 1$, the integral curve $\zeta_\infty$ joins the critical points $p_1^+$ and $A_1^+$.
\end{proposition}
\begin{proof} 
In the new coordinate, the critical point $p_0^+$ is $\left(1,\frac{1}{3}\sqrt{\frac{4m+5}{2m+3}},\arctan\left(\sqrt{\frac{2m+3}{2m+2}}\right),\frac{1}{3}\right)$. As $\gamma_\infty$ joins $p_0^+$ and $p_2^+$ in the $(X_1,X_2,Y,Z)$-coordinate, it is clear that the integral curve joins $p_0^+$ and one of $A_i^+$ or $B_i^+$ in the new coordinate. From Proposition \ref{prop: Ricci flat invariant set} and Proposition \ref{prop: no rotation} we know that $r\cos(\theta)\geq 0$ and $\sqrt{\frac{2m+3}{2m+2}}r\cos(\theta)-r\sin(\theta)\geq 0$ along $\gamma_\infty$. In particular, the second inequality reaches equality so that $\theta=\arctan\left(\sqrt{\frac{2m+3}{2m+2}}\right)$ along $\gamma_\infty$ for $m=1$. Note that $0<a_0^+<\arctan\left(\sqrt{\frac{2m+3}{2m+2}}\right)\leq b_0^+$ and the last inequality reaches equality for $m=1$. The $\theta$-coordinate for $p_0^+$ is positive and $\{\theta\geq 0\}$ is clearly invariant. Hence for $m\geq 2$, the integral curve $\gamma_\infty$ converges to $A_0^+$; for $m=1$, the integral curve $\gamma_\infty$ converges to $B_0^+$ as $\theta=b_0^+$ along $\gamma_\infty$.

In the $(H,r,\theta,Y)$-coordinate, the critical point $p_1^+$ is $\left(1,\sqrt{\frac{2m+2}{2m+3}}\frac{2m+2}{n},\pi,\frac{1}{n}\right)$. It is established in \cite{chi_einstein_2020} that $\zeta_\infty$ is an integral curve that joins $p_1^+$ and $p_2^+$ in the $(X_1,X_2,Y,Z)$-coordinate. Therefore, in the $(H,r,\theta,Y)$-coordinate, the integral curve $\zeta_\infty$ joins $p_1^+$ and one of $A_i^+$'s or $B_i^+$'s. As $\zeta_\infty$ is in $\tilde{\mathcal{B}}_{RF}$ initially, by Remark \ref{rem: another Ricci flat invariant set} we know that $\cos(\theta)=\sqrt{\frac{2m+2}{2m+3}}(Y-(2m+3)Z)\leq 0$ and $\sqrt{\frac{2m+3}{2m+2}}r\cos(\theta)-r\sin(\theta)\leq 0$ along $\zeta_\infty$. Hence we know that $\frac{\pi}{2}\leq \theta\leq \arctan\left(\sqrt{\frac{2m+3}{2m+2}}\right)+\pi\leq b_1^+$ along $\zeta_\infty$. Therefore $\zeta_\infty$ converges to $A_1^+$ for $m\geq 2$. We claim that the integral curve $\zeta_\infty$ also converges to $A_1^+$ for $m=1$. Recall Remark \ref{rem: integral curve chi} and Remark \ref{rem: more on spin(7)}. The integral curve $\chi$ also converges to $p_2^+$ and along $\chi$ we have $X_2-X_1\geq 0$ and $X_2-X_1+Y-5Z=0$. Hence $\chi$ converges to $B_1^+$ in the $(H,r,\theta,Y)$-coordinate. As the linearization at $B_1^+$ has only one stable eigenvector, we know that $\zeta_\infty$ converges to $A_1^+$.
\end{proof}

We claim the following lemma.
\begin{lemma}
\label{lem: limit angle}
Let $\Pi$ be the integral curve of the subsystem \eqref{eqn:_rotational subsystem} that emanates from $A_0^+$. Let $(0,0,\theta_*,(2m+3)z_0)$ be the midpoint of $\Pi$ at which it passes through $H=0$. For $m\geq 2$, we have $\theta_*=\lim\limits_{s_1\to\infty}\theta(\gamma_{s_1}\cap \{H=0\})$.
\end{lemma}
\begin{proof}
We think of \eqref{eqn:_rotational system} on $\mathcal{E}$ as a dynamical system in the three dimensional $Hr\theta$-space with $Y$ as a function in $(H,r,\theta)$. As mentioned in Remark \ref{rem: an old point into two new points}, each $A_i^+$ is a saddle whose linearization has two stable eigenvectors and one unstable eigenvector. Furthermore, the two stable eigenvectors are parallel to $\mathcal{E}\cap \{H=1\}$, meaning that each $A_i^+$ is a sink in the Ricci-flat subsystem and $\Pi$ is the only integral curve that emanates from $A_i^+$ in $\mathcal{E}\cap \{H<1\}$. By Hartman–Grobman theorem, there is a local homeomorphism $(\phi, \mathcal{U})$ defined around $A_0^+$ through which the system \eqref{eqn:_rotational system} is topologically equivalent to the following linear dynamical system
\begin{equation}
\label{eqn:_linearized solution xyz}
\begin{bmatrix}
x\\
y\\
z
\end{bmatrix}'=\begin{bmatrix}
-1&0&0\\
0&-1&0\\
0&0&1\\
\end{bmatrix}\begin{bmatrix}
x\\
y\\
z
\end{bmatrix}.
\end{equation}
In particular, we have 
$$\phi(A_0^+)=(0,0,0),\quad \phi(\mathcal{U}\cap\{H=1\})\subset \{(x,y,0)\mid x,y\in \mathbb{R}\},\quad \phi(\mathcal{U}\cap\Pi)\subset \{(0,0,z)\mid z\in \mathbb{R}\}.$$ Let $D_\epsilon$ be an open disk on the $xy$-plane with radius $\epsilon_1$. Let $U_1$ be the cylinder $D_{\epsilon_1}\times (-\epsilon_2,\epsilon_2]$. Note that integral curves in $U_1\cap \{z>0\}$ can only escape through the face $D_{\epsilon_1}\times \{\epsilon_2\}$. Choose small enough $\epsilon_1$ and $\epsilon_2$ so that $\mathcal{U}_1:=\phi^{-1}(U_1)$ is contained in $\mathcal{U}$. Then $p:=\phi^{-1}(0,0,\epsilon_2)$ is a point on $\mathcal{U}\cap \Pi$. 

Let $\mathcal{U}_0$ be an open neighborhood around the point $(0,0,\theta_*)$ in the $Hr\theta$-space. By the continuous dependence, there exists an open set $\mathcal{U}_2\ni p$ in $\mathcal{U}$ such that any point in $\mathcal{U}_2$ lies on an integral curve that enters $\mathcal{U}_0$. It is clear that $p\in \mathcal{U}_1\cap \mathcal{U}_2$. Modify $U_1$ by shrinking $\epsilon_1$ so that $D_{\epsilon_1}\times \{\epsilon_2\}$ is contained in $\phi(\mathcal{U}_2)$ while leaving $\epsilon_2$ unchanged.  Then correspondingly with the modified $\mathcal{U}_1$, an integral curve in $\mathcal{U}_1\cap \{H<1\}$ must enter $\mathcal{U}_2$. Since $\gamma_\infty$ converges to $A_0^+$, there exists a point $q\in \gamma_\infty\cap \mathcal{U}_1$. By the continuous dependence, there exists a large enough $N$ such that $s_1>N$ implies $\gamma_{s_1}$ must enter $\mathcal{U}_1\cap \{H<1\}$ and hence $\mathcal{U}_0$.
\end{proof}

Note that proving Lemma \ref{lem: limit angle} for $m=1$ is more subtle. As $\gamma_\infty$ converges to $B_0^+$ and there is an obvious integral curve that joins $B_0^+$ and $A_0^+$, a more delicate analysis is needed to show that $\gamma_{s_1}$ with a large enough $s_1$ must enter $U_0$. On the other hand, since $\zeta_{\infty}$ converges to $A_1^+$ for $m\geq 1$, we have the following corollary to Lemma \ref{lem: limit angle}.
\begin{corollary}
\label{cor: limit angle zeta}
For $m\geq 1$, the number $\theta_*+\pi$ is the limiting winding angle of $\zeta_{s_2}$ around $\Phi$ while $H>0$ as $s_2\to \infty$.
\end{corollary}

\begin{lemma}
\label{lem: limiting angle determines the existence}
For $m\geq 2$, there exist at least two Einstein metrics on $\mathbb{HP}^{m+1}\sharp\overline{\mathbb{HP}}^{m+1}$ if $\theta_*< \pi$. For $m\geq 1$, there exist at least two Einstein metrics on $\mathbb{S}^{4m+4}$ if $\theta_*> \pi$.
\end{lemma}
\begin{proof}
Consider $\mathbb{HP}^{m+1}\sharp\overline{\mathbb{HP}}^{m+1}$ for $m\geq 2$.
Let $\tilde{s}_1=\max\left\{\frac{3}{m-1},s_\star\right\}$, where $\gamma_{s_\star}$ is the heterocline in the proof of theorem \ref{thm: first Einstein metric} that joins $p_0^\pm$. The proof of Theorem \ref{thm: first Einstein metric} shows that we have $\sharp C(\gamma_{s_1})\geq 1$ for $s_1\in \left(\tilde{s}_1,\frac{9(5m+3)(4m^2+4m+3)}{n^2(2m+3)(m-1)}\right)$. If $\theta_*< \pi$, then $\lim\limits_{s_1\to \infty} \sharp C(\gamma_{s_1})=0$ by Lemma \ref{lem: limit angle}. By Theorem \ref{thm: Bohm local constant} there exists a heterocline $\gamma_{s_{\star\star}}$ for some $s_{\star \star}\in \left(\tilde{s}_1,\infty\right)$ and $s_{\star \star}\neq s_\star$.

Consider $\mathbb{S}^{4m+4}$ for $m\geq 1$. If $\theta_*>\pi$, then $\Pi+(0,0,\pi,(2m+3)z_0)$ is an integral curve that emanates from $A_1^+$ and passes $(0,0,\theta_*+\pi,(2m+3)z_0)$ and $\theta_*+\pi>2\pi$.  By Proposition \ref{prop:turning or W-intersection}, any $\zeta_{s_2}$ with $s_2>0$ either has a turning point or a $W$-intersection point. We learn from the linearized solution \eqref{eqn:_linearized positive einstein near P1+} that along $\zeta_{s_2}$ with $s_2> 0$ the functions $Y-Z$ and $X_2-X_1$ are positive initially. Since $\left(\frac{Y}{Z}\right)'=2\frac{Y}{Z}(X_2-X_1)$ by \eqref{eqn:_derivative of Z-Y}, the function $Y-Z=Z\left(\frac{Y}{Z}-1\right)$ can only have a zero after $X_2-X_1$ changes sign. In particular, the function $X_2-X_1$ must vanish first before any $W$-intersection point occur.

Replace $(s_1,\gamma_{s_1})$ by $(s_2,\zeta_{s_2})$ in Definition \ref{def: counting} and Theorem \ref{thm: Bohm local constant}. For $\theta_*>\pi$ we have $\lim\limits_{s_2\to \infty}\sharp C(\zeta_{s_2})\geq 1$ by Corollary \ref{cor: limit angle zeta}. On the other hand, from Proposition \ref{prop: positive X1-X2} we know that $X_2-X_1> 0$ along $\zeta_{\frac{1}{2m+6}}$ and hence $\sharp C\left(\zeta_{\frac{1}{2m+6}}\right)=0$. By Theorem \ref{thm: Bohm local constant} there exists some $s_\bullet\in \left(\frac{1}{2m+6},\infty\right)$ such that $\zeta_{s_\bullet}$ is a heterocline. 

Assume $\theta_*>\pi$ so that such an $s_\bullet$ exists. We claim that the Einstein metric $\hat{g}_{\mathbb{S}^{4m+4}}$ represented by $\zeta_{s_\bullet}$ is not the standard sphere metric. If $\hat{g}_{\mathbb{S}^{4m+4}}$ were the standard sphere metric, it would have constant sectional curvature. In particular, we must have $\frac{\ddot{f_2}}{f_2}\frac{f_1}{\ddot{f_1}}=1$. From \eqref{eqn:_initial condition related to s_2} we must have $s_2=0$. Hence $\hat{g}_{\mathbb{S}^{4m+4}}$ is not the standard sphere metric. The proof is complete.
\end{proof}

From Lemma \ref{lem: limiting angle determines the existence} we learn that the number $\theta_*$ plays an important role in proving the existence of Einstein metrics on $\mathbb{HP}^{m+1}\sharp\overline{\mathbb{HP}}^{m+1}$ and $\mathbb{S}^{4m+4}$. We can apply the Runge--Kutta algorithm to estimate $\Pi$. Then one can only set the initial step near $A_0^+$, making the approximation less accurate as $m$ increases. To bypass this issue, we make use of the symmetry of \eqref{eqn:_rotational subsystem} and estimate $\theta_*$ using the 4th order Runge--Kutta algorithm with a well-defined initial step. 

Consider the $H\theta$-plane in the following. It is obvious that \eqref{eqn:_rotational subsystem} admits $\mathbb{Z}_2$-symmetry in the sign of $(H,\theta)$. The system also admits translation symmetry $(H,\theta)\rightarrow (H,\theta+i\pi)$ for any $i\in\mathbb{Z}$. Let $O_i^\pm$ be either $A_i^\pm$ or $B_i^\pm$ and correspondingly, let $o_i^\pm$ be either $a_i^\pm$ or $b_i^\pm$. Let $\Psi$ be the integral curve with the initial condition $\Psi(0)=(0,0)$. In general, the integral curve $\Psi$ must converge to some $O_i^-$ with $i\geq 1$. By symmetry, we know that $\Psi$ is defined on $\mathbb{R}$ and joins $O_{-i}^+$ and $O_i^-$. Then $\Psi+(0,\pi)$ is an integral curve that joins $O_{-i+1}^+$ and $O_{i+1}^-$ and passes through $(0,\pi)$, forming a barrier for estimating $\theta_*$. In particular, if $\Psi$ converges to $O_1^-$, then we have $o_1^-<\pi$. Then $\Psi+(0,\pi)$ passes $(0,\pi)$ and either joins $A_0^+$ and $A_2^-$ or joins $B_0^+$ and $B_2^-$. In both cases, the integral curve $\Pi$ passes through $(0,\theta_*)$ for some $\theta_*\leq \pi$. On the other hand, if $\Psi$ converges to $O_i^-$ with $i\geq 2$, then $\Psi+(0,\pi)$ passes $(0,\pi)$ and joins $O_{-i+1}^+$ and $O_{i+1}^-$. But $o_{-i+1}^+<a_0^-$, hence $\Pi$ passes through $(0,\theta_*)$ for some $\theta_*> \pi$. We present two sets of graphs on $H\theta$-plane for $\Psi$ and $\Psi+(0,\pi)$ in Figure \ref{fig: Psi curve 1} to illustrate our argument. 
\begin{figure}[h!] 
\centering
\begin{subfigure}{.3\textwidth}
  \centering 
  \includegraphics[clip,width=1\linewidth]{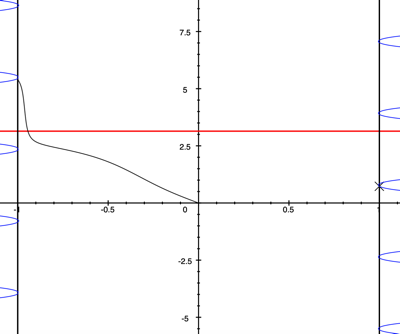}
    \caption{$m=1,\quad \Psi(\eta), \eta\geq 0$}
\end{subfigure}
\begin{subfigure}{.3\textwidth}
  \centering
  \includegraphics[clip,width=1\linewidth]{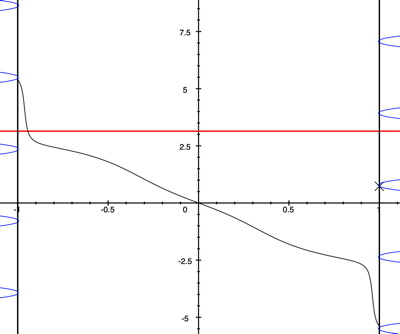}
    \caption{$m=1,\quad \Psi(\eta), \eta\in\mathbb{R}$}
\end{subfigure}
\begin{subfigure}{.3\textwidth}
  \centering
  \includegraphics[clip,width=1\linewidth]{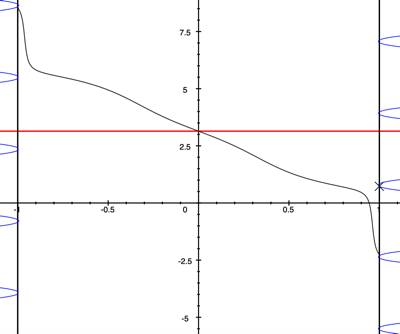}
  \caption{$m=1,\quad \Psi(\eta)+(0,\pi), \eta\in\mathbb{R}$}
\end{subfigure}\\
\begin{subfigure}{.3\textwidth}
  \centering
  \includegraphics[clip,width=1\linewidth]{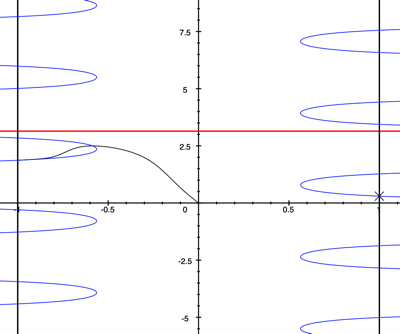}
  \caption{$m=5,\quad \Psi(\eta), \eta\geq 0$}
\end{subfigure}
\begin{subfigure}{.3\textwidth}
  \centering
  \includegraphics[clip,width=1\linewidth]{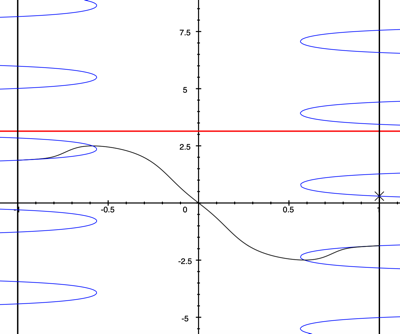}
  \caption{$m=5,\quad \Psi(\eta), \eta\in\mathbb{R}$}
\end{subfigure}
\begin{subfigure}{.3\textwidth}
  \centering
  \includegraphics[clip,width=1\linewidth]{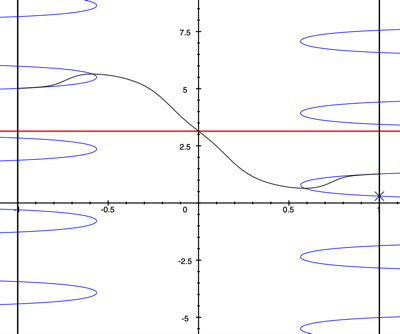}
  \caption{$m=5,\quad \Psi(\eta)+(0,\pi), \eta\in\mathbb{R}$}
\end{subfigure}
\caption{In each figure above, the red horizontal line is  $\theta=\pi$ and the blue curve is the level curve $2\sqrt{(2m+2)(2m+3)}z_0-\frac{n-1}{n}H\sin(\theta)\cos(\theta)=0$. The cross at $\{H=1\}$ is $A_0^+$. Plots above show that for $m=1$, the graph of $\Pi$ must be above $\Psi+(0,\pi)$ and hence $\theta_*>\pi$. On the other hand for $m=5$, the graph of $\Pi$ is below $\Psi+(0,\pi)$ and hence $\theta_*<\pi$.}
\label{fig: Psi curve 1}
\end{figure}
Note that if $\Psi$ converges to $A_1^-$, then $\Pi=\Psi+(0,\pi)$ and we have $\theta_*=\pi$. In such a situation, a more delicate analysis is needed to obtain $\lim\limits_{s_1\to \infty} \sharp C(\gamma_{s_1})$ and $\lim\limits_{s_2\to \infty} \sharp C(\zeta_{s_2})$. In fact, Lemma \ref{lem: limiting angle determines the existence} implies that as long as $\Psi$ does not converge to $A_1^-$ for a fixed $m$, we either have a second Einstein metric on $\mathbb{HP}^{m+1}\sharp\overline{\mathbb{HP}}^{m+1}$ or a new Einstein metric on $\mathbb{S}^{4m+4}$.

Fortunately, the 4th order Runge--Kutta algorithm shows that for $m\in [2,100]$, the integral curve $\Psi$ converges to $B_1^-$. Hence $\Pi$ must pass $(0,\theta_*)$ for some $\theta_*<\pi$. Therefore, by Lemma \ref{lem: limiting angle determines the existence}, the second Einstein metric exists on $\mathbb{HP}^{m+1}\sharp\overline{\mathbb{HP}}^{m+1}$ for $m\in[2,100]$. The function $H$ in \eqref{eqn:_rotational subsystem} can be solved explicitly. By Remark \ref{rem: fix the guage}, it is clear that $H=-\tanh\left(\frac{\eta}{n}\right).$ Hence the above discussion can be summarized into a more compact statement as in Theorem \ref{thm: second Einstein metric}. Since $\Psi$ converges to one of the $O_i^-$'s, from \eqref{eqn:_basic range} the inequality $\Omega<\frac{3\pi}{4}$ in Theorem \ref{thm: second Einstein metric} essentially means that $\Psi$ converges to $B_1^-$. As shown by the algorithm, for $m=1$ the integral curve $\Psi$ converges to $B_2^-$, meaning that $\Pi$ must pass $(0,\theta_*)$ for some $\theta_*>\pi$. We show some plots of $\Psi$ for different $m$ on the $H\theta$-plane in Figure \ref{fig: Psi curve 2}, generated by the 4th order Runge--Kutta algorithm with step size 0.01 in Grapher.
\begin{figure}[h!] 
\centering
\begin{subfigure}{.3\textwidth}
  \centering 
  \includegraphics[clip,width=1\linewidth]{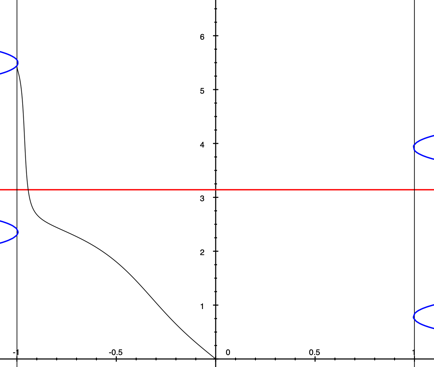}
    \caption{$m=1$}
\end{subfigure}
\begin{subfigure}{.3\textwidth}
  \centering
  \includegraphics[clip,width=1\linewidth]{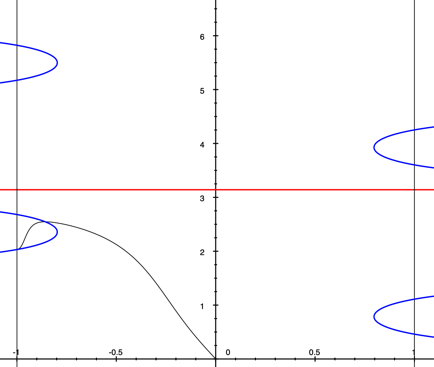}
    \caption{$m=2$}
\end{subfigure}
\begin{subfigure}{.3\textwidth}
  \centering
  \includegraphics[clip,width=1\linewidth]{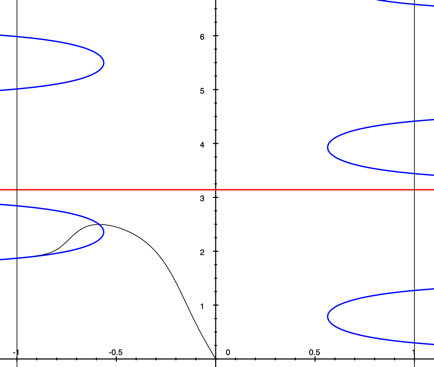}
  \caption{$m=5$}
\end{subfigure}\\
\begin{subfigure}{.3\textwidth}
  \centering
  \includegraphics[clip,width=1\linewidth]{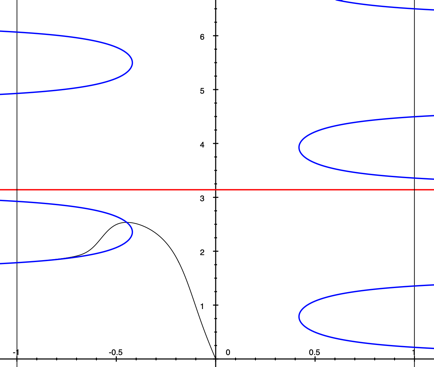}
  \caption{$m=10$}
\end{subfigure}
\begin{subfigure}{.3\textwidth}
  \centering
  \includegraphics[clip,width=1\linewidth]{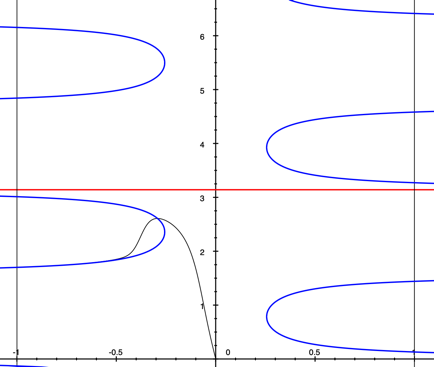}
  \caption{$m=29$}
\end{subfigure}
\begin{subfigure}{.3\textwidth}
  \centering
  \includegraphics[clip,width=1\linewidth]{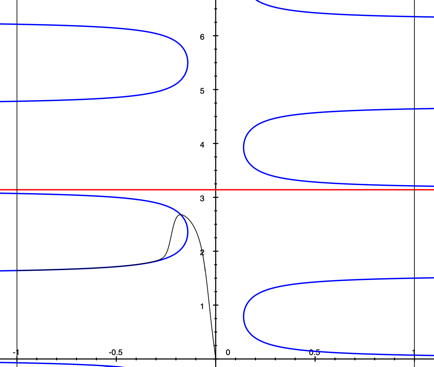}
  \caption{$m=100$}
\end{subfigure}
\caption{}
\label{fig: Psi curve 2}
\end{figure}

In the following lemma, we prove that $\Omega>\pi$ for $m=1$. Therefore, the inequality $\theta_*>\pi$ is indeed valid for $m=1$.
\begin{lemma}
\label{lem: for thm 1.3}
Let $\Psi$ be an integral curve to the following dynamical system
\begin{equation}
\label{eqn:_rotational subsystem m=1}
\begin{bmatrix}
H\\
\theta
\end{bmatrix}'=\tilde{V}(H,\theta):=
\begin{bmatrix}
(H^2-1)\frac{1}{7}\\
\frac{4\sqrt{5}}{21}-\frac{3}{7}H\sin(2\theta)
\end{bmatrix}
\end{equation}
with $\Psi(0)=(0,0)$. Then $\lim\limits_{\eta\to \infty}\theta(\Psi) >\pi$.
\end{lemma}
\begin{proof}
Let $H_1=\frac{2}{e^{\frac{\sqrt{5}\pi}{6}}+1}-1\approx -0.527$ and $H_2= \frac{2}{\pi-2-2a_1^-}\approx -0.543$.
Define the following function.
$$
\Theta(H)=\left\{\begin{array}{ll}
\frac{1}{H}+1+a_1^-& -1\leq H\leq H_2\\
\frac{\pi}{2}&  H_2<H< H_1\\
\frac{3\sqrt{5}}{5}\ln\left(\frac{1-H}{1+H}\right)& H_1\leq H\leq 0
\end{array}\right..
$$
Recall in \eqref{eqn:_critical point in rotational} we have $a_1^-=\pi-\arctan\left(\frac{2\sqrt{5}}{5}\right)$. Then we have 
$$\sin(a_1^-)=\frac{2}{3},\quad \cos(a_1^-)=-\frac{\sqrt{5}}{3},\quad\sin(2a_1^-)=-\frac{4\sqrt{5}}{9},\quad \cos(2a_1^-)=\frac{1}{9}.$$
The function $\Theta$ is non-increasing. The numbers $H_1$ and $H_2$ are chosen so that $\{\theta-\Theta(H)=0\}$ is a continuous curve that joins the origin and $A_1^-$.

We show that $\tilde{V}$ restricted to $\{\theta-\Theta(H)=0\}$ points upward. For $(H,\theta)\in (H_1,0)\times \left(0,\frac{\pi}{2}\right)$, we have 
\begin{equation}
\label{eqn:_first barrier for m=1}
\begin{split}
\left\langle \nabla\left(\theta-\frac{3\sqrt{5}}{5}\ln\left(\frac{1-H}{1+H}\right)\right),\tilde{V}\right\rangle &=\frac{4\sqrt{5}}{21}-\frac{3}{7}H\sin(2\theta)+\frac{6\sqrt{5}}{5}\frac{1}{1-H^2}\frac{H^2-1}{7}\\
&=\frac{2\sqrt{5}}{105}-\frac{3}{7}H\sin(2\theta)\\
&> 0.
\end{split}
\end{equation}
Therefore, the function $\theta-\Theta(H)$ remains positive along $\Psi$ as $H$ decreases from $0$ to $H_1$. The computation above also shows that the function $\theta-\Theta(H)$ is positive along $\Psi(\eta)$ once the integral curve leaves the origin.
If $\theta=\frac{\pi}{2}$, then $\left\langle \nabla \theta,\tilde{V} \right\rangle=\frac{4\sqrt{5}}{21}>0.$ Hence $\theta>\frac{\pi}{2}$ along $\Psi$ as $H$ decreases from $H_1$ to $H_2$. 

Finally, as $H$ decreases from $H_2$ to $-1$, the function $\Theta(H)$ increases from $\frac{\pi}{2}$ to $a_1^-$. We first claim that for $\theta\in \left[\frac{\pi}{2},a_1^-\right]$, the inequality 
\begin{equation}
\label{eqn:_m=1technical1}
D(\theta):=\sin(2\theta)- \left(\theta-\frac{\pi}{2}\right)\left(\frac{7}{4}(\theta-a_1^-)+\frac{\sin(2a_1^-)}{a_1^--\frac{\pi}{2}}\right)\geq 0
\end{equation}
is valid. It is obvious that $D(a_1^-)=D\left(\frac{\pi}{2}\right)=0$. Hence there exists some $\theta_1\in\left(\frac{\pi}{2},a_1^-\right)$ that $D'(\theta_1)=0$ by the mean value theorem. On the other hand, we have
$$
D'(\theta)=2\cos(2\theta)-\left(\frac{7}{4}(\theta-a_1^-)+\frac{\sin(2a_1^-)}{a_1^--\frac{\pi}{2}}\right)-\frac{7}{4}\left(\theta-\frac{\pi}{2}\right),\quad D''(\theta)=-4\sin(2\theta)-\frac{7}{2}.
$$
Hence $D''>0$ on $\left(\frac{\pi}{2}+\frac{1}{2}\arcsin\left(\frac{7}{8}\right),\pi-\frac{1}{2}\arcsin\left(\frac{7}{8}\right)\right)$. Since $\pi-\frac{1}{2}\arcsin\left(\frac{7}{8}\right)\approx 2.609> 2.412\approx a_1^-$, it is clear that $D'$ decreases in $\left(\frac{\pi}{2},\frac{\pi}{2}+\frac{1}{2}\arcsin\left(\frac{7}{8}\right)\right)$ and increases in $\left(\frac{\pi}{2}+\frac{1}{2}\arcsin\left(\frac{7}{8}\right),a_1^-\right)$. Since 
\begin{equation}
\begin{split}
&D'\left(\frac{\pi}{2}\right)=-2-\left(\frac{7}{4}\left(\frac{\pi}{2}-a_1^-\right)+\frac{\sin(2a_1^-)}{a_1^--\frac{\pi}{2}}\right)\approx 0.654 >0,\\
&D'\left(a_1^-\right)=2\cos(2a_1^-)-\frac{\sin(2a_1^-)}{a_1^--\frac{\pi}{2}}-\frac{7}{4}\left(a_1^- -\frac{\pi}{2}\right)\approx -0.068 <0,
\end{split}
\end{equation}
we know that $D'$ only vanishes once in $\left[\frac{\pi}{2},a_1^-\right]$.
Therefore, the function $D$ is indeed positive for $\theta\in \left(\frac{\pi}{2},a_1^-\right)$.

Then from \eqref{eqn:_m=1technical1} we have
\begin{equation}
\label{eqn:_m=1technical2}
\begin{split}
&\left.\left\langle \nabla\left(\theta-\frac{1}{H}-1-a_1^-\right),\tilde{V} \right\rangle\right|_{\theta-\frac{1}{H}-1-a_1^-= 0}\\
 &=\frac{4\sqrt{5}}{21}-\frac{3}{7}H\sin(2\theta)+\frac{1}{7H^2}(H^2-1)\\
&\geq \frac{4\sqrt{5}}{21}-\frac{3}{7}H\left(\theta-\frac{\pi}{2}\right)\left(\frac{7}{4}(\theta-a_1^-)+\frac{\sin(2a_1^-)}{a_1^- -\frac{\pi}{2}}\right)+\frac{1}{7H^2}(H^2-1)\\
&=     \frac{4\sqrt{5}}{21}-\frac{3}{7}H\left(\frac{1+H}{H}+a_1^--\frac{\pi}{2}\right)\left(\frac{7}{4}\left(\frac{1+H}{H}\right)+\frac{\sin(2a_1^-)}{a_1^- -\frac{\pi}{2}}\right)+\frac{1}{7H^2}(H^2-1)\\
&\quad \text{since $\theta-\frac{1}{H}-1-a_1^-= 0$}\\
& = \frac{4\sqrt{5}}{21}(1+H)-\left(\frac{3}{4}\left(a_1^--\frac{\pi}{2}\right)+\frac{3}{7}\frac{\sin(2a_1^-)}{a_1^- -\frac{\pi}{2}}\right)(1+H)-\frac{3}{4}\frac{(1+H)^2}{H}+\frac{1}{7H^2}(H^2-1).
\end{split}
\end{equation}
Since $\frac{4\sqrt{5}}{21}-\left(\frac{3}{4}\left(a_1^--\frac{\pi}{2}\right)+\frac{3}{7}\frac{\sin(2a_1^-)}{a_1^- -\frac{\pi}{2}}\right)\approx 0.3015 >0.3$, the computation above continues as
\begin{equation}
\begin{split}
\left.\left\langle \nabla\left(\theta-\frac{1}{H}-1-a_1^-\right),\tilde{V} \right\rangle\right|_{\theta-\frac{1}{H}-1-a_1^-= 0} &\geq \frac{3}{10}(1+H)-\frac{3}{4}\frac{(1+H)^2}{H}+\frac{1}{7H^2}(H^2-1)\\
&=\frac{1+H}{H^2}\left(-\frac{9}{20}H^2-\frac{17}{28}H-\frac{1}{7}\right).
\end{split}
\end{equation}
A straightforward computation shows that the factor $-\frac{9}{20}H^2-\frac{17}{28}H-\frac{1}{7}$ is positive on $\left[-1,-\frac{1}{2}\right]$. As $H_2<-\frac{1}{2}$, it is proved that $\Psi$ does not pass the barrier $\theta-\Theta(H)=0$ where $H\in (-1,H_2)$. Therefore, we must have $\lim\limits_{\eta\to \infty}\theta(\Psi)\geq a_1^-$.

We claim that $\Psi$ does not converge to $A_1^-$. The linearization of \eqref{eqn:_rotational subsystem m=1} at $A_1^-$ is $\begin{bmatrix}
-\frac{2}{7}&0\\
\frac{4\sqrt{5}}{21}&\frac{2}{21}
\end{bmatrix}$, whose only stable eigenvalue and eigenvector are respectively $-\frac{2}{7}$ and $\begin{bmatrix}
2\\
-\sqrt{5}\\
\end{bmatrix}$. Hence the linearized solution in $\{H> -1\}$ takes the form of $A_1^- + \begin{bmatrix}
2\\
-\sqrt{5}\\
\end{bmatrix}e^{-\frac{2}{7}\eta}$. Suppose $\Psi$ is the integral curve that tends to $A_1^-$. We must have 
$$
(\theta-\Theta(H))(\Psi(\eta))\sim (\theta-\Theta(H)) \left(A_1^- +\begin{bmatrix}
2\\
-\sqrt{5}\\
\end{bmatrix}e^{-\frac{2}{7}\eta}\right)=-e^{-\frac{2}{7}\eta}\left(\sqrt{5}-\frac{2}{1-2e^{-\frac{2}{7}\eta}}\right)<0
$$
as $\eta\to \infty$, which is a contradiction. Therefore, the integral curve $\Psi$ converges to some $O_i^-$ with $i\geq 2$. As $A_2^->B_2^->\pi$, we conclude that $\lim\limits_{\eta\to \infty}\theta(\Psi)>\pi$. 
\end{proof}
Theorem \ref{thm: 8-sphere} is proved with Lemma \ref{lem: for thm 1.3} established.

\newpage
\section{Appendix}
\label{sec: Appendix}

\subsection{Detailed computation of \eqref{eqn:_deri fof A} and \eqref{eqn:_deri fof P}}
We first give detailed computation for $\left\langle\nabla A,V\right\rangle$ in the following.
\begin{equation}
\begin{split}
&\left\langle\nabla A,V\right\rangle\\
&=Y\left(X_2H\left(G+\frac{1}{n}(1-H^2)-1\right)+R_2-\frac{1}{n}(1-H^2)\right)+YX_2\left(H\left(G+\frac{1}{n}(1-H^2)\right)-X_1\right)\\
&\quad -\frac{3}{m}Z\left(\left(X_1+\frac{2m}{3}X_2\right)H\left(G+\frac{1}{n}(1-H^2)-1\right)+R_1+\frac{2m}{3}R_2-\left(1+\frac{2m}{3}\right)\frac{1}{n}(1-H^2)\right)\\
&\quad- \frac{3}{m}Z\left(X_1+\frac{2m}{3}X_2\right)\left(H\left(G+\frac{1}{n}(1-H^2)\right)+X_1-2X_2\right)\\
&=Y\left(2X_2H\left(G+\frac{1}{n}(1-H^2)\right)-X_2H+R_2-\frac{1}{n}(1-H^2)\right)+YX_2\left(-X_1\right)\\
&\quad -\frac{3}{m}Z\left(2\left(X_1+\frac{2m}{3}X_2\right)H\left(G+\frac{1}{n}(1-H^2)\right)-\left(X_1+\frac{2m}{3}X_2\right)H+R_1+\frac{2m}{3}R_2-\left(1+\frac{2m}{3}\right)\frac{1}{n}(1-H^2)\right)\\
&\quad- \frac{3}{m}Z\left(X_1+\frac{2m}{3}X_2\right)(X_1-2X_2)\\
&=2YX_2H\left(G+\frac{1}{n}(1-H^2)\right)+Y\left(-X_2H+R_2-\frac{1}{n}(1-H^2)\right)+YX_2\left(-X_1\right)\\
&\quad -2\frac{3}{m}Z\left(X_1+\frac{2m}{3}X_2\right)H\left(G+\frac{1}{n}(1-H^2)\right)\\
&\quad -\frac{3}{m}Z\left(-\left(X_1+\frac{2m}{3}X_2\right)H+R_1+\frac{2m}{3}R_2-\left(1+\frac{2m}{3}\right)\frac{1}{n}(1-H^2)\right)\\
&\quad- \frac{3}{m}Z\left(X_1+\frac{2m}{3}X_2\right)(X_1-2X_2)\\
&=2AH\left(G+\frac{1}{n}(1-H^2)\right)+Y\left(-X_2H+R_2-\frac{1}{n}(1-H^2)\right)+YX_2\left(-X_1\right)\\
&\quad -\frac{3}{m}Z\left(-\left(X_1+\frac{2m}{3}X_2\right)H+R_1+\frac{2m}{3}R_2-\left(1+\frac{2m}{3}\right)\frac{1}{n}(1-H^2)\right)\\
&\quad- \frac{3}{m}Z\left(X_1+\frac{2m}{3}X_2\right)(X_1-2X_2)\\
&=2AH\left(G+\frac{1}{n}(1-H^2)\right)+Y\left(-X_2(2X_1+(4m+2)X_2)-X_2(2X_1-2X_2)+R_2-\frac{1}{n}(1-H^2)\right)\\
&\quad -\frac{3}{m}Z\left(-\left(X_1+\frac{2m}{3}X_2\right)(2X_1+(4m+2)X_2)+R_1+\frac{2m}{3}R_2-\left(1+\frac{2m}{3}\right)\frac{1}{n}(1-H^2)\right)\\
&=2AH\left(G+\frac{1}{n}(1-H^2)\right)-YX_2(2X_1+(4m+2)X_2)+Y\left(-X_2(2X_1-2X_2)+R_2-\frac{1}{n}(1-H^2)\right)\\
&\quad +\frac{3}{m}Z\left(X_1+\frac{2m}{3}X_2\right)(2X_1+(4m+2)X_2)-\frac{3}{m}Z\left(R_1+\frac{2m}{3}R_2-\left(1+\frac{2m}{3}\right)\frac{1}{n}(1-H^2)\right)\\
&=A\left(2H\left(G+\frac{1}{n}(1-H^2)\right)-(2X_1+(4m+2)X_2)\right)\\
&\quad +Y\left(-X_2(2X_1-2X_2)+R_2-\frac{1}{n}(1-H^2)\right) -\frac{3}{m}Z\left(R_1+\frac{2m}{3}R_2-\left(1+\frac{2m}{3}\right)\frac{1}{n}(1-H^2)\right).
\end{split}
\end{equation}
The last two terms in the computation above continue as the following.
\begin{equation}
\begin{split}
&Y\left(-X_2(2X_1-2X_2)+R_2-\frac{1}{n}(1-H^2)\right) -\frac{3}{m}Z\left(R_1+\frac{2m}{3}R_2-\left(1+\frac{2m}{3}\right)\frac{1}{n}(1-H^2)\right)\\
&=Y\left(R_2-\frac{1}{n}(1-H^2)\right)-2YX_2(X_1-X_2)-\frac{3}{m}Z\left(R_1+\frac{2m}{3}R_2-\left(1+\frac{2m}{3}\right)\frac{1}{n}(1-H^2)\right)\\
&=\frac{1}{X_2(X_1+\frac{2m}{3}X_2)}\left(YX_2(X_1+\frac{2m}{3}X_2)\left(R_2-\frac{1}{n}(1-H^2)\right)\right)\\
&\quad -\frac{1}{X_2(X_1+\frac{2m}{3}X_2)}\left(2YX_2^2(X_1-X_2)\left(X_1+\frac{2m}{3}X_2\right)\right)\\
&\quad +\frac{1}{X_2(X_1+\frac{2m}{3}X_2)}\left(-\frac{3}{m}ZX_2(X_1+\frac{2m}{3}X_2)\left(R_1+\frac{2m}{3}R_2-\left(1+\frac{2m}{3}\right)\frac{1}{n}(1-H^2)\right)\right)\\
&=\frac{1}{X_2(X_1+\frac{2m}{3}X_2)}\left(YX_2(X_1+\frac{2m}{3}X_2)\left(R_2-\frac{1}{n}(1-H^2)\right)\right)\\
&\quad -\frac{1}{X_2(X_1+\frac{2m}{3}X_2)}\left(2YX_2^2(X_1-X_2)\left(X_1+\frac{2m}{3}X_2\right)\right)\\
&\quad +\frac{1}{X_2(X_1+\frac{2m}{3}X_2)}\left((A-YX_2)X_2\left(R_1+\frac{2m}{3}R_2-\left(1+\frac{2m}{3}\right)\frac{1}{n}(1-H^2)\right)\right)\\
&=\frac{1}{X_2(X_1+\frac{2m}{3}X_2)}\left(YX_2(X_1+\frac{2m}{3}X_2)\left(R_2-\frac{1}{n}(1-H^2)\right)\right)\\
&\quad -\frac{1}{X_2(X_1+\frac{2m}{3}X_2)}\left(2YX_2^2(X_1-X_2)\left(X_1+\frac{2m}{3}X_2\right)\right)\\
&\quad +\frac{1}{X_2(X_1+\frac{2m}{3}X_2)}\left(-YX_2^2\left(R_1+\frac{2m}{3}R_2-\left(1+\frac{2m}{3}\right)\frac{1}{n}(1-H^2)\right)\right)\\
&\quad +\frac{1}{X_2(X_1+\frac{2m}{3}X_2)}\left(AX_2\left(R_1+\frac{2m}{3}R_2-\left(1+\frac{2m}{3}\right)\frac{1}{n}(1-H^2)\right)\right)\\
&=\frac{Y}{X_1+\frac{2m}{3}X_2}(X_1+\frac{2m}{3}X_2)\left(R_2-\frac{1}{n}(1-H^2)\right)\\
&\quad -\frac{Y}{X_1+\frac{2m}{3}X_2}\left(2X_2(X_1-X_2)\left(X_1+\frac{2m}{3}X_2\right)\right)\\
&\quad +\frac{Y}{X_1+\frac{2m}{3}X_2}\left(-X_2\left(R_1+\frac{2m}{3}R_2-\left(1+\frac{2m}{3}\right)\frac{1}{n}(1-H^2)\right)\right)\\
&\quad +\frac{A}{X_1+\frac{2m}{3}X_2}\left(R_1+\frac{2m}{3}R_2-\left(1+\frac{2m}{3}\right)\frac{1}{n}(1-H^2)\right)\\
&=\frac{Y}{X_1+\frac{2m}{3}X_2}\left((X_1+\frac{2m}{3}X_2)\left(R_2-\frac{1}{n}(1-H^2)\right)-X_2\left(R_1+\frac{2m}{3}R_2-\left(1+\frac{2m}{3}\right)\frac{1}{n}(1-H^2)\right)\right)\\
&\quad -\frac{Y}{X_1+\frac{2m}{3}X_2}\left(2X_2(X_1-X_2)\left(X_1+\frac{2m}{3}X_2\right)\right)\\
&\quad +\frac{A}{X_1+\frac{2m}{3}X_2}\left(R_1+\frac{2m}{3}R_2-\left(1+\frac{2m}{3}\right)\frac{1}{n}(1-H^2)\right)\\
&=\frac{Y}{X_1+\frac{2m}{3}X_2}\left(X_1\left(R_2-\frac{1}{n}(1-H^2)\right)-X_2\left(R_1-\frac{1}{n}(1-H^2)\right)-2X_2(X_1-X_2)\left(X_1+\frac{2m}{3}X_2\right)\right)\\
&\quad +\frac{A}{X_1+\frac{2m}{3}X_2}\left(R_1+\frac{2m}{3}R_2-\left(1+\frac{2m}{3}\right)\frac{1}{n}(1-H^2)\right)\\
&=\frac{Y}{X_1+\frac{2m}{3}X_2}P +\frac{A}{X_1+\frac{2m}{3}X_2}\left(R_1+\frac{2m}{3}R_2-\left(1+\frac{2m}{3}\right)\frac{1}{n}(1-H^2)\right).
\end{split}
\end{equation}
Hence computation \eqref{eqn:_deri fof A} indeed holds.

We give detailed computation for $\langle \nabla P,V\rangle$ in the following. First of all, we have 
\begin{equation}
\begin{split}
&\left\langle\nabla \left(R_1-\frac{1}{n}(1-H^2)\right),V\right\rangle\\
&= \left\langle\nabla \left(2Y^2+4mZ^2-\frac{1}{n}(1-H^2)\right),V\right\rangle\\
&= 4Y^2\left(H\left(G+\frac{1}{n} (1-H^2)\right)-X_1\right)+8mZ^2\left(H\left(G+\frac{1}{n} (1-H^2)\right)+X_1-2X_2\right)\\
&\quad + \frac{2H}{n}(H^2-1)\left(G+\frac{1}{n} (1-H^2)\right)\\
&= 2\left(R_1-\frac{1}{n}(1-H^2)\right)H\left(G+\frac{1}{n} (1-H^2)\right)-4Y^2X_1+8mZ^2(X_1-2X_2)\\
&=2\left(R_1-\frac{1}{n}(1-H^2)\right)H\left(G+\frac{1}{n} (1-H^2)\right)-(4Y^2-8mZ^2)(X_1-X_2)-2R_1X_2,
\end{split}
\end{equation}
and 
\begin{equation}
\begin{split}
&\left\langle\nabla \left(R_2-\frac{1}{n}(1-H^2)\right),V\right\rangle\\
&= \left\langle\nabla \left((4m+8)YZ-6Z^2-\frac{1}{n}(1-H^2)\right),V\right\rangle\\
&= 2(4m+8)YZ\left(H\left(G+\frac{1}{n} (1-H^2)\right)-X_2\right)-12Z^2\left(H\left(G+\frac{1}{n} (1-H^2)\right)+X_1-2X_2\right)\\
&\quad + \frac{2H}{n}(H^2-1)\left(G+\frac{1}{n} (1-H^2)\right)\\
&= 2\left(R_2-\frac{1}{n}(1-H^2)\right)H\left(G+\frac{1}{n} (1-H^2)\right)-2(4m+8)YZX_2-12Z^2(X_1-2X_2)\\
&=2\left(R_2-\frac{1}{n}(1-H^2)\right)H\left(G+\frac{1}{n} (1-H^2)\right)-((4m+8)YZ+6Z^2)(X_1-X_2)-R_2(3X_2-X_1).
\end{split}
\end{equation}
Then it follows that
\begin{equation}
\label{eqn:_X1R2-X2R1}
\begin{split}
&X_1\left\langle\nabla \left(R_2-\frac{1}{n}(1-H^2)\right),V\right\rangle-X_2\left\langle\nabla \left(R_1-\frac{1}{n}(1-H^2)\right),V\right\rangle\\
&=2X_1\left(R_2-\frac{1}{n}(1-H^2)\right)H\left(G+\frac{1}{n} (1-H^2)\right)-((4m+8)YZ+6Z^2)X_1(X_1-X_2)-R_2X_1(3X_2-X_1)\\
&\quad -2X_2\left(R_1-\frac{1}{n}(1-H^2)\right)H\left(G+\frac{1}{n} (1-H^2)\right)+(4Y^2-8mZ^2)X_2(X_1-X_2)+2R_1X_2^2\\
&=2X_1\left(R_2-\frac{1}{n}(1-H^2)\right)H\left(G+\frac{1}{n} (1-H^2)\right)-((4m+8)YZ+6Z^2)X_1(X_1-X_2)-R_2X_1(3X_2-X_1)\\
&\quad -2X_2\left(R_1-\frac{1}{n}(1-H^2)\right)H\left(G+\frac{1}{n} (1-H^2)\right)+(4Y^2-8mZ^2)X_2(X_1-X_2)+2R_1X_2^2\\
&\quad -4X_2\left(X_1+\frac{2m}{3}X_2\right)(X_1-X_2)H\left(G+\frac{1}{n} (1-H^2)\right)\\
&\quad +4X_2\left(X_1+\frac{2m}{3}X_2\right)(X_1-X_2)H\left(G+\frac{1}{n} (1-H^2)\right)\\
&=2PH\left(G+\frac{1}{n} (1-H^2)\right)\\
&\quad -((4m+8)YZ+6Z^2)X_1(X_1-X_2)-R_2X_1(3X_2-X_1)+(4Y^2-8mZ^2)X_2(X_1-X_2)+2R_1X_2^2\\
&\quad +4X_2\left(X_1+\frac{2m}{3}X_2\right)(X_1-X_2)H\left(G+\frac{1}{n} (1-H^2)\right)\\
&=2PH\left(G+\frac{1}{n} (1-H^2)\right)\\
&\quad +(X_1-X_2)\left((4Y^2-8mZ^2)X_2-((4m+8)YZ+6Z^2)X_1\right) -R_2X_1(3X_2-X_1)+2R_1X_2^2\\
&\quad +4X_2\left(X_1+\frac{2m}{3}X_2\right)(X_1-X_2)H\left(G+\frac{1}{n} (1-H^2)\right).
\end{split}
\end{equation}
Then we have
\begin{equation}
\begin{split}
&\langle\nabla P,V\rangle\\
&=\left\langle\nabla \left(X_1\left(R_2-\frac{1}{n}(1-H^2)\right)-X_2\left(R_1-\frac{1}{n}(1-H^2)\right)-2X_2\left(X_1+\frac{2m}{3}X_2\right)(X_1-X_2)\right),V\right\rangle\\
&=\left(X_1H\left(G+\frac{1}{n} (1-H^2)-1\right)+R_1-\frac{1}{n}(1-H^2)\right)\left(R_2-\frac{1}{n}(1-H^2)\right)+X_1\left\langle\nabla \left(R_2-\frac{1}{n}(1-H^2)\right),V\right\rangle\\
&\quad -\left(X_2H\left(G+\frac{1}{n} (1-H^2)-1\right)+R_2-\frac{1}{n} (1-H^2)\right)\left(R_1-\frac{1}{n}(1-H^2)\right)-X_2\left\langle\nabla \left(R_1-\frac{1}{n}(1-H^2)\right),V\right\rangle\\
&\quad -2\left(X_1+\frac{2m}{3}X_2\right)(X_1-X_2)\left(X_2H\left(G+\frac{1}{n} (1-H^2)-1\right)+R_2-\frac{1}{n} (1-H^2)\right)\\
&\quad -2X_2(X_1-X_2)\left(\left(X_1+\frac{2m}{3}X_2\right)H\left(G+\frac{1}{n} (1-H^2)-1\right)+R_1+\frac{2m}{3}R_2-\left(1+\frac{2m}{3}\right)\frac{1}{n}(1-H^2)\right)\\
&\quad -2X_2\left(X_1+\frac{2m}{3}X_2\right)\left((X_1-X_2)H\left(G+\frac{1}{n} (1-H^2)-1\right)+R_1-R_2\right)\\
&=X_1\left(R_2-\frac{1}{n}(1-H^2)\right)H\left(G+\frac{1}{n} (1-H^2)-1\right)+X_1\left\langle\nabla \left(R_2-\frac{1}{n}(1-H^2)\right),V\right\rangle\\
&\quad -X_2\left(R_1-\frac{1}{n}(1-H^2)\right)H\left(G+\frac{1}{n} (1-H^2)-1\right)-X_2\left\langle\nabla \left(R_1-\frac{1}{n}(1-H^2)\right),V\right\rangle\\
&\quad -2\left(X_1+\frac{2m}{3}X_2\right)(X_1-X_2)X_2H\left(G+\frac{1}{n} (1-H^2)-1\right)-2\left(X_1+\frac{2m}{3}X_2\right)(X_1-X_2)\left(R_2-\frac{1}{n} (1-H^2)\right)\\
&\quad -2X_2(X_1-X_2)\left(\left(X_1+\frac{2m}{3}X_2\right)H\left(G+\frac{1}{n} (1-H^2)-1\right)+R_1+\frac{2m}{3}R_2-\left(1+\frac{2m}{3}\right)\frac{1}{n}(1-H^2)\right)\\
&\quad -2X_2\left(X_1+\frac{2m}{3}X_2\right)\left((X_1-X_2)H\left(G+\frac{1}{n} (1-H^2)-1\right)+R_1-R_2\right)\\
&=PH\left(G+\frac{1}{n} (1-H^2)-1\right)\\
&\quad +X_1\left\langle\nabla \left(R_2-\frac{1}{n}(1-H^2)\right),V\right\rangle -X_2\left\langle\nabla \left(R_1-\frac{1}{n}(1-H^2)\right),V\right\rangle\\
&\quad -2\left(X_1+\frac{2m}{3}X_2\right)(X_1-X_2)\left(R_2-\frac{1}{n} (1-H^2)\right)\\
&\quad -4X_2(X_1-X_2)\left(X_1+\frac{2m}{3}X_2\right)H\left(G+\frac{1}{n} (1-H^2)-1\right)\\
&\quad -2X_2(X_1-X_2)\left(R_1+\frac{2m}{3}R_2-\left(1+\frac{2m}{3}\right)\frac{1}{n}(1-H^2)\right) -2X_2\left(X_1+\frac{2m}{3}X_2\right)(R_1-R_2).
\end{split}
\end{equation}
The first term in the last line of the computation above is obtained by gathering the first, the third, and the fifth term from the second last line. From \eqref{eqn:_X1R2-X2R1}, the computation above continues as the following.
\begin{equation}
\label{eqn:_so much pain}
\begin{split}
&=PH\left(G+\frac{1}{n} (1-H^2)-1\right)\\
&\quad +2PH\left(G+\frac{1}{n} (1-H^2)\right)\\
&\quad +(X_1-X_2)\left((4Y^2-8mZ^2)X_2-((4m+8)YZ+6Z^2)X_1\right) -R_2X_1(3X_2-X_1)+2R_1X_2^2\\
&\quad +4X_2\left(X_1+\frac{2m}{3}X_2\right)(X_1-X_2)H\left(G+\frac{1}{n} (1-H^2)\right)\\
&\quad -2\left(X_1+\frac{2m}{3}X_2\right)(X_1-X_2)\left(R_2-\frac{1}{n} (1-H^2)\right)\\
&\quad -4X_2(X_1-X_2)\left(X_1+\frac{2m}{3}X_2\right)H\left(G+\frac{1}{n} (1-H^2)-1\right)\\
&\quad -2X_2(X_1-X_2)\left(R_1+\frac{2m}{3}R_2-\left(1+\frac{2m}{3}\right)\frac{1}{n}(1-H^2)\right)-2X_2\left(X_1+\frac{2m}{3}X_2\right)(R_1-R_2)\\
&=PH\left(3G+\frac{3}{n} (1-H^2)-1\right)\\
&\quad +(X_1-X_2)\left((4Y^2-8mZ^2)X_2-((4m+8)YZ+6Z^2)X_1\right) -R_2X_1(3X_2-X_1)+2R_1X_2^2\\
&\quad -2(X_1-X_2)\left(X_1+\frac{2m}{3}X_2\right)\left(R_2-\frac{1}{n} (1-H^2)\right)\\
&\quad +4X_2(X_1-X_2)\left(X_1+\frac{2m}{3}X_2\right)H\\
&\quad -2X_2(X_1-X_2)\left(R_1+\frac{2m}{3}R_2-\left(1+\frac{2m}{3}\right)\frac{1}{n}(1-H^2)\right)-2X_2\left(X_1+\frac{2m}{3}X_2\right)(R_1-R_2)\\
&=PH\left(3G+\frac{3}{n} (1-H^2)-1\right)\\
&\quad +(X_1-X_2)\left((4Y^2-8mZ^2)X_2-((4m+8)YZ+6Z^2)X_1\right)\\
&\quad -2(X_1-X_2)\left(X_1+\frac{2m}{3}X_2\right)\left(R_2-\frac{1}{n} (1-H^2)\right)\\
&\quad +4X_2(X_1-X_2)\left(X_1+\frac{2m}{3}X_2\right)H\\
&\quad -2X_2(X_1-X_2)\left(R_1+\frac{2m}{3}R_2-\left(1+\frac{2m}{3}\right)\frac{1}{n}(1-H^2)\right)\\
&\quad +R_2X_1(X_1-X_2)+2R_1X_2(X_2-X_1)+\frac{4m}{3}X_2(X_2R_2-X_2R_1)\\
&=PH\left(3G+\frac{3}{n} (1-H^2)-1\right)\\
&\quad +(X_1-X_2)\left((4Y^2-8mZ^2)X_2-((4m+8)YZ+6Z^2)X_1\right)\\
&\quad -2(X_1-X_2)\left(X_1+\frac{2m}{3}X_2\right)\left(R_2-\frac{1}{n} (1-H^2)\right)\\
&\quad +4X_2(X_1-X_2)\left(X_1+\frac{2m}{3}X_2\right)H\\
&\quad -2X_2(X_1-X_2)\left(R_1+\frac{2m}{3}R_2-\left(1+\frac{2m}{3}\right)\frac{1}{n}(1-H^2)\right)\\
&\quad +R_2X_1(X_1-X_2)+2R_1X_2(X_2-X_1)+\frac{4m}{3}X_2R_2(X_2-X_1)+\frac{4m}{3}X_2(X_1R_2-X_2R_1).
\end{split}
\end{equation}
Write the last term as 
\begin{equation}
\begin{split}
&\frac{4m}{3}X_2(X_1R_2-X_2R_1)\\
&=\frac{4m}{3}X_2\left(X_1\left(R_2-\frac{1}{n}(1-H^2)\right)-X_2\left(R_1-\frac{1}{n}(1-H^2)\right)-2X_2\left(X_1+\frac{2m}{3}X_2\right)(X_1-X_2)\right)\\
&\quad +\frac{4m}{3}X_2\left((X_1-X_2)\frac{1}{n}(1-H^2)+2X_2\left(X_1+\frac{2m}{3}X_2\right)(X_1-X_2)\right)\\
&=\frac{4m}{3}PX_2 +\frac{4m}{3}(X_1-X_2)X_2\left(\frac{1}{n}(1-H^2)+2X_2\left(X_1+\frac{2m}{3}X_2\right)\right).
\end{split}
\end{equation}
Then \eqref{eqn:_so much pain} continues as
\begin{equation}
\begin{split}
&=PH\left(3G+\frac{3}{n} (1-H^2)-1\right)\\
&\quad +(X_1-X_2)\left((4Y^2-8mZ^2)X_2-((4m+8)YZ+6Z^2)X_1\right)\\
&\quad -2(X_1-X_2)\left(X_1+\frac{2m}{3}X_2\right)\left(R_2-\frac{1}{n} (1-H^2)\right)\\
&\quad +4X_2(X_1-X_2)\left(X_1+\frac{2m}{3}X_2\right)H\\
&\quad -2X_2(X_1-X_2)\left(R_1+\frac{2m}{3}R_2-\left(1+\frac{2m}{3}\right)\frac{1}{n}(1-H^2)\right)\\
&\quad +R_2X_1(X_1-X_2)+2R_1X_2(X_2-X_1)+\frac{4m}{3}X_2R_2(X_2-X_1)\\
&\quad +\frac{4m}{3}PX_2 +\frac{4m}{3}(X_1-X_2)X_2\left(\frac{1}{n}(1-H^2)+2X_2\left(X_1+\frac{2m}{3}X_2\right)\right)\\
&=P\left(H\left(3G+\frac{3}{n} (1-H^2)-1\right)+\frac{4m}{3}X_2\right)+(X_1-X_2)Q.
\end{split}
\end{equation}

\subsection{Coefficients for $\tilde{B}_{m,\kappa}$ and $\tilde{Q}_{m,\kappa}$}
We list coefficients for $\tilde{B}_{m,\kappa}$ and $\tilde{Q}_{m,\kappa}$ in the following. 
\begin{itemize}
\item
$\tilde{B}_{m,\kappa}(x)=b_2x^2+b_1x+b_0$
\\
\begin{equation}
\label{eqn:_coefficients for tildeB}
\begin{split}
b_2&=-\frac{(96 \kappa^3 m^3+264 \kappa^2 m^2+216 \kappa m+54)}{(1-\kappa)(2\kappa^2m(8m-5)+6\kappa(4m-1)+9)}<0,\\
b_1&=\frac{4(m+2)(4 \kappa m+3) (2 \kappa^2 m+4 \kappa m+3)}{(1-\kappa)(2\kappa^2m(8m-5)+6\kappa(4m-1)+9)}-\frac{2 (4m+3)^2 (2 m+3) (m-1)}{m(2m+1)(8m+3)},\\
b_0&=-\frac{3 (4+(4m-2) \kappa) (4\kappa m+3) \kappa}{(1-\kappa)(2\kappa^2m(8m-5)+6\kappa(4m-1)+9)}<0.
\end{split}
\end{equation}
\item
$\tilde{Q}_{m,\kappa}(x)=q_2x^2+q_1x+q_0$
\\
\begin{equation}
\label{eqn:_coefficients for tildeQ}
\begin{split}
q_2&=-\frac{4(2\kappa m+3)(32\kappa^3m^3+\kappa^2m^2(96-68\kappa)+\kappa m(90-84\kappa)+27(1-\kappa))}{3(1-\kappa)(2\kappa^2m(8m-5)+6\kappa(4m-1)+9)}<0,\\
q_1&=\frac{16 \kappa(m+2)(16 \kappa^3 m^3-18 \kappa^3 m^2+32  \kappa^2 m^2-24  \kappa^2 m+27 \kappa m-9 \kappa+9)}{3(1- \kappa)(2\kappa^2m(8m-5)+6\kappa(4m-1)+9)}>0,\\
q_0&=-\frac{4\kappa^2(32\kappa^2m^3-20\kappa^2m^2-6\kappa^2m+48\kappa m^2-6\kappa m-9\kappa+18m+9)}{3(1-\kappa)(2\kappa^2m(8m-5)+6\kappa(4m-1)+9)}<0.
\end{split}
\end{equation}
\end{itemize}

\begin{proposition}
\label{prop: some property of B}
For any $(m,\kappa)\in [1,\infty)\times (0,1)$, the function $\tilde{B}_{m,\kappa}$ has a real root $\sigma(m,\kappa)$ in the interval $\left(0,\frac{m\kappa}{3+2m\kappa}\right)$. For $m=1$, polynomials $\tilde{B}_{1,\kappa}$ and $\tilde{Q}_{1,\kappa}$ share a common real root $\sigma(1,\kappa)=\frac{\kappa}{3+2\kappa}$. 
\end{proposition}
\begin{proof}
From \eqref{eqn:_coefficients for tildeB}, computations show that
\begin{equation}
\begin{split}
\tilde{B}_{m,\kappa}(0)&=b_0<0,\\
\tilde{B}_{m,\kappa}\left(\frac{m\kappa}{3+2m\kappa}\right)&=\frac{768(m-1)\kappa B_*}{(8  m + 3)  (2  \kappa  m + 3)  (2  m + 1)  (1 - \kappa)  (2\kappa^2m(8m-5)+6\kappa(4m-1)+9)}
\end{split}
\end{equation}
where
\begin{equation}
\begin{split}
B_*&=\left(m^4 + \frac{15}{8}  m^3 + \frac{5}{16} m^2 - \frac{45}{32} m - \frac{45}{64}\right) m\kappa^3\\
&\quad +\left(\frac{13}{4} m^4 +  \frac{163}{32} m^3 + \frac{255}{64} m^2 + \frac{9}{16}  m - \frac{27}{64}\right)  \kappa^2 \\
&\quad +\left( \frac{21}{8} m^3 + \frac{147}{64} m^2 + \frac{243}{128} m + \frac{27}{32}\right)  \kappa \\
&\quad + \frac{21}{32}m^2 - \frac{27}{128}\\
&>0
\end{split}
\end{equation}
for any $(m,\kappa)\in [1,\infty)\times (0,1)$.
Hence such a $\sigma(m,\kappa)$ exists. Furthermore, we have 
\begin{equation}
\begin{split}
\tilde{B}_{1,\kappa}(x)&=\frac{2(3+4\kappa)((2\kappa+1)x-\kappa-2)((2\kappa+3)x-\kappa)}{(2\kappa^2+6\kappa+3)(\kappa-1)},\\
\tilde{Q}_{1,\kappa}(x)&=-\frac{4((12\kappa^3-4\kappa^2-21\kappa-9)x+2 \kappa^3+11 \kappa^2+9 \kappa)((2\kappa+3)x-\kappa)}{3(2\kappa^2+6\kappa+3)(\kappa-1)}.
\end{split}
\end{equation}
The proof is complete.
\end{proof}

\subsection{Non-positivity of $r(\tilde{Q}_{m,\kappa},\tilde{B}_{m,\kappa})$}
\begin{proposition}
\label{prop: resultant is non-positive}
The resultant $r(\tilde{Q}_{m,\kappa},\tilde{B}_{m,\kappa})$ in \eqref{eqn:_resultant QB} is non-positive for any $(m,\kappa)\in [1,\infty)\times (0,1)$ and vanishes if and only if $m=1$.
\end{proposition}
\begin{proof}
Recall that 
\scriptsize
\begin{equation}
\label{eqn:_resultant QB}
\begin{split}
r(\tilde{Q}_{m,\kappa},\tilde{B}_{m,\kappa})&=-\frac{64\kappa^2(m-1)(2\kappa m+3)}{(8m+3)^2(2m+1)^2m^2(1-\kappa)(2\kappa^2m(8m-5)+6\kappa(4m-1)+9)^2}\tilde{r},\\
\tilde{r}&=262144 \kappa ^4 m^{10}+(516096 \kappa ^4+679936 \kappa ^3) m^9+(373760 \kappa ^4+1233920 \kappa ^3+675840 \kappa ^2) m^8\\
&\quad +(-275904 \kappa ^4+1151040 \kappa ^3+1143744 \kappa ^2+308160 \kappa ) m^7\\
&\quad +(-926496 \kappa ^4+248832 \kappa ^3+1432512 \kappa ^2+507456 \kappa +54432) m^6\\
&\quad +(-800496 \kappa ^4-1256256 \kappa ^3+1472688 \kappa ^2+857520 \kappa +92016) m^5\\
&\quad +(-281880 \kappa ^4-1525392 \kappa ^3+266328 \kappa ^2+1353024 \kappa +199584) m^4\\
&\quad +(-33048 \kappa ^4-644436 \kappa ^3-672624 \kappa ^2+957420 \kappa +375192) m^3\\
&\quad +(-90396 \kappa ^3-433026 \kappa ^2 +186624 \kappa +333882) m^2\\
&\quad +(-74358 \kappa ^2-59049 \kappa +133407) m+19683(1-\kappa).
\end{split}
\end{equation}
\normalsize
It is clear that the coefficient for $\tilde{r}$ is non-positive on $[1,\infty)\times (0,1)$ and vanishes if and only if $m=1$.

Consider the polynomial $\tilde{r}$. Since coefficients for $m^i$ is obviously positive for $k\in(0,1)$ if $i\geq 5$, we must have 
\scriptsize
\begin{equation}
\begin{split}
\tilde{r}&> 262144 \kappa ^4 m^4+(516096 \kappa ^4+679936 \kappa ^3) m^4+(373760 \kappa ^4+1233920 \kappa ^3+675840 \kappa ^2) m^4\\
&\quad +(-275904 \kappa ^4+1151040 \kappa ^3+1143744 \kappa ^2+308160 \kappa ) m^4\\
&\quad +(-926496 \kappa ^4+248832 \kappa ^3+1432512 \kappa ^2+507456 \kappa +54432) m^4\\
&\quad +(-800496 \kappa ^4-1256256 \kappa ^3+1472688 \kappa ^2+857520 \kappa +92016) m^4\\
&\quad +(-281880 \kappa ^4-1525392 \kappa ^3+266328 \kappa ^2+1353024 \kappa +199584) m^4\\
&\quad +(-33048 \kappa ^4-644436 \kappa ^3-672624 \kappa ^2+957420 \kappa +375192) m^3\\
&\quad +(-90396 \kappa ^3-433026 \kappa ^2 +186624 \kappa +333882) m^2\\
&\quad +(-74358 \kappa ^2-59049 \kappa +133407) m+19683(1-\kappa)\\
&=(-1132776\kappa ^4+532080\kappa ^3+4991112\kappa ^2+3026160\kappa +346032)m^4\\
&\quad +(-33048\kappa ^4-644436\kappa ^3-672624\kappa ^2+957420\kappa +375192)m^3\\
&\quad +(-90396\kappa ^3-433026\kappa ^2+186624\kappa +333882)m^2\\
&\quad +(-74358\kappa ^2-59049\kappa +133407)m+19683(1-\kappa)\\
&\geq (-1132776\kappa ^4+532080\kappa ^3+4991112\kappa ^2+3026160\kappa +346032)m^3\\
&\quad +(-33048\kappa ^4-644436\kappa ^3-672624\kappa ^2+957420\kappa +375192)m^3\\
&\quad +(-90396\kappa ^3-433026\kappa ^2+186624\kappa +333882)m^2\\
&\quad +(-74358\kappa ^2-59049\kappa +133407)m+19683(1-\kappa)\\
&=(-1165824\kappa^4-112356\kappa^3+4318488\kappa^2+3983580\kappa+721224)m^3\\
&\quad +(-90396\kappa^3-433026\kappa^2+186624\kappa+333882)m^2\\
&\quad +(-74358\kappa^2-59049\kappa+133407)m+19683(1-\kappa)\\
&\geq (-1165824\kappa^4-112356\kappa^3+4318488\kappa^2+3983580\kappa+721224)m^2\\
&\quad +(-90396\kappa^3-433026\kappa^2+186624\kappa+333882)m^2\\
&\quad +(-74358\kappa^2-59049\kappa+133407)m+19683(1-\kappa)\\
&=(-1165824\kappa^4-202752\kappa^3+3885462\kappa^2+4170204\kappa+1055106)m^2\\
&\quad +(-74358\kappa^2-59049\kappa+133407)m+19683(1-\kappa)\\
&> 0.
\end{split}
\end{equation}
\normalsize
Since $\tilde{r}$ is positive on $[1,\infty)\times (0,1)$, the proof is complete.
\end{proof}

\subsection{Visual Summaries}
We summarize Theorem \ref{thm: first Einstein metric}-\ref{thm: 8-sphere} with the following four sets of figures generated by Grapher, where integral curves presented are generated by the 4th order Runge--Kutta algorithm with step size 0.01 and the initial step is set in a neighborhood around $p_0^+$ or $p_1^+$. All figures are in the $X_1X_2Z$-space and the variable $Y$ is eliminated by \eqref{eqn:_new positive conservation law 1}.
\begin{figure}[h!] 
\centering
\begin{subfigure}{.32\textwidth}
  \centering 
  \includegraphics[clip,width=1\linewidth]{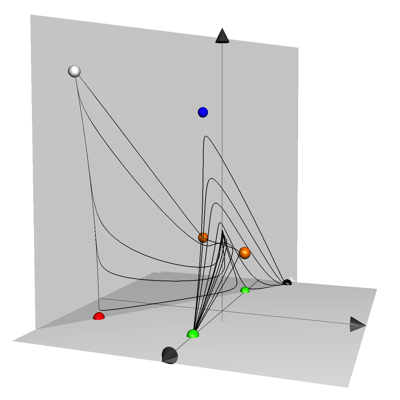}
    \caption{$\gamma_{s_1}$ for $m=1$}
\end{subfigure}
\begin{subfigure}{.32\textwidth}
  \centering
  \includegraphics[clip,width=1\linewidth]{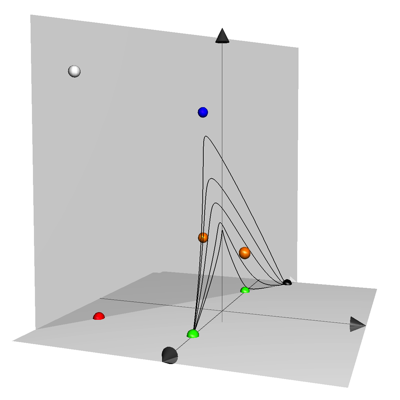}
    \caption{$s_1\in \left(0,s_\star\right)$}
\end{subfigure}
\begin{subfigure}{.32\textwidth}
  \centering
  \includegraphics[clip,width=1\linewidth]{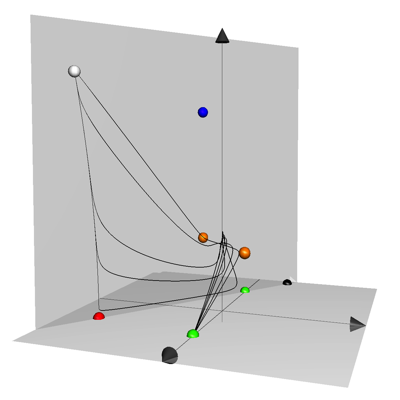}
    \caption{$s_2\in \left(s_\star,\infty\right)$}
\end{subfigure}
\caption{The set $\mathcal{S}$ is degenerate as shown in Proposition \ref{prop: S1 for m=1}. Hence only one integral curve $\gamma_{s_\star}$ is known to join $p_0^\pm$ (green) and it represents the Bohm's metric. Figures above indicate that $\gamma_{s_1}$ converges to $q_1^-$ (black) for $s_1\in (0,s_\star)$ and converges to $q_3^-$ (white) for $s_1\in (s_{\star},\infty)$.}
\label{fig: theorem 1.1-1.2-1}
\end{figure}
\begin{figure}[h!] 
\centering
\begin{subfigure}{.24\textwidth}
  \centering 
  \includegraphics[clip,width=1\linewidth]{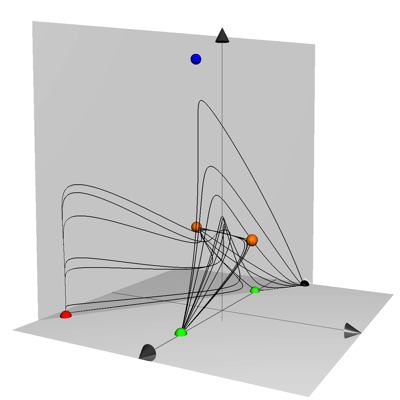}
    \caption{$\gamma_{s_1}$ for $m=2$}
\end{subfigure}
\begin{subfigure}{.24\textwidth}
  \centering
  \includegraphics[clip,width=1\linewidth]{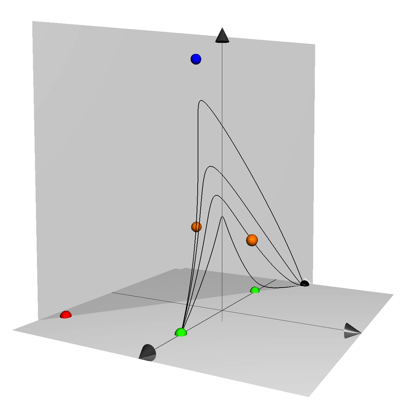}
    \caption{$s_1\in \left(0,s_\star\right)$}
\end{subfigure}
\begin{subfigure}{.24\textwidth}
  \centering
  \includegraphics[clip,width=1\linewidth]{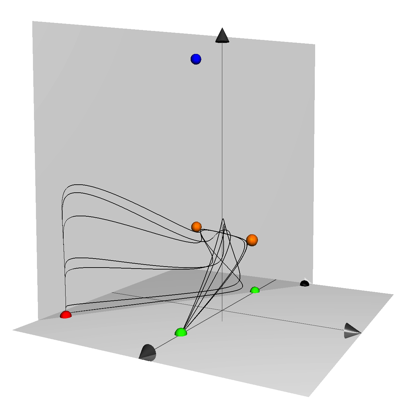}
    \caption{$s_2\in \left(s_\star,s_{\star\star}\right)$}
\end{subfigure}
\begin{subfigure}{.24\textwidth}
  \centering
  \includegraphics[clip,width=1\linewidth]{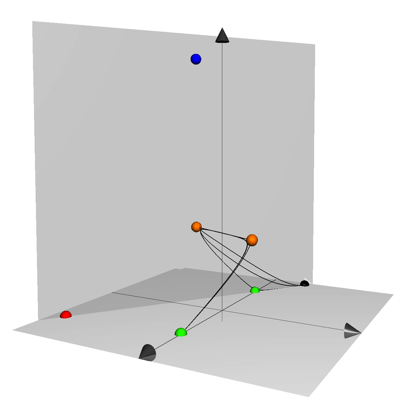}
    \caption{$s_2\in (s_{\star\star},\infty)$}
\end{subfigure}
\caption{Theorem \ref{thm: first Einstein metric} and Theorem \ref{thm: second Einstein metric} claim that for $m\geq 2$ there are at least two integral curves $\gamma_{s_\star}$ and $\gamma_{s_{\star\star}}$ that join $p_0^\pm$ (green). Figures above indicate that $\gamma_{s_1}$ converges to $q_1^-$ (black) for $s_1\in (0,s_\star)\cup (s_{\star\star},\infty)$ and converges to $q_2^-$ (red) for $s_1\in (s_{\star},s_{\star\star})$.}
\label{fig: theorem 1.1-1.2-2}
\end{figure}
\begin{figure}[h!] 
\centering
\begin{subfigure}{.24\textwidth}
  \centering 
  \includegraphics[clip,width=1\linewidth]{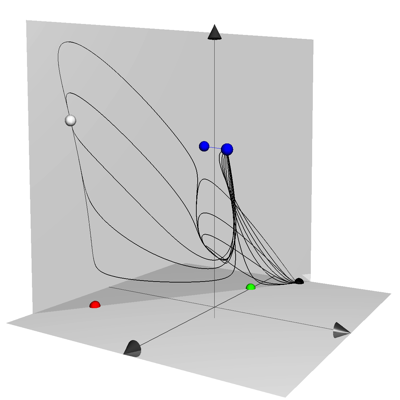}
    \caption{$\zeta_{s_2}$ for $m=1$}
\end{subfigure}
\begin{subfigure}{.24\textwidth}
  \centering
  \includegraphics[clip,width=1\linewidth]{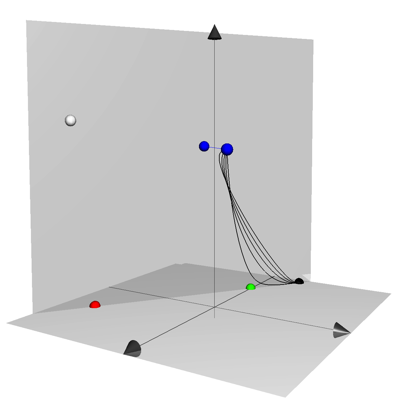}
    \caption{$s_2\in \left(0,\frac{1}{2m+6}\right)$}
\end{subfigure}
\begin{subfigure}{.24\textwidth}
  \centering
  \includegraphics[clip,width=1\linewidth]{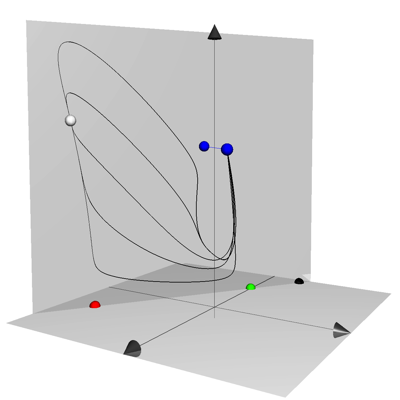}
    \caption{$s_2\in \left(\frac{1}{2m+6},s_\bullet\right)$}
\end{subfigure}
\begin{subfigure}{.24\textwidth}
  \centering
  \includegraphics[clip,width=1\linewidth]{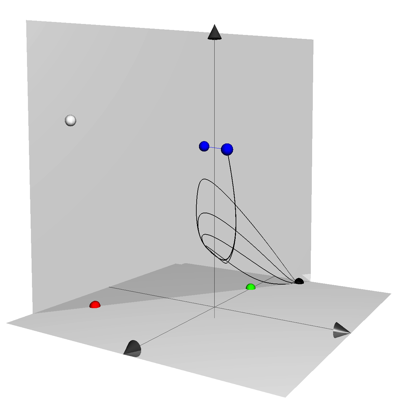}
    \caption{$s_2\in (s_\bullet,\infty)$}
\end{subfigure}
\caption{Theorem \ref{thm: 8-sphere} is realized in the plots above. For $m=1$, the graph of $\zeta_0$ is the straight line that joins $p_0^\pm$ (blue). As $s_2$ increases from $0$, the integral curve $\zeta_{s_2}$ converges to $q_1^-$ (black), until $\zeta_{\frac{1}{2m+6}}$ converges to $p_0^-$ (green). For $s_2>\frac{1}{2m+6}$, the integral curve $\zeta_{s_2}$ converges to $q_3^-$ (white), until $s_2=s_\bullet$ once again joins $p_0^\pm$. For the $s_2>s_\bullet$, the integral curve $\zeta_{s_2}$ converges again to $q_1^-$.}
\label{fig: theorem 1.3-1}
\end{figure}
\begin{figure}[h!] 
\centering
\begin{subfigure}{.32\textwidth}
  \centering 
  \includegraphics[clip,width=1\linewidth]{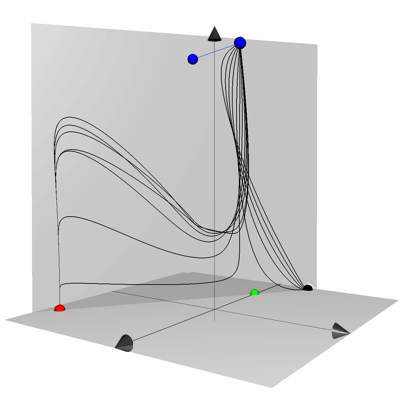}
    \caption{$\zeta_{s_2}$ for $m=1$}
\end{subfigure}
\begin{subfigure}{.32\textwidth}
  \centering
  \includegraphics[clip,width=1\linewidth]{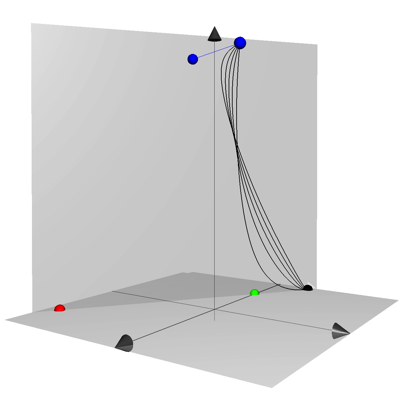}
    \caption{$s_2\in\left(0,\frac{1}{2m+6}\right)$}
\end{subfigure}
\begin{subfigure}{.32\textwidth}
  \centering
  \includegraphics[clip,width=1\linewidth]{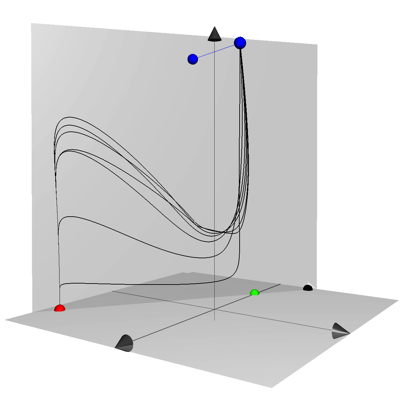}
    \caption{$s_2\in\left(\frac{1}{2m+6},\infty\right)$}
\end{subfigure}
\caption{For $m\geq 2$, the behavior of $\zeta_{s_2}$ is relatively simpler. For $s_2\in \left(0,\frac{1}{2m+6}\right)$, the integral curve $\zeta_{s_2}$ converges to $q_1^-$ (black). For $s_2>\frac{1}{2m+6}$, the integral curve $\zeta_{s_2}$ converges to $q_2^-$ (red).}
\label{fig: theorem 1.3-2}
\end{figure}
\newpage
\bibliography{PEMPO}
\bibliographystyle{alpha}
\end{document}